\documentclass[11pt]{article}
\usepackage[english]{babel}
\usepackage[utf8]{inputenc}
\usepackage{johd}
\usepackage{amsmath}
\usepackage{mathtools}
\usepackage{verbatim}
\usepackage{appendix}
\usepackage{bm}
\allowdisplaybreaks[4]

\numberwithin{equation}{section}

\usepackage{amsmath,amsthm}       
\usepackage{mathrsfs}      
\usepackage{amssymb}       
\usepackage{amsfonts}      

\usepackage{cancel}
\usepackage{indentfirst}

\usepackage{graphicx}
\usepackage{subcaption}
\usepackage{wrapfig}  
\usepackage{multicol}
\usepackage{tikz}     
\usetikzlibrary{arrows.meta}
\usepackage[framemethod=TikZ]{mdframed}
\usepackage{color}
\usepackage{authblk}
\usepackage{textcomp}
\usepackage{xcolor}

\newtheorem{thm}{Theorem}[section]
\newtheorem{lemma}[thm]{Lemma}
\newtheorem{definition}[thm]{Definition}
\newtheorem{assumption}[thm]{Assumption}
\newtheorem{remark}[thm]{Remark}
\newtheorem{prop}[thm]{Proposition}
\newtheorem{example}[thm]{Example}
\newtheorem{corollary}[thm]{Corollary}

\def\cA{{\cal A}}
\def\cB{{\cal B}}

\def\cD{{\cal D}}

\def\cF{{\cal F}}
\def\cG{{\cal G}}
\def\cH{{\cal H}}

\def\cJ{{\cal J}}
\def\cK{{\cal K}}
\def\cL{{\cal L}}
\def\cM{{\cal M}}
\def\cN{{\cal N}}
\def\cO{{\cal O}}
\def\cP{{\cal P}}

\def\cW{{\cal W}}
\def\cX{{\cal X}}

\def\cZ{{\cal Z}}

\def\dbA{\mathbb{A}}

\def\dbE{\mathbb{E}}
\def\dbF{\mathbb{F}}
\def\dbG{\mathbb{G}}

\def\dbL{\mathbb{L}}

\def\dbP{\mathbb{P}}
\def\dbR{\mathbb{R}}
\def\dbS{\mathbb{S}}
\def\dbT{\mathbb{T}}

\def\dbX{\mathbb{X}}

\def\dbZ{\mathbb{Z}}

\newcommand{\ol}{\overline}
\newcommand{\ul}{\underline}

\newcommand{\ba}{\begin{array}}
\newcommand{\ea}{\end{array}}
\newcommand{\be}{\begin{equation}}
\newcommand{\ee}{\end{equation}}
\newcommand{\bea}{\begin{eqnarray}}
\newcommand{\eea}{\end{eqnarray}}
\newcommand{\beaa}{\begin{eqnarray*}}
\newcommand{\eeaa}{\end{eqnarray*}}

\def\a{\alpha}
\def\b{\beta}

\def\d{\delta}
\def\e{\varepsilon}

\def\l{\lambda}
\def\m{\mu}

\def\si{\sigma}

\def\f{\varphi}
\def\th{\theta}

\def\o{\omega}

\def\D{\Delta}

\def\O{\Omega}
\def\q{\quad}

\newcommand{\bmu}{\mbox{$\raisebox{-0.59ex}
  {$l$}\hspace{-0.18em}\mu\hspace{-0.88em}\raisebox{-0.98ex}{\scalebox{2}
  {$\color{white}.$}}\hspace{-0.416em}\raisebox{+0.88ex}
  {$\color{white}.$}\hspace{0.46em}$}{}}

\begin{document}

%\begin{linenumbers}

\title{\bf On Information Controls}

\author{
Zihao Gu\thanks{\noindent USC Mathematics Department, Los Angeles, 90089; email: zihaogu@usc.edu.}
 ~ and ~ Jianfeng Zhang \thanks{ \noindent USC Mathematics Department, Los Angeles, 90089;
email: jianfenz@usc.edu. This author is supported in part by NSF grant  \#DMS-2510403. The author would like to thank Jaksa Cvitanic for very inspiring discussions on persuasion games in the early stage of the project, and Nizar Touzi for lots of insightful discussions.}}

\date{\today}
\maketitle

\begin{abstract}
In this paper we study an optimization problem in which the control is  {\it information}, more precisely,  the control is a $\sigma$-algebra or a filtration. In a dynamic setting, we establish the dynamic programming principle and the law invariance of the value function. The latter requires a condition slightly stronger than the (H)-hypothesis for the admissible filtration,  and enables us to define the value function on $\mathcal P_2(\mathcal P_2(\mathbb R^d))$, the space of laws of random probability measures. By using a new It\^o's formula for smooth functions on $\mathcal P_2(\mathcal P_2(\mathbb R^d))$, we characterize the value function of the information control problem by an Hamilton–Jacobi–Bellman equation on this space.
\end{abstract}


%\tableofcontents

%\vfill\eject
\section{Introduction}
For a probability space $(\O, \cF, \dbP)$,  the $\si$-algebra $\cF$ is typically interpreted as the information, and the probability measure $\dbP$ as the law of the state variable $X$. In the majority of the stochastic control literature, the player controls the law of $X$, denoted as $\dbP^\a$, and the objective is to choose a control $\a$ to maximize certain utility $J(\dbP^\a, \a)$. Note that $\a$ is required to be measurable to the information $\cG$ the player observes, which is often a sub-$\si$-algebra of $\cF$. Then, by using  $\a\in \cG$ to denote the $\cG$-measurability of $\a$, the value of the standard control problem is:
\bea
\label{cJ0}
\cJ_0(\cG) := \sup_{\a\in \cG} J(\dbP^\a, \a).
\eea
At above the player's information $\cG$ is fixed. In many practical situations, however, the player may acquire information through various means, such as learning,  conducting experiments, purchasing it,  or even (illegally)  spying. That is, he may control the $\si$-algebra $\cG$. Note that $\cJ_0(\cG)$ is increasing in $\cG$: $\cJ_0(\cG_1) \le \cJ_0(\cG_2)$ whenever $\cG_1\subset \cG_2$, and thus the optimization problem $\sup_\cG \cJ_0(\cG)$ is trivial. However, typically the player will incur a certain cost $C(\cG)$ to acquire the information $\cG$, then his objective becomes 
\bea
\label{cJ}
\cJ(\cG) := \cJ_0(\cG) - C(\cG).
\eea
In general, $\cJ$ is no longer monotone in $\cG$; in other words, it is not true that more information is always better. 
Then the following information control problem, or its dynamic version with filtration $\dbG$ as a control, makes sense:
\bea
\label{supcJ}
\sup_\cG \cJ(\cG) \q\mbox{or}\q  \sup_\dbG \cJ(\dbG),
\eea
where the supremum  is over all admissible $\si$-algebras $\cG$ or admissible filtrations $\dbG$. The information has been a central consideration in stochastic controls, especially in games. There are numerous publications related  to information, and we shall provide a brief literature review on works related particularly to information controls in Subsection \ref{sect-literature} below. However, there seems to lack a systematic study on problems controlling the $\si$-algebras or filtrations directly. As a long term project we aim to investigate this important issue systematically,  and this paper serves as a first step, by focusing on the control problem \eqref{supcJ} only. 

Before turning to our results, we would like to provide further motivations for problems in the form \eqref{cJ}-\eqref{supcJ}, where in particular $\cJ$ is not monotone. Besides the possible cost for acquiring information, typically it is also not free to analyze or compute the information. For example, Saparito-Zhang \cite{SZ-delay} studied the following information delay problem. Let $\dbF = \{\cF_t\}_{0\le t\le T}$ denote the information the player observes. Note that in practice it takes some time, say a constant $\d>0$ for simplicity,  for the player to analyze the observed information and compute the corresponding strategy, then at time $t$ the player can only use the information up to time $t-\d$, and thus the admissible control $\a_t$ needs to be $\cF_{t-\d}$-measurable. That is, one can only use $\dbG^\d = \{\cG^\d_t\}$ with $\cG^\d_t := \cF_{t-\d}$. One may reduce $\d$ and hence increase $\dbG^\d$,  say, by hiring more capable experts and/or improving the computational power, which will of course increase the cost as well.

Another related situation is that using the information may involve costs or risks, as illustrated by the illicit insider trading. Let $\cG$ denote the inside information the insider uses in his decision making, then $C(\cG)$ stands for the legal risk he faces. In this case the insider may decide how much inside information he wants to use in order to maximize the net utility.  We shall present a static example which is essentially a linear quadratic problem and can be solved explicitly in Section \ref{sect-insider1}, and a dynamic example which leads to an HJB equation in Section \ref{sect-insider2}.  

The more interesting cases arise from nonzero-sum games, with Braess’s Paradox  as a typical example, see Braess-Nagurney-Wakolbinger \cite{Braess_2005} and Roughgarden \cite{Roughgarden_2016} for details and Appendix \ref{appendixA} for a brief introduction. It is well-known that the game values are not monotone in terms of information. For example, consider a two player nonzero-sum game where Player $1$ has some private information. If Player 1 shares some of his private information to Player 2, it will change the equilibrium, which may possibly lead to higher utility for Player 1. Then Player $1$ needs to decide how much information he wants to share with Player 2. This type of information sharing is a common practice, even among rivals.\footnote{One example could be US and Soviet Union during the cold war. Quite often, without proper communication, there would be misjudgments on the other side's intention and an equilibrium could lead to very serious conflicts or even wars. Such information sharing was crucial for avoiding these disastrarous consequences.} We shall provide a motivating example in Appendix \ref{appendixA}.
 Another highly related problem is the Bayesian persuasion game, which has received very strong attention in the literature in recent years. In a simple setting, this is a Stackelberg game where the leader has private information and designs a signal to reveal it. We shall provide a brief introduction as well as some technical discussions in Appendix \ref{appendixA}.  It will be very interesting to provide a unified approach for persuasion games and our information control problems,  which we leave for future research.

We now turn to the contents of this paper. In Section  \ref{sect-insider1} we study an insider trading problem with legal risk. Besides the standard legal risk due to aggressive abnormal trading, we also introduce a new type of legal risk due to the pattern of insider trading. The latter is directly related to our information control. In a simple setting, we solve the problem explicitly and show that the solution relies on $\cL_{\cL_{X|\cG}}$, the law of conditional law, where $X$ is the fundamental value of the asset and $\cG\subset \si(X)$ is the information the insider would use. This motivates us to consider $\cP(\cP(\dbR^d))$, the space of laws of random probability measures, as our underlying space.  We shall study the basic properties of this space in Section \ref{sect-space}. We note that, while serving as a motivating example, the main result of Section \ref{sect-insider1} is interesting in its own right and provides new insights for the insider trading problem.    

Our main focus is on dynamic problems, where the control becomes a sub-filtration $\dbG = \{\cG_t\}_{t\in [0, T]}$.  In a Brownian motion setting with the largest filtration $\dbF = \cF_0 \vee \dbF^B$, for technical reasons we assume the admissible controls $\dbG\subset \dbF$ satisfy a constraint, called the (H*)-hypothesis.
That is, once the player chooses to use information $\cG_t$ at time $t$, at later time he can only explore the new information $\si(B_s-B_t, s\ge t)$ induced by the increments of $B$, but cannot use the remaining information in $\cF_t\backslash \cG_t$.  Alternatively, this can be interpreted as that the player does not keep track of the unused information and thus cannot dig out more information from the past. This constraint simplifies our analyses significantly, and we shall explore the general case without this constraint in future research. We note that the (H*)-hypothesis is slightly stronger than the (H)-hypothesis, which is widely used in the literature, especially in credit risk theory, see e.g. Bremaud-Yor \cite{BY},  Dellacherie-Meyer \cite{DM}, Elliott-Jeanblanc-Yor \cite{EJY}, and  Kusuoka \cite{Kusuoka}.  

We next introduce the dynamic value function of our information control problem, and establish two crucial properties: the dynamic programming principle and the law invariance. The former does not involve the subtle measurability issue and thus its proof is straightforward, as in deterministic control problems. This property can be easily extended to more general settings. The law invariance, however, is quite subtle and relies heavily on the (H*)-hypothesis, and the general case will be left for future research. We note that the law invariance enables us to write the value function  on the space $\cP_2(\cP_2(\dbR^d))$, the subspace of $\cP(\cP(\dbR^d))$ equipped with the $2$-Wasserstein distance. Some basic properties of this space will be presented in Section \ref{sect-space}. As a consequence, we shall also establish the Lipschitz regularity of the value function, under the $2$-Wasserstein distance.

Another main result of this paper is the It\^{o} formula for smooth functions on $\cP_2(\cP_2(\dbR^d))$. Our results, both the definition of the derivatives and the It\^{o} formula, are natural generalization of those on $\cP_2(\dbR^d)$ in the mean field theory. We refer to Carmona-Delarue \cite{CD1, CD2} for a general exposition of the mean field theory, and Buckdahn-Li-Peng-Rainer \cite{BLPR}, Chassagneux-Crisan-Delarue \cite{CCD}, and Wu-Zhang \cite{WuZhang} for It\^{o} formula on $\cP_2(\dbR^d)$. There have been various extensions of It\^{o}'s formula, especially on measure-valued processes or on conditional laws, see e.g. Cox-Kallblad-Larsson-Svaluto-Ferro \cite{cox2024controlled},  Guo-Zhang \cite{guo2024itosformulaflowsconditional},  Liu-Zhang \cite{LiuZhang}, and Talbi-Touzi \cite{TalbiTouzi}. Roughly speaking, these works consider the It\^{o} formula in the form $u(t, \cL_{X_t|\cG_t})$, while we consider the form $U(t, \cL_{\cL_{X_t|\cG_t}})$, and they are consistent when our function $U$ is linear in $\bmu\in \cP_2(\cP_2(\dbR^d))$:  $U(t, \bmu) = \int_{\cP_2(\dbR^d)} u(t, \mu) \bmu(d\mu)$.  We shall provide some discussions on this in Section \ref{sect-general}. We note that, instead of using finite dimensional approximation and applying the result of the standard It\^{o} formula, as in \cite{WuZhang} we follow the arguments for the standard It\^{o} formula, namely by using time discretization and Taylor expansion. While it involves some subtlety due to the filtration issue, in particular it relies on the (H*)-hypothesis, the main idea of our proof of the It\^{o} formula is quite natural.

Combining the dynamic programming principle and the It\^{o} formula, it follows from rather standard arguments to derive the HJB equation on the space $[0, T]\times \cP_2(\cP_2(\dbR^d))$ for our dynamic value function of the information control problem. We show that, when the value function is smooth, it is the unique classical solution to the HJB equation. Moreover, we provide a verification result which helps to construct an optimal filtration by using the smooth value function. However, we shall note that, since the problem is nonlinear and the variable $\bmu$ is infinite dimensional, it is in general a tall order to obtain a classical solution. We shall leave the viscosity solution approach for future research. We also note that there have been very serious efforts on viscosity solutions for mean field control problems, namely HJB equations on $[0, T]\times \cP_2(\dbR^d)$, see e.g. Zhou-Touzi-Zhang \cite{ZTZ} and the references therein,  and it will be interesting to explore  those approaches for our HJB equation on $[0, T]\times \cP_2(\cP_2(\dbR^d))$.

The rest of the paper is organized as follows.  In Subsection \ref{sect-literature} we provide a brief literature review on some works closely related to our information control. In Section \ref{sect-insider1} we present a toy model of insider trading, which motivates us to consider the space $\cP_2(\cP_2(\dbR^d))$. Some basic properties of this space are presented in Section \ref{sect-space}. In Section \ref{sect-model} we introduce our continuous time information control problem and establish the basic properties of its dynamic value function. Section \ref{sect-Ito} is devoted to the It\^{o} formula for smooth functions on $\mathcal P_2(\mathcal P_2(\mathbb R^d))$. In Section \ref{sect-HJB} we derive the Hamilton–Jacobi–Bellman equation  for the dynamic value function and establish its well-posedness, provided that the value function has the desired regularity. In Section \ref{sect-general} we extend the It\^{o} formula further and discuss its connection with the related It\^{o} formulae in the the mean field literature. Finally, in Appendix \ref{appendixA} we present some examples, including a brief discussion on persuasion games; and in Appendix \ref{appendixB} we complete some technical proofs.

 \subsection{Some related literature}
 \label{sect-literature}
 First we emphasize that here information refers to $\si$-algebra or filtration, which is fundamentally different from the Shannon information \cite{Shannon1, Shannon2} and the Fisher information \cite{Fisher}. Information (in our sense) plays a critical role in many areas, e.g. filtering theory and game theory, and their applications, e.g. in economics and political science. There are numerous publications concerning asymmetric information and partial information, including the algorithmic side in computer science. Clearly, it is neither possible nor our intention to provide an overview here. In this short subsection we shall only comment on some works closely related to our information control, or say those comparing certain values under different information settings.

Our work is closely related to the Bayesian persuasion games, also known as information design, introduced by the seminal paper Kamenica and Gentzkow \cite{kamenica2011bayesian}, see also the earlier work  Crawford-Sobel \cite{crawford1982strategic}. This is a special type of signaling game that provides a specific way to disclose information, cf. Riley \cite{Riley} and Sobel \cite{Sobel}. The sender, who possesses private information represented by a random variable $X$, sends a signal according to a distribution that depends on $X$. The  receiver then infers the distribution of $X$ from the observed signal using the Bayes' rule. This is a Stackelberg game, where the receiver is the follower who solves a standard control problem, while the sender is the leader who designs the signal in an optimal way and thus solves an information control problem. Such a game has received very strong attention in the past decade, see e.g.  Caplin-Dean-Leahy \cite{caplin2022rationally}, Dworczak-Martini \cite{DworczakMartini}, Gentzkow-Kamenica \cite{gentzkow2014costly},  Kamenica \cite{kamenica2019bayesian}, Kamenica-Kim-Zapechelnyuk \cite{kamenica2021bayesian},  Orlov-Skrzypacz-Zryumov \cite{OSZ}, to mention a few. The model has also been extended to settings in which the receiver incurs costs for processing signals and thus also controls information, see e.g. Bloedel-Segal \cite{bloedel2021persuading} and Lipnowski-Ravid \cite{lipnowski2020cheap}, as well as to Nash games in which both players can control certain signals, see e.g. Au-Kawai \cite{AuKawai}.

All the above works consider only static models. Dynamic persuasion games were first studied by Ely \cite{ely2017beeps},
where the sender controls the rate of deterministic information provision. 
The work A\"id-Bonesini-Callegaro-Campi \cite{aid2025continuous} models stochastic dynamic persuasion as a linear-quadratic ergodic game with partial observation. Here the information control is through a process, namely one considers the filtration generated by a controlled process, and thus with the help of the filtering theory technically one may transform the problem into a somewhat standard control problem. More recently, in a discrete time setting, Liu-Zhou \cite{liu2025optimal} considers a Stackelberg game where the leader's private information is a random shock time. The leader designs randomized strategies to disclose the information, which is exactly in the spirit of the persuasion games, and \cite{liu2025optimal}  solves the problem using the dynamic programming principle. 

We would also like to mention the highly related work Kremer-Mansour-Perry \cite{KMP} on incentivizing exploration
via information asymmetry, and some very interesting discussions in Slivkins \cite{Silvkins}. This is also a leader follower problem. The leader, who possesses private information and aims to maximize social welfare, chooses an optimal information disclosure policy to incentivize the (multiple) followers, who arrive sequentially, to explore and generate new information.  

Another closed related area is the insider trading, especially the Kyle-Back model initiated by Kyle \cite{Kyle} and Back \cite{Back}. This is a Nash game in which the insider possesses private information about the fundamental value of a financial asset, and the market maker sets the asset’s price based on her perceived distribution of this information. The insider trader chooses a trading strategy to maximize expected gains by exploiting his informational advantage. Although his objective is not to disclose information, the market maker can infer some information from his trading strategy, typically via filtering. In this sense, the insider effectively engages in information control through the filtration associated with his strategy.
Such a model has also received very strong attention in the literature, we refer to Back-Baruch \cite{BackBaruch},  Bose-Ekren \cite{BoseEkren}, Cetin \cite{cetin}, Cetin-Danilova \cite{CetinDanilova}, Qiao-Zhang \cite{QiaoZhang}, and references therein for works on continuous time models. We would also like to mention the paper Jaimungal-Shi \cite{Jaimungal-Shi} which studies an optimal investment problem in a market with private information. However, instead of facing legal risk, the investor can purchase the information at a certain price. 

The following two interesting works in the mean field game framework are also closely related. The work Bergault-Cardaliaguet-Yan \cite{bergault2025optimal} considered a Stackelberg game in a linear quadratic setting for an insider trading problem. Unlike the Kyle-Back model, here the leader is an informed broker and the follower is a family of uninformed traders. However, the information disclosure mechanism is similar to the previous works: The broker controls the disclosure of the inside information through her trading strategy, and the traders infer the distribution of the inside information through filtering and solve a mean field game problem. So in \cite{bergault2025optimal} the information control is for the broker while the mean field game is for the traders with fixed information.  The work Becherer-Reisinger-Tam \cite{becherer2023mean} considered a mean field game with a different information structure: There is information delay (called observation delay) and each agent can control his/her speed of access to the information with certain cost, so the mean field equilibrium here indeed involves the information control. 

We emphasize that, in all the works mentioned above, information control is exercised either through a random variable or process, or via a conditional/joint law, whereas in this paper we directly control the $\si$-algebra or filtration.

We also remark that the persuasion game typically involves the space $\cP(\cP(\dbR^d))$, as we do in this paper. This space has also been used in the literature for different purposes. For example, Bertucci \cite{bertucci2022mean} studies a mean field game where the initial distribution of the population is unknown, and thus the prior distribution of this unknown initial distribution lies in the space of  $\mathcal{P}(\mathcal{P}(\mathbb{T}^d))$\footnote{The state space in \cite{bertucci2022mean} is the torus $\dbT^d$, rather than $\dbR^d$.}. The work  Beiglbock-Pflugl-Schrott \cite{BPS} studies adapted Wasserstein space in a discrete time setting. To capture the information in filtrations relevant to the process at each time step, they define the canonical spaces $\{\cZ_t\}_{t=0,\cdots, T}$ recursively backward, where the space $\cZ_t$ depends on $\cP(\cZ_{t+1})$. As a result, this involves not only $\cP(\cP(\dbR^d))$ but also (higher-order) nested Wasserstein space.

\section{A toy model of insider trading}
\label{sect-insider1}
To motivate our formulation of information control problems, we study a variation of the well-known Kyle \cite{Kyle} model on insider trading. The model consists of three agents for a financial asset: (i) a representative noise trader, who trades non-strategically and is the liquidity provider; (ii) an insider who observes the asset's fundamental value $X$ and chooses a trading strategy $\a$ to maximize his return; and (iii) a market maker who observes the total trading volume of the noise trader and the insider, but not the inside information $X$, and determines a pricing rule to break-even. The public, especially the market maker, view $X$ as a random variable and have some prior distribution. This is a Nash game between the insider and the market maker.

We consider several variations. First, in order to focus on the information control, we assume the market maker simply chooses $P=\dbE[X]$ as the pricing rule, instead of the conditional expectation (conditional on the total trading volume) as in the standard model.  We remark that this can be justified asymptotically by the camouflage effect of the insider when the population size of the noise traders is large, see the recent paper Ma-Xia-Zhang \cite{MXZ}, and then the game problem is reduced to a control problem of the insider. Second, we assume that the insider will actually use only partial information of $X$, modeled as a sub-$\si$-algebra $\cG\subset \si(X)$. This can be interpreted as that the insider himself is not allowed to trade, but he may choose to leak partial information  $\cG$ to a friend who will trade the asset on his behalf. Then, given $\cG$,  the insider's (or his friend's) optimization problem  is:
\bea
\label{insider-V0}
\sup_{\a\in \cG} \dbE\Big[\a \big(X - P\big)\Big] = \sup_{\a\in \cG} \dbE\Big[\a \big(\dbE[X|\cG] - \dbE[X]\big)\Big].
\eea
 This problem has a trivial solution: $\a = \infty$ if $\dbE[X|\cG] > \dbE[X]$ and $\a = -\infty$ if $\dbE[X|\cG] < \dbE[X]$. This is of course not reasonable in practice, and the missing piece is the legal risk of the insider. 
 
 We model the legal risk in two ways. First, we restrict to $|\a|\le R$ for some constant $R>0$, which is interpreted as the benchmark for abnormal trading and the insider will get investigated and then prosecuted if $|\a|>R$.\footnote{A more natural model is to introduce a cost function $C(\a)$ for abnormal trading. Here for simplicity we choose $C(\a) = \infty 1_{\{|\a|>R\}}$.}  Next, we assume that insider trading follows a pattern rather than occurring as a one-time action. Consequently, even if the trading volume appears ‘normal’ at each instance, the pattern may be detectable in the long run. We model this risk by $-{1\over 2} Var(X|\cG)$:  the larger $\cG$ is, the larger this risk is. Putting these into consideration, the insider's problem becomes:
\bea
\label{insider-V}
V:= \sup_{\cG\subset \si(X)} J(\cG),\q\mbox{where}\q  J(\cG):= \sup_{\a\in \cG, |\a|\le R}  \dbE\Big[\a \big(\dbE[X|\cG] - \dbE[X]\big) + {1\over 2} Var(X|\cG)\Big].
\eea
The optimization problem $J(\cG)$ is trivial: with optimal strategy $\a^* := R~ \mbox{sign}\big(\dbE[X|\cG] - \dbE[X]\big)$,
\bea
\label{insider-u}
J(\cG):= \dbE\Big[ R \big|\dbE[X|\cG] - \dbE[X]\big| + {1\over 2} Var(X|\cG)\Big].
\eea
Then \eqref{insider-V} becomes the following information control problem:
\bea
\label{insider-V2}
V:= \sup_{\cG\subset \si(X)} \dbE\Big[ R \big|\dbE[X|\cG] - \dbE[X]\big| + {1\over 2} Var(X|\cG)\Big].
\eea

\begin{prop}
\label{prop-insider}
Let $X$ be a standard normal. Then the problem \eqref{insider-V2} has an explicit solution:
\bea
\label{insider-sol1}
\displaystyle\mbox{If}~ R \ge R_0: &&\!\!\!\!  V = R_0 R - \tfrac{1}{\pi} + \tfrac{1}{2},\q\mbox{with optimal}~\cG^* := \si\big(\{X> 0\}\big).\\
\label{insider-sol2}
\displaystyle\mbox{If}~ R <  R_0: &&\!\!\!\!   V =  \tfrac{1}{2} R^2 +  \tfrac{1}{2},\q\mbox{with optimal}~\cG^* := \si\big(\{-a_R < X<0\} \cup \{X>a_R\}\big),
\eea
where $R_0:= \dbE[|X|] = \sqrt{\tfrac{2}{\pi}}$, and $a_R = I^{-1}(R)$ for the following funtion $I$: 
\bea
\label{insider-a*}
I(a):= R_0 - 2 \dbE\big[|X| 1_{\{|X|\le a\}}\big], ~a\ge 0.
\eea
\end{prop}
\proof Note that $\dbE[X]=0$ and $\dbE[X^2] = 1$, then
\beaa
J(\cG) &=& \dbE\Big[ R \big|\dbE[X|\cG]\big| + \tfrac{1}{2}\dbE[X^2|\cG] - \tfrac{1}{2} \big|\dbE[X|\cG]\big|^2\Big] \\
&=&-\tfrac{1}{2} \dbE\Big[\big(\big|\dbE[X|\cG]\big| - R\big)^2\Big] + \tfrac{1}{2} R^2 + \tfrac{1}{2} \dbE[X^2]\\
&=& \tfrac{1}{2}\big[-\tilde J(\cG) + R^2 + 1\big],
\eeaa
where
\bea
\label{tildeu}
\tilde J(\cG) :=  \dbE\Big[\big(\big|\dbE[X|\cG]\big| - R\big)^2\Big].
\eea
This implies that
\bea
\label{insider-V3}
V = \tfrac{1}{2}\big[- \tilde V + R^2 + 1\big],\q\mbox{where}\q \tilde V:= \inf_{\cG\subset \si(X)} \tilde J(\cG).
\eea

To solve the optimization problem $\tilde V$, we note that 
\beaa
\dbE\Big[\big|\dbE[X|\cG]\big|\Big] \le \dbE\Big[\dbE[|X||\cG]\Big] = \dbE[|X|] = R_0.
\eeaa
When $R\ge R_0$,  we have
\beaa
\tilde J(\cG) \ge \Big(\dbE\big[\big|\dbE[X|\cG]\big| - R\big]\Big)^2 \ge \Big(R- \dbE[|X|]\Big)^2 \ge \big(R-R_0\big)^2.
\eeaa
Note further that,  for the $\cG^*$ in \eqref{insider-sol1}, we have 
\beaa
\dbE[X|\cG^*] = {\dbE[ X1_{\{X>0\}}]\over \dbP(X>0) } 1_{\{X>0\}} + {\dbE[ X1_{\{X\le 0\}}]\over \dbP(X\le 0) } 1_{\{X\le 0\}} = R_0 1_{\{X>0\}} - R_0 1_{\{X\le 0\}}.
\eeaa
Then
\beaa
\big|\dbE[X|\cG^*]\big| = R_0,~\mbox{a.s., and thus}\q \tilde J(\cG^*) = (R-R_0)^2.
\eeaa
This implies that $\tilde V = \big(R-R_0\big)^2$ and $V = R_0 R - \tfrac{1}{\pi} + \tfrac{1}{2}$, proving \eqref{insider-sol1}.

We now consider the case  $R<R_0$. Denote $E_a:= \{-a < X<0\} \cup \{X>a\}$ and $\cG_a := \si(E_a)$, for $a\ge 0$. Since $X$ is symmetric, it is clear that 
\beaa
&&\dbP(E_a) = \frac{1}{2},\q \dbE[X 1_{E_a}] = \dbE[X 1_{\{X>0\}}] - \dbE[|X|1_{\{|X|\le a\}}] = \tfrac{1}{2} I(a),\\
&&\dbE[X|\cG_a] = {\dbE[X 1_{E_a}]\over \dbP(E_a)} 1_{E_a} +  {\dbE[X 1_{E_a^c}]\over \dbP(E_a^c)} 1_{E_a^c} = I(a) 1_{E_a} - I(a) 1_{E_a^c},
\eeaa
and thus $\big|\dbE[X|\cG_a]\big|=I(a)$, a.s. It is clear that $I(a)$ is continuous and strictly decreasing in $a$. One can easily check that $I(0) = R_0$ and $I(\infty) = -R_0$, then there exists unique $a_R>0$ such that $I(a_R) =R$.  This implies that $\big|\dbE[X|\cG_{a_R}]\big|=I(a_R) = R$, a.s., and thus $\tilde J(\cG_{a_R}) = 0$. Note that $\tilde J(\cG)\ge 0$ for all $\cG$, then we obtain $\tilde V =0$ and hence $V=  \tfrac{1}{2} R^2 +  \tfrac{1}{2}$ with optimal $\cG^*=\cG_{a_R}$.
\qed

\begin{remark}
\label{rem-insider}
The market price is $P=\dbE[X]=0$, so  the asset is underpriced when $X>0$, and overpriced when $X<0$.

\smallskip
\noindent(i) When $R \ge R_0$, the insider will leak information indicating whether the asset is underpriced  or overpriced, without telling the exact value of $X$. The friend trading on his behalf will take the full advantage of this information with optimal trading strategy $\a^* = R 1_{\{X>0\}} + (- R) 1_{\{X<0\}}$, namely he will buy the maximum amount $R$ whenever the asset is underpriced and sell $R$ whenever the asset is overpriced. We note that, although each trade is safe (since the trading volume is below the benchmark $R$),  such a pattern can be easily detected. However, since $R$ is large, the gain from the trading surpasses the legal risk, so the insider is willing to take the risk.

\smallskip
\noindent(ii) When $R < R_0$, the corresponding $\a^* = R 1_{E_{a_R}} + (- R) 1_{E_{a_R}^c}$. In this case, the gain from each trade is not large enough to outweigh the legal risk associated with the pattern. Therefore, the insider chooses $\cG^*$ to camouflage his trading: he buys even when the asset is slightly overpriced ($-a_R < X <0$) and sells when it is slightly underpriced ($0<X <a_R$). 

\smallskip
\noindent(iii) The above analysis is based on our modeling of the pattern-related legal risk as $-\tfrac{1}{2} Var(X|\cG)$. If the regulator detects the pattern in (ii) and pays special attention to this type of trading behavior, the legal risk may change, which in turn could alter the optimal $\cG^*$. Such kind of consideration will lead to a game between the regulator and the insider, which will be left for future research.

\smallskip
\noindent(iv) In \eqref{insider-sol1} the optimal $\cG^*$ is unique, while in  \eqref{insider-sol2} $\cG^*$ is not unique. In the latter case, a sub-$\si$-algebra $\cG$ is optimal if and only if $\big|\dbE[X|\cG]\big| = R$, a.s.  So, when the insider faces a strategic regulator, as described in (iii), he may choose different $\cG^*$ over time to further camouflage his behavior.
\end{remark}

To conclude this section, we note that the integrand $R \big|\dbE[X|\cG] - \dbE[X]\big| + {1\over 2} Var(X|\cG)$ in \eqref{insider-u} is a function of the conditional law $\cL_{X|\cG}$, which is a random measure, or say a $\cP(\dbR)$-valued random variable. Then $J(\cG)$ can be viewed as a (deterministic) function of $\cL_{\cL_{X|\cG}}$, which is an element of $\cP(\cP(\dbR))$. In general, we shall consider the following optimization problem:
\bea
\label{insider-general}
V := \sup_{\cG \subset \cF} G(\cL_{\cL_{X|\cG}}),
\eea
where $G: \cP(\cP(\dbR^d))\to \dbR$, $X$ is an $\dbR^d$-valued random variable on some probability space, and $\cF$ is the largest information the player could possibly access. In Section  \ref{sect-model}, we extend the analysis to the dynamic setting. We note that, in general, $X$ may or may not be $\cF$-measurable; that is, the player may not have full access to 
$X$, even if he is willing to incur all associated costs. In this paper, however, we do not consider this level of generality.

\section{A primer on the space $\cP(\cP(\dbR^d))$}
\label{sect-space}

In this paper we shall consider $\cP_2(\cP_2(\dbR^d))$ equipped with the $2$-Wasserstein metric $\cW_2$. To be precise, for any $p\ge 1$,  $\mu, \mu_1, \mu_2\in \cP_p(\dbR^d)$, $\bmu, \bmu_1, \bmu_2 \in \cP_p(\cP_p(\dbR^d))$,
\bea
\label{W2}
\left.\ba{lll}
\displaystyle W^p_p(\mu_1, \mu_2) := \inf\Big\{ \int_{\dbR^{d}\times \dbR^d} |x_1-x_2|^p \pi(dx_1, dx_2): \pi\in \cP_p(\mu_1, \mu_2)\Big\};\\
\displaystyle\cW^p_p(\bmu_1, \bmu_2) := \inf\Big\{ \int_{\cP_p(\dbR^d)\times \cP_p(\dbR^d)} W_p^p(\mu_1, \mu_2)\Pi(d\mu_1, d\mu_2): \Pi\in \cP_p(  \bmu_1,  \bmu_2)\Big\};\\
\displaystyle \|\mu\|_p^p := W_p^p(\mu, \d_0) = \int_{\dbR^d} |x|^p \mu(dx),\q \|\bmu\|_p^p:= \cW_p^p(\bmu, \d_{\d_0}) = \int_{\cP_p(\dbR^d)}\int_{\dbR^d} |x|^p \mu(dx) \bmu(d\mu).
\ea\right.
\eea
where
\bea
\label{cPmu12}
&& \cP_p(\mu_1, \mu_2) := \Big\{\pi \in \cP_p(\dbR^d\times \dbR^d): \pi(dx_1, \dbR^d) = \mu_1(dx_1), \pi(\dbR^d, dx_2) = \mu_2(dx_2)\Big\};\nonumber\\
&& \cP_p(\bmu_1, \bmu_2) := \Big\{\Pi \in \cP_p(\cP_p(\dbR^d)\times \cP_p(\dbR^d)): \\
&&\qquad\qquad\qquad\qquad \Pi(d\mu_1, \cP_p(\dbR^d)) = \bmu_1(d\mu_1), \Pi(\cP_p(\dbR^d), d\mu_2) = \bmu_2(d\mu_2)\Big\}.\nonumber
 \eea
Since $(\cP_p(\dbR^d), W_p)$ is a Polish space, then the space $(\cP_p(\cP_p(\dbR^d)), \cW_p)$ is also Polish, in particular it is separable,   cf. \cite[Theorem 6.18]{OTV}. 
Moreover, we shall use the following well-known Kantorovich representation for $W_1$ and $\cW_1$:
\bea
\label{W1rep}
\left.\ba{lll}
\displaystyle W_1(\mu_1, \mu_2) = \sup_\f \int_{\dbR^d} \f(x) [\mu_1(dx)-\mu_2(dx)],\\
\displaystyle \cW_1(\bmu_1, \bmu_2) = \sup_\Phi \int_{\cP_1(\dbR^d)} \Phi(\mu) [\bmu_1(d\mu) - \bmu_2(d\mu)],
\ea\right.
\eea
where the supremums are over all Lipschitz continuous functions $\f: \dbR^d\to \dbR$ and $\Phi: \cP_1(\dbR^d) \to \dbR$ with Lipschitz constant $1$.

We start with a simple result which identifies $\cP(\cP(\dbR^d
))$ with sub-$\si$-algebras  $\cG$ when considering the canonical random variable.
\begin{prop}
\label{prop-equiv}
Let  $(\O, \cF, \dbP)$ be a complete probability space and $X$ an $\dbR^d$-valued random variable. For any sub-$\si$-algebras $\cG_1, \cG_2\subset \si(X)$,  $\cL_{\cL_{X|\cG_1}} = \cL_{\cL_{X|\cG_2}}$ if and only if $\cG_1=\cG_2$, a.s.
\end{prop}
\noindent Here and in the sequel, $\cG_1\subset \cG_2$, a.s. means $\cG_1\subset \cG_2\vee \cN$, where $\cN$ denotes the set of all null events. Similarly $\cG_1 = \cG_2$, a.s. means that $\cG_1\subset \cG_2$, a.s. and $\cG_2\subset \cG_1$, a.s. 
\proof The "if" direction is obvious, so we shall prove the "only if" direction. Assume $\cL_{\cL_{X|\cG_1}} = \cL_{\cL_{X|\cG_2}}$, and without loss of generality it suffices to prove $\cG_1 \subset \cG_2$. We note that, all the equalities and inclusions in this proof should be interpreted in a.s. sense.

 Since $\cL_{\cL_{X|\cG_1}} = \cL_{\cL_{X|\cG_2}}$, we have $\cL_{\dbE[\phi(X)|\cG_1]} = \cL_{\dbE[\phi(X)|\cG_2]}$ for any bounded and Borel measurable function $\phi$. Now for any $A\in \cG_1\subset \si(X)$, there exists $\tilde A\in \cB(\dbR^d)$ such that $1_A = 1_{\tilde A}(X)$. By considering $\phi = 1_{\tilde A}$, we have $\cL_{\dbE[1_A|\cG_1]} = \cL_{\dbE[1_A|\cG_2]}$. This implies that, since $A\in \cG_1$,
 \beaa
 \dbE\big[1_A\big] = \dbE\big[\big|\dbE[1_A|\cG_1]\big|^2\big] =  \dbE\big[\big|\dbE[1_A|\cG_2]\big|^2\big] \le  \dbE\big[\dbE[|1_A|^2|\cG_2]\big] = \dbE\big[1_A\big].
 \eeaa
So we must have  $ \big|\dbE[1_A|\cG_2]\big|^2= \dbE[|1_A|^2|\cG_2]$, a.s. This implies that $A\subset \cG_2$, and thus $\cG_1\subset \cG_2$.
\qed

When $\cG_1, \cG_2$ are not included in $\si(X)$, however, the "only if" direction in the above statement fails, as the following simple example shows. In this paper, we thus focus on the case that the admissible $\si$-algebras are sub-$\si$-algebras of the state process. The general case will go beyond $\cP(\cP(\dbR^d))$ and we shall leave to future research.

\begin{example}
\label{eg-nonequiv}
Let $(\O, \cF, \dbP)$ be a probability space, which supports independent random variables $X, Y\sim$ Bernoulli(${1\over 2}$). Set $\cG_1 :=\{\emptyset, \O\} $ and $\cG_2 := \sigma(Y)$. Then $\cL_{\cL_{X|\cG_1}} =  \cL_{\cL_{X|\cG_2}} = \d_{\cL_X}$, but obviously $\cG_1 \neq \cG_2$, a.s.
\end{example}

The next result helps us to understand the space $\cP(\cP(\dbR^d))$ further. 
\begin{thm}
\label{thm-support}
Let $(\O, \cF, \dbP)$ be an atomless probability space. For any $\bmu\in \cP(\cP(\dbR^d))$, there exist random variables $X, Y\in \dbL^0(\cF)$ such that $\bmu = \cL_{\cL_{X|Y}}$. 
\end{thm}
\begin{proof} Since the probability space is atomless, it supports independent random variables $U_1\sim$ Uniform$([0,1])$ and $U_2\sim$ Uniform($[0,1]^d$). We prove the result in three cases.

{\bf Case 1.} We first assume $\bmu$ is discrete, that is, $\bmu = \sum_{i\ge 1} p_i \d_{\mu_i}$, where $\mu_i\in \cP(\dbR^d)$. Let $Y\in \dbL^0(\si(U_1))$ be discrete such that $\dbP(Y=y_i) = p_i$ and $X_i \in \dbL^0(\si(U_2))$ such that $\cL_{X_i} = \mu_i$, $i\ge 1$. Define $X := \sum_{i\ge 1} X_i 1_{\{Y=y_i\}}$. Then it is straightforward to verify that $\cL_{\cL_{X|Y}}=\bmu$.

{\bf Case 2.}  We next prove the case that  $ supp(\bmu)\subset   \cP([0,1)^d)$. That is, for each $\mu\in supp(\bmu)$, $supp(\mu) \subset [0,1)^d$. Our idea is to use discrete $\bmu^n$ to approximate $\bmu$, and by Case 1, we have $\bmu^n = \cL_{\cL_{X^n|Y^n}}$ for appropriate  random variables $(X^n, Y^n)$. The goal is then to prove the convergence of $(X^n, Y^n)$. However, note that the conditional distribution operator is discontinuous under the joint distribution, this proof is quite lengthy and we postpone it to Appendix \ref{appendixB}.

{\bf Case 3.} We now consider the general case $\bmu \in \cP(\cP(\dbR^d))$. Introduce $I: \dbR^d\to (0,1)^d$ by $I_i(x) := {e^{x_i}\over 1+e^{x_i}}$, $i=1,\cdots, d$, and for any $\mu\in  \cP(\dbR^d)$, $I_{\#}\mu\in \cP((0,1)^d)$ denotes the push-forward of $\mu$ by $I$. That is, $I_{\#}\m(D) := \mu(I^{-1}(D))$ for any $D\in \cB((0,1)^d)$. Since $I: \dbR^d\to (0,1)^d$ is a bijection, it is clear that the mapping $I_\#: \cP(\dbR^d) \to \cP((0,1)^d)$ is also a bijection. Denote further that $I_{\#}E:=\{I_{\#}\mu: \mu\in E\}$ for any $E\in \cB(\cP(\dbR^d))$, then $I_\#$ is also a bijection from $\cB(\cP(\dbR^d))$ to $\cB(\cP((0,1)^d))$. 
Now given $\bmu \in \cP(\cP(\dbR^d))$, define $\tilde \bmu\in  \cP(\cP((0,1)^d))$ by:
\beaa
\tilde \bmu(\tilde E) := \bmu\Big(\big\{\mu\in \cP(\dbR^d): I_{\#}\mu \in \tilde E\big\}\Big),\q \forall \tilde E\in \cB(\cP((0,1)^d)).
\eeaa
By the previous case, there exist $(\tilde X, Y)$ such that $\cL_{\cL_{\tilde X|Y}} = \tilde \bmu$. It is clear that $0<\tilde X<1$, a.s. We claim that  $\cL_{\cL_{I^{-1}(\tilde X)|Y}}=\bmu$, which proves the theorem with $X := I^{-1}(\tilde X)$.  

To see this, for any $E\in \cB(\cP(\dbR^d))$, we have
\beaa
\bmu(E) &=& \tilde \bmu(I_{\#}E)=\cL_{\cL_{\tilde X|Y}}(I_{\#}E)=\dbP\left(\{\o: \cL_{\tilde X|Y}(\o)\in I_{\#}E \}\right)\\
&=&\dbP\left(\cup_{\mu\in E}\{\o: \cL_{\tilde X|Y}(\o)=I_{\#}\mu\}\right)\\
&=&\dbP\left(\cup_{\mu\in E}\{\o: \cL_{\tilde X|Y}(\o)(\tilde D)=I_{\#}\mu(\tilde D), \forall \tilde D\in \cB((0,1)^d)\}\right)\\
&=&\dbP\left(\cup_{\mu\in E} \{\o:  \cL_{X|Y}(\o)( D)=I_{\#}\mu(I(D))=\mu( D),\forall  D\in \cB(\dbR^d)\}\right)\\
&=&\dbP\left(\cup_{\mu\in E}\{\o: \cL_{X|Y}(\o)=\mu\}\right)=\dbP\left(\{\o: \cL_{X|Y}(\o)\in E\}\right)=\cL_{\cL_{X|Y}}(E).
\eeaa
That is,  $\bmu=\cL_{\cL_{X|Y}}$.
\end{proof}

The following result  will be crucial for the regularity of the value function in the next section. 

\begin{thm}
\label{thm-coupling}
Let $(\O, \cF, \dbP)$ be an atomless probability space. For any  $\bmu_1,\bmu_2\in\cP_p(\cP_p(\dbR^d))$, there exist random variables $X_1,  X_2\in \dbL^p(\cF)$ and $ Y\in \dbL^0(\cF)$  such that 
\bea
\label{coupling}
\cL_{\cL_{X_i|Y}}=\bmu_i, ~i=1,2;\q\mbox{and}\q \dbE[|X_1-X_2|^p]=\cW_p^p(\bmu_1,\bmu_2).
\eea
\end{thm}
The main idea is to combine the following result, which extends Theorem \ref{thm-support} to the space $\cP(\dbR^d)\times \cP(\dbR^d)$, and a measurable selection theorem.  We postpone the proof to Appendix \ref{appendixB}.

\begin{corollary}
\label{cor-multidim-support}
Let $(\O, \cF, \dbP)$ be an atomless probability space. For any $\bmu\in \cP(\cP(\dbR^d)\times\cP(\dbR^d) )$, there exist random variables $X_1, X_2, Y\in \dbL^0(\cF)$ such that $\bmu = \cL_{\left(\cL_{X_1|Y},\cL_{X_2|Y}\right) }$. 
\end{corollary}
\proof For any $\mu_1, \mu_2\in \cP(\dbR^d)$, introduce the independent decomposition: $(\mu_1\otimes\mu_2)(A_1\times A_2) := \mu_1(A_1)\mu_2(A_2)$, $\forall A_1, A_2\in \cB(\dbR^d)$. Clearly the mapping $(\mu_1, \mu_2)\to \mu_1\otimes\mu_2$ is one to one. Its image $\cK:= \{\mu_1\otimes\mu_2: \mu_1, \mu_2\in \cP(\dbR^d)\times \cP(\dbR^d)\}\subset  \cP(\dbR^{2d})$ is  Borel measurable (cf. \cite[Theorem 15.1]{Kechris1995}). Then, for any $\bmu\in \cP(\cP(\dbR^d)\times\cP(\dbR^d) )$, we may lift it to $\cP(\cP(\dbR^{2d}))$: 
\bea
\label{bmulift}
\hat \bmu(A) := \bmu\big(A \cap \cK\big),\q\forall A\in \cB(\cP(\dbR^{2d})).
\eea
By Theorem \ref{thm-support}, there exist $X=(X_1, X_2), Y \in \dbL^0(\cF)$ such that $\hat\bmu = \cL_{\cL_{(X_1, X_2)|Y}}$. Now by \eqref{bmulift}, one can easily verify that $\bmu = \cL_{\left(\cL_{X_1|Y},\cL_{X_2|Y}\right) }$.
\qed

\section{A continuous time information control problem}
\label{sect-model}
Fix a finite time duration $[0, T]$. Consider the canonical setting:  $\O := C([0, T]; \dbR^d)$,  $\bar B(\o) := \o$, and $\dbF := \dbF^{\bar B}$. Let $B_t := \bar B_t - \bar B_0$, and $\dbP$ the Wiener measure so that $B$ is a $\dbP$-Brownian motion. Note that $\dbF =\{\cF_t\}_{0\le t\le T} := \cF_0\vee \dbF^B$, and we shall always consider its augmented version. We also assume  the $\si$-algebra $\cF_0$ is atomless under $\dbP$, as required in Theorem \ref{thm-support}. Given $t\in [0, T]$, introduce the shifted objects: $B^t_s:= B_s-B_t$, $\cF^t_s:= \si(B^t_r, t\le r\le s)$, and $\dbF^t := \{\cF^t_s\}_{t\le s\le T}$. 
The player's control is a sub-filtration $\dbG = \{\cG_t\}_{0\le t\le T}\subset \dbF$. In this paper, we shall restrict $\dbG$ to satisfy  the following (H*)-hypothesis:
\begin{itemize}
\item[]{\bf (H*)-hypothesis:} For any $0\le t<s\le T$, $\cG_s\subset\cG_t\vee \cF_s^t$, a.s., namely $\cG_s\subset\cG_t\vee \cF_s^t\vee \cN$.
\end{itemize}

 \noindent In our context, this means that, given $\cG_t$,  at any future time $s>t$, the player can only acquire additional information from $\cF^t_s$; apart from $\cG_t$, the information contained in $\cF_t$ is no longer available. This restriction makes our analyses a lot easier, especially for the law invariance property in Theorem \ref{thm-conditional-law-invariant} and the It\^{o} formula in Theorem \ref{thm-Ito} below, and we shall investigate the general case in future research. A typical example satisfying the (H*)-hypothesis is as follows. \begin{example}
 \label{eg-Hhypothesis}
Let $\cX_0 \in \dbL^2(\cF_0)$, then $\mathbb G:=\mathbb F^{\cX}$ satisfies the (H*)-hypothesis:
$$
\cX_t= \cX_0+\int_0^t \a(s, \cX_{[0,s]})ds+\int_0^t\b(s, \cX_{[0,s]})dB_s,
$$
where $\a, \b$ are either uniformly Lipschitz continuous in $\cX$ under the uniform norm, or piecewise constant in the sense that $\a(t,  \cX_{[0,t]})= \sum_{i=0}^{n-1} \a(t_i,  \cX_{[0,t_i]})1_{[t_i, t_{i+1})}(t)$ for some partition $0=t_0<\cdots<t_n=T$, and similarly for $\b$.
\end{example}

As the name suggests, the (H*)-hypothesis is closely related to the (H)-hypothesis, which is widely used in the literature  (cf. \cite{BY, DM, EJY, Kusuoka}):
\begin{itemize}
\item[] {\bf (H)-hypothesis:} Any square integrable $\dbG$-martingale is an $\dbF$-martingale. Equivalently, for any $t$, $\cF_t$ and $\cG_T$ are conditionally independent, conditional on $\cG_t$.
\end{itemize}

\smallskip

 \begin{prop}
 \label{prop-H}
A sub-filtration $\dbG\subset \dbF$ satisfies the (H*)-hypothesis if and only if, 
\bea
\label{H*}
\mbox{for any $t<s$, $\cG_s \vee \cF^t_s$ and $\cF_t$ are conditionally independent, conditional on $\cG_t$.}
\eea
 In particular, (H*)-hypothesis implies (H)-hypothesis.
 \end{prop}

We postpone the proof to Appendix. 
Note that the (H)-hypothesis implies that both $\cG_T$ and $\cF^t_T$ are conditionally independent  with  $\cF_t$, conditionally on $\cG_t$. However, pairwise independence does not imply the joint independence. In general, this does not imply the conditional independence between $\cG_T\vee \cF^t_T$ and $\cF_t$. Indeed, in general (H)-hypothesis does not imply (H*)-hypothesis, as the following simple example in the discrete setting shows. 
\begin{example}
Let $X_1$ and $X_2$ be independent Bernoulli random variables with $\dbP(X_i=1)=\dbP(X_i=0)=\frac{1}{2}$. Set $\dbF:=\{\cF_1,\cF_2\}$ with $\cF_1:=\si(X_1)$, $\cF_2:=\si(X_1,X_2)$, $\cF^1_2 := \si(X_2)$, and $\dbG:=\{\cG_1,\cG_2\} \subset \dbF$ with $\cG_1:=\{\emptyset, \O\}$, $\cG_2:=\si(\{X_1=X_2\})$. Then, one can easily verify that $\cG_2$ and $\cF_1$ are independent, and thus $\dbG\subset\dbF$ satisfies the (H)-hypothesis. However,  note that $\cG_2 \vee \cF^1_2 = \si(X_1, X_2)$ is not a subset of $\cG_1 \vee \cF^1_2 = \si(X_2)$, and in fact $\cG_2 \vee \cF^1_2 $ is not independent of $\cF_1$, so (H*)-hypothesis fails.
\end{example}

We now introduce our control problem.  For each $t\in [0, T]$, denote by $\cA_t$ the set of all sub-filtrations $\dbG\subset \dbF$ on $[t, T]$ satisfying the (H*)-hypothesis. That is, $\dbG = \{\cG_s\}_{s\in [t, T]}$ such that $\cG_s \subset \cF_s$, $s\in [t, T]$, and $\cG_{s_2} \subset \cG_{s_1}\vee \cF^{s_1}_{s_2}$, a.s., $t\le s_1 < s_2 \le T$.  For any $t\in [0, T]$, $\cG^0_t \subset \cF_t$,  set 
\bea
\label{cA}
\cA(t, \cG^0_t) := \{\dbG\in \cA_t: \cG_t = \cG^0_t\}.
\eea
Note that, for any  $\cG^0_t \subset \cF_t$, $\cA(t, \cG^0_t)$ contains the following trivial element and  thus is non-empty:   $\cG_s \equiv \cG^0_t$, $s\in [t, T]$.
The state process is modeled as follows:
\bea
\label{X}
dX_t = b(t, X_t) dt + \si(t, X_t) dB_t.
\eea
Since $\si$ is allowed to be degenerate, we assume without loss of generality that $X$ is also $d$-dimensional. We remark that in this paper the dynamics of $X$ is fixed and does not involve any control. Motivated by \eqref{insider-general}, for any $t\in [0, T]$, $\xi\in \dbL^2(\cF_t)$, and $\cG^0_t \subset \cF_t$,  our information control problem takes the form:
\bea
\label{VtG}
V_0(t, \ul\xi, \cG^0_t) := \sup_{\dbG \in \cA(t,\cG^0_t)} J(t, \ul\xi, \dbG),~\mbox{where}~ J(t, \ul\xi, \dbG):= G(\cL_{\cL_{X^{t,\xi}_T|\cG_T}}) + \int_t^T F(s,\cL_{\cL_{X^{t,\xi}_s|\cG_s}})ds,
\eea 
where $F: [0, T]\times \cP_2(\cP_2(\dbR^d))\to \dbR$, $G:  \cP_2(\cP_2(\dbR^d))\to \dbR$, and $X^{t,\xi}$ denotes the solution to SDE \eqref{X} on $[t, T]$ with initial condition $X_t=\xi$. Here the functions $J$ and $V_0$ depend on the whole random variable $\xi$ (or say, the mapping $\xi: \O\to\dbR$), and throughout the paper we use the notation $\ul\xi$ to emphasize such a dependence.  In particular, $J(t, \ul \xi, \dbG), V_0(t, \ul \xi, \cG^0_t)\in \dbR$ are deterministic values.

Throughout the paper, the following standing assumptions are in force.
\begin{assumption}
\label{assum-standing}
(i) $b, \si$ are continuous in $t$ and uniformly Lipschitz continuous in $x$;\\
(ii) $F$ is continuous in $t$, and $F, G$ are locally uniformly continuous\footnote{In the finite dimensional case, the corresponding $\cK_R:= \{x\in \dbR^n: |x|\le R\}$ is compact, and thus local uniform continuity is equivalent to pointwise continuity. For the infinite dimensional space here, $\cK_R$ is not compact, and thus local uniform continuity is stronger than pointwise continuity.}  in $\bmu$ under $\cW_2$ in the sense that they are uniformly continuous  in the set $\cK_{R}:=\{\bmu\in \cP_2(\cP_2(\dbR^d)): \|\bmu\|_2 \le R\}$ for $\forall R>0$.
\end{assumption}
Under these conditions, for any initial condition $(t, \xi)$,  \eqref{X} is well-posed with
\bea
\label{Xest}
\sup_{t\le s\le T} \|\cL_{\cL_{X^{t,\xi}_s|\cG_s}}\|_2 \le \Big(\dbE\big[ \sup_{t\le s\le T} |X^{t,\xi}_s|^2\big]\Big)^{1\over 2} \le C[1+\|\xi\|_2] =: R_0(\xi).
\eea
Then $F, G$ are uniformly continuous in $\cK_{R_0(\xi)}$. Note that uniform continuity implies linear growth and hence boundedness in $\cK_{R_0(\xi)}$, then $J$ and $V_0$ are well defined.

We first prove the dynamic programming principle for $V_0$, which is technically easy  since $J(t, \ul\xi,  \dbG)$ involves only deterministic laws and there is no subtle  measurable selection issue. 
\begin{prop}[DPP]
\label{prop-DPP}
For any $0\le t<t+\d\le T$, $\xi\in \dbL^2(\cF_t)$,  and $\cG^0_t\subset \cF_t$, it holds that
\bea
\label{DPP}
V_0(t, \ul\xi, \cG^0_t) = \sup_{\dbG \in \cA(t,\cG^0_t)} \Big[V_0(t+\d, \ul {X^{t,\xi}_{t+\d}}, \cG_{t+\d}) + \int_t^{t+\d} F(s,\cL_{\cL_{X^{t,\xi}_s|\cG_s}})ds\Big].
\eea 
\end{prop}
\proof Let $\tilde V_0(t, \ul\xi, \cG^0_t)$ denote the right side of \eqref{DPP}. We first prove $V_0(t, \ul\xi, \cG^0_t)\le \tilde V_0(t, \ul\xi, \cG^0_t)$. For any $\dbG\in \cA(t, \cG^0_t)$, it is clear that $X^{t,\xi}_s = X^{t+\d, X^{t,\xi}_{t+\d}}_s$, $s\ge t+\d$, and $\dbG \in \cA(t+\d, \cG_{t+\d})$. Then
\beaa
J(t, \ul \xi, \dbG) 
&=& \Big[G(\cL_{\cL_{X^{t,\xi}_T|\cG_T}}) + \int_{t+\d}^T F(s,\cL_{\cL_{X^{t,\xi}_s|\cG_s}})ds\Big] + \int_t^{t+\d} F(s,\cL_{\cL_{X^{t,\xi}_s|\cG_s}})ds\\
&\le& V_0(t+\d,  \ul {X^{t,\xi}_{t+\d}}, \cG_{t+\d}) + \int_t^{t+\d} F(s,\cL_{\cL_{X^{t,\xi}_s|\cG_s}})ds \le \tilde V_0(t, \ul \xi, \cG^0_t).
\eeaa
By the arbitrariness of $\dbG\in \cA(t, \cG^0_t)$, it follows from \eqref{VtG} that $V_0(t, \ul \xi, \cG^0_t) \le \tilde V_0(t, \ul \xi, \cG^0_t)$.

We next prove $\tilde V_0(t, \ul \xi, \cG^0_t) \le V_0(t, \ul \xi, \cG^0_t)$. For any $\dbG\in \cA(t, \cG^0_t)$ and any $\e>0$, by \eqref{VtG} for $V_0(t+\d, \ul {X^{t,\xi}_{t+\d}}, \cG_{t+\d})$ there exists $\dbG^\e\in \cA(t+\d, \cG_{t+\d})$ such that 
\beaa
V_0(t+\d, \ul {X^{t,\xi}_{t+\d}}, \cG_{t+\d}) \le G(\cL_{\cL_{X^{t,\xi}_T|\cG^\e_T}}) + \int_{t+\d}^T F(s,\cL_{\cL_{X^{t,\xi}_s|\cG^\e_s}})ds+\e.
\eeaa
Denote $\bar \cG^\e_s:= \cG_s 1_{[t, t+\d)}(s) + \cG^\e_s1_{[t+\d, T]}(s)$. Since $\cG^\e_{t+\d} = \cG_{t+\d}$, one can easily check that $\bar\dbG^\e\in \cA(t, \cG^0_t)$. Then 
\beaa
&&V_0(t+\d, \ul {X^{t,\xi}_{t+\d}},  \cG_{t+\d}) + \int_t^{t+\d} F(s,\cL_{\cL_{X^{t,\xi}_s|\cG_s}})ds\\
&&\le G(\cL_{\cL_{X^{t,\xi}_T|\cG^\e_T}}) + \int_{t+\d}^T F(s,\cL_{\cL_{X^{t,\xi}_s|\cG^\e_s}})ds+\e + \int_t^{t+\d} F(s,\cL_{\cL_{X^{t,\xi}_s|\cG_s}})ds\\
&&= G(\cL_{\cL_{X^{t,\xi}_T|\bar\cG^\e_T}}) + \int_{t}^T F(s,\cL_{\cL_{X^{t,\xi}_s|\bar\cG^\e_s}})ds+\e \le V_0(t, \ul\xi, \cG^0_t)+\e.
\eeaa
Since $\dbG\in \cA(t, \cG^0_t)$ is arbitrary, we have $\tilde V_0(t, \ul\xi, \cG^0_t) \le  V_0(t, \ul\xi,  \cG^0_t)+\e$. Moreover, since $\e>0$ is arbitrary, we obtain $\tilde V_0(t, \ul\xi, \cG^0_t) \le  V_0(t, \ul\xi,  \cG^0_t)$, and thus complete the proof.
\qed

\smallskip
We remark that the above DPP actually does not require the (H*)-hypothesis. We next establish the crucial law invariance property of $V$, which relies on the (H*)-hypothesis. For this purpose, we first need a lemma, whose proof is postponed to Appendix \ref{appendixB}.

\begin{lemma}
\label{lem-invariance}
Let $0\le t=t_0<t_1<\cdots<t_n\le T$, $\xi, \xi'\in \dbL^2(\cF_{t})$, and $\cG^0_t, \cG^{'0}_t\subset \cF_t$ with $\cL_{\cL_{\xi|\cG_{t}}} = \cL_{\cL_{\xi'|\cG^{'0}_{t}}}$. Then, for any $\dbG\in \cA(t, \cG^0_t)$ and $\e>0$, there exists $\dbG^{'\e}\in \cA(t, \cG^{'0}_t)$ such that
\bea
\label{invariance}
\cW_1\Big(\cL_{\cL_{(\xi', B^t_{t_1}, \cdots, B^t_{t_i})|\cG^{'\e}_{t_i}}}, \cL_{\cL_{(\xi, B^t_{t_1}, \cdots, B^t_{t_i})|\cG_{t_i}}}\Big)\le \e, \q i=1,\cdots, n.
\eea
\end{lemma}

\begin{thm}[Law invariance]
\label{thm-conditional-law-invariant}
Let $t\in [0, T]$, $\xi, \xi'\in \dbL^2(\cF_t)$, and $\cG^0_t, \cG^{'0}_t\subset \cF_t$. If $\cL_{\cL_{\xi|\cG_{t}}} = \cL_{\cL_{\xi'|\cG^{'0}_{t}}}$, then $V_0(t, \ul\xi, \cG^0_t) = V_0(t, \ul \xi', \cG^{'0}_t)$. Consequently, we may introduce a function $V$ on $[0, T]\times \cP_2(\cP_2(\dbR^d))$ as follows:
\bea
\label{VcU}
V(t, \cL_{\cL_{\xi|\cG^0_t}} ) = V_0(t, \ul\xi, \cG^0_t).
\eea
\end{thm}
\noindent{We remark that, by Theorem \ref{thm-support}, $V$ is well defined for all $\bmu\in  \cP_2(\cP_2(\dbR^d))$.}

\proof We shall only prove $V_0(t, \ul\xi, \cG^0_t) \le V_0(t, \ul\xi',  \cG^{'0}_t)$. The opposite inequality can be proved similarly. By \eqref{VtG}, it suffices to show that, for any $\dbG\in \cA(t, \cG^0_t)$, 
\bea
\label{invariant-est1}
J(t, \ul\xi, \dbG) \le V_0(t, \ul\xi', \cG^{'0}_t).
\eea

For any $n\ge 1$, set $t_i := t+ {i\over n}(T-t)$, $i=0,\cdots, n$, and consider the standard Euler scheme for $X^{t,\xi}$ on $[t, T]$, but with truncated initial condition:
\bea
\label{Xn}
X^n_{t_0}:= (-n)\vee \xi \wedge n, \q X^n_s := X^n_{t_i} + b(t_i, X^n_{t_i})(s-t_i) + \si(t_i, X^n_{t_i})B^{t_i}_s, ~s\in (t_i, t_{i+1}].
\eea
 for $i=0,\cdots, n-1$. Moreover, for any $R>0$, consider a further truncation:
\bea
\label{XnR}
X^{n,R}_{t_0}:= X^n_{t_0}, \q X^{n,R}_s := X^{n,R}_{t_i} + b(t_i, X^{n,R}_{t_i})(s-t_i) + \si(t_i, X^{n,R}_{t_i})\big[(-R)\vee B^{t_i}_s \wedge R\big].
\eea
Define $X^{'n}$ and $X^{'n,R}$ similarly corresponding to $\xi'$. By standard SDE theory, it is clear that
\bea
\label{Xnconv}
\left.\ba{c}
\displaystyle \lim_{n\to\infty} [\e_n + \e_n'] =0, \q \lim_{R\to\infty} [\e_{n,R} + \e_{n,R}'] =0,\q\mbox{where}\smallskip\\
\displaystyle |\e_n|^2:= \dbE\Big[\sup_{t\le s\le T}|X^n_s - X^{t,\xi}_s|^2\Big], \q |\e_{n,R}|^2 := \dbE\Big[\sup_{t\le s\le T}|X^{n,R}_s - X^{n}_s|^2\Big], 
\ea\right.
\eea
and $\e_n', \e_{n,R}'$ are defined similarly.

Now for any $\e>0$, let $\dbG^{'\e}$ be as in Lemma \ref{lem-invariance}. Fix $n, R$. For each $i=0,\cdots, n-1$, denote $\bmu_i :=  \cL_{\cL_{(\xi, B^t_{t_1}, \cdots, B^t_{t_i})|\cG_{t_i}}}$, $\bmu^{'\e}_i := \cL_{\cL_{(\xi', B^t_{t_1}, \cdots, B^t_{t_i})|\cG^{'\e}_{t_i}}}$. By Lemma \ref{lem-invariance}, $\cW_1(\bmu_i, \bmu_i^{'\e}) \le \e$.
Note that $X^{n,R}_{t_i} = \f_i(\xi,B^t_{t_1},\cdots, B^t_{t_i})$, for some Lipschitz continuous  function $\f_i$ with a Lipschitz constant $C_{n,R}$. We also have $X^{'n,R}_{t_i}=\f_i(\xi',B^t_{t_1},\cdots, B^t_{t_i})$ for the same function $\f_i$. Now for any Lipschitz continuous function $\Phi: \cP_1(\dbR^d)\to \dbR$ with Lipschitz constant $1$ under $W_1$, denote
\beaa
\hat\Phi_i(\mu) := \Phi(\mu\circ \f_i^{-1}),\q  \mu\in \cP(\dbR^{d(i+1)}). 
\eeaa
It is clear that $\hat \Phi_i: \cP_1(\dbR^{d(i+1)})\to \dbR$ is Lipschitz continuous under $W_1$ with Lipschitz constant $C_{n,R}$.
Then
\beaa
&&\dbE\Big[\Phi(\cL_{X^{n,R}_{t_i}|\cG_{t_i}}) -\Phi(\cL_{X^{'n,R}_{t_i}|\cG^{'\e}_{t_i}})\Big] = \dbE\Big[\hat \Phi_i(\cL_{(\xi, B^t_{t_1}, \cdots, B^t_{t_i})|\cG_{t_i}})- \hat \Phi_i(\cL_{(\xi', B^t_{t_1}, \cdots, B^t_{t_i})|\cG^{'\e}_{t_i}})\Big] \\
&&\le C_{n,R} \cW_1(\bmu_i, \bmu_i^{'\e}) \le C_{n,R}\e.
\eeaa
This implies that $\cW_1(\cL_{\cL_{X^{n,R}_{t_i}|\cG_{t_i}}}, \cL_{\cL_{X^{'n,R}_{t_i}|\cG^{'\e}_{t_i}}}) \le C_{n,R}\e$. 
On the other hand, it is obvious that $\cW_3(\cL_{\cL_{X^{n,R}_{t_i}|\cG_{t_i}}}, \cL_{\cL_{X^{'n,R}_{t_i}|\cG^{'\e}_{t_i}}}) \le C_{n,R}$, for a possibly larger $C_{n,R}$. Then
\bea
\label{Xnconv2}
\cW_2(\cL_{\cL_{X^{n,R}_{t_i}|\cG_{t_i}}}, \cL_{\cL_{X^{'n,R}_{t_i}|\cG^{'\e}_{t_i}}}) \le C_{n,R}\sqrt{\e}.
\eea

Introduce 
\beaa
&J_{n,R}(t, \ul\xi, \dbG) := G(\cL_{\cL_{X^{n,R}_{t_n}|\cG_{t_n}}}) + \sum_{i=0}^{n-1} \int_{t_i}^{t_{i+1}}F(s,\cL_{\cL_{X^{n,R}_{t_i}|\cG_{t_i}}})ds;\\
&J_{n}(t, \ul\xi, \dbG) := G(\cL_{\cL_{X^{n}_{t_n}|\cG_{t_n}}}) + \sum_{i=0}^{n-1} \int_{t_i}^{t_{i+1}}F(s,\cL_{\cL_{X^{n}_{t_i}|\cG_{t_i}}})ds;\\
&\tilde J_{n}(t, \ul\xi, \dbG) := G(\cL_{\cL_{X_{t_n}|\cG_{t_n}}}) + \sum_{i=0}^{n-1} \int_{t_i}^{t_{i+1}}F(s,\cL_{\cL_{X_{t_i}|\cG_{t_i}}})ds.
\eeaa
By \eqref{Xnconv2} and the uniform continuity of $F$ and $G$, we have
\bea
\label{Xnconv3} 
\lim_{\e\to 0} \big|J_{n,R}(t, \ul\xi, \dbG) - J_{n,R}(t, \ul\xi', \dbG^{'\e})\big| = 0.
\eea
Let $\rho$ denote the modulus of continuity function of $F$, $G$ with respect to $\bmu$. By \eqref{Xnconv} we have
\bea
\label{Xnconv4} 
\big|J_{n,R}(t, \ul\xi, \dbG) - J_{n}(t, \ul\xi, \dbG)\big| \le C\rho(\e_{n,R}),\q 
\big|J_{n}(t, \ul\xi, \dbG) - \tilde J_{n}(t, \ul\xi, \dbG)\big| \le C\rho(\e_{n}).
\eea
Similar uniform estimates hold for $J_{n,R}(t, \ul\xi', \dbG^{'\e})$, $J_{n}(t, \ul\xi', \dbG^{'\e})$, and $\tilde J_{n}(t, \ul\xi', \dbG^{'\e})$.
Then, by sending $\e\to 0$ and $R\to \infty$, it follows from \eqref{Xnconv},  \eqref{Xnconv3}, and \eqref{Xnconv4} that
\bea
\label{Xnconv5} 
 \limsup_{\e\to 0} \big|\tilde J_{n}(t, \ul\xi, \dbG) - \tilde J_{n}(t, \ul\xi', \dbG^{'\e})\big| \le C\rho(\e_n).
\eea
Moreover, for $s\in [t_i, t_{i+1}]$, 
\beaa
\cW^2_2(\cL_{\cL_{X_{s}|\cG_{s}}}, \cL_{\cL_{X_{t_i}|\cG_{t_i}}}) = \cW^2_2(\cL_{\cL_{X_{s}|\cG_{s}}}, \cL_{\cL_{X_{t_i}|\cG_{s}}}) \le \dbE[|X_s - X_{t_i}|^2] \le {C\over n}.
\eeaa
Then, 
\beaa
\big|\tilde J_{n}(t, \ul\xi, \dbG) - J(t, \ul\xi, \dbG)\big| \le \sum_{i=0}^{n-1} \int_{t_i}^{t_{i+1}}\big| F(s,\cL_{\cL_{X_{s}|\cG_{s}}})-F(s,\cL_{\cL_{X_{t_i}|\cG_{t_i}}})\big|ds \le \rho({C\over \sqrt{n}}).
\eeaa
Similarly  $\big|\tilde J_{n}(t, \ul\xi', \dbG^{'\e})-  J(t, \ul\xi', \dbG^{'\e})\big| \le \rho({C\over \sqrt{n}})$. Then, together with \eqref{Xnconv5}, we obtain
\beaa
\limsup_{\e\to 0} \big|J(t, \ul\xi, \dbG) - J(t, \ul\xi', \dbG^{'\e})\big| \le C\rho(\e_n) + C\rho({C\over \sqrt{n}}).
\eeaa
Send $n\to \infty$, this clearly implies \eqref{invariant-est1}, and hence completes the proof.
\qed

We conclude this section with the following regularity result of the value function $V$.
\begin{thm}
\label{reg-VcU}
Besides Asumption \ref{assum-standing}, assume further that $F$ and $G$ are uniformly Lipschitz continuous in $\bmu$ under $\cW_2$. Then there exists a constant $C>0$, depending only on $T, d$, the Lipschitz constant of $b, \si$ in $x$, the Lipschitz constant of $F, G$ in $\bmu$, and  $\sup_t [|b(t, 0)|+|\si(t,0)| + |F(t, \d_{\d_0})|] + |G(\d_{\d_0})|$, such that
\bea
\label{Vreg}
\left.\ba{c}
|V(t, \bmu_1) - V(t, \bmu_2)| \le C\cW_2(\bmu_1, \bmu_2),\q \forall t\in [0, T], \bmu_1, \bmu_2\in \cP_2(\cP_2(\dbR^d));\smallskip\\
|V(t_1, \bmu) - V(t_2, \bmu)| \le C[1+\|\bmu\|_2] \sqrt{|t_1-t_2|},\q \forall t_1, t_2\in [0, T], \bmu\in \cP_2(\cP_2(\dbR^d)).
\ea\right.
\eea
\end{thm}
\proof We first prove the regularity in $\bmu$. Given $t, \bmu_1, \bmu_2$, by Theorem \ref{thm-coupling},
there exist $\xi_1, \xi_2\in \dbL^2(\cF_t)$ and $\cG^0_t = \si(Y)\subset \cF_t$ such that $\cL_{\cL_{\xi_i|\cG^0_t}}=\bmu_i$, $i=1,2$, and $\dbE[|\xi_1-\xi_2|^2] = \cW_2^2(\bmu_1, \bmu_2)$. For any $\dbG\in \cA(t, \cG^0_t)$, by the Lipschitz continuity of $b, \si$ and  standard SDE estimates, we have 
\beaa
\sup_{t\le s\le T} \cW_2^2(\cL_{\cL_{X^{t, \xi_1}_s|\cG_s}}, \cL_{\cL_{X^{t, \xi_2}_s|\cG_s}}) &\le& \sup_{t\le s\le T}\dbE\big[|X^{t, \xi_1}_s - X^{t, \xi_2}_s|^2\big]\\
&\le& C\dbE[|\xi_1-\xi_2|^2] = C\cW_2^2(\bmu_1, \bmu_2).
\eeaa  
Then, by the Lipchitz continuity of $F$ and $G$, we have
\beaa
|V(t, \bmu_1) - V(t, \bmu_2)| \le C\sup_{t\le s\le T} \cW_2(\cL_{\cL_{X^{t, \xi_1}_s|\cG_s}}, \cL_{\cL_{X^{t, \xi_2}_s|\cG_s}})  \le C\cW_2(\bmu_1, \bmu_2).
\eeaa

We next prove the regularity in $t$. Assume $t_1 < t_2$ and let $\bmu=\cL_{\cL_{\xi|\cG_{t_1}^0}}$ for some $\xi\in \dbL^2(\cF_{t_1})$ and $\cG^0_{t_1} \subset \cF_{t_1}$.  By the (H*)-hypothesis,  $\bmu=\cL_{\cL_{\xi|\cG_{t_2}}}$ for any $\dbG\in \cA(t_1, \cG^0_{t_1})$. Then, by the Dynamic Programming Principle \eqref{DPP} and the above Lipschitz continuity in $\bmu$, it follows from standard SDE estimates that
\beaa
&&|V(t_1,\bmu)-V(t_2,\bmu)|\\
&&= \Big|\sup_{\dbG\in\cA(t_1,\cG_{t_1}^0)}\Big[\int_{t_1}^{t_2} F(s,\cL_{\cL_{X^{t_1,\xi}_s|\cG_s}})ds+V(t_2,\cL_{\cL_{X^{t_1,\xi}_{t_2}|\cG_{t_2}}})-V(t_2,\cL_{\cL_{\xi|\cG^0_{t_1}}})\Big]\Big|\\
&&\le C|t_2-t_1| \sup_{t_1\le s\le t_2}\Big[1+\|\cL_{\cL_{X^{t_1,\xi}_s|\cG_s}}\|_2\Big] + C\cW_2\big(\cL_{\cL_{X^{t_1,\xi}_{t_2}|\cG_{t_2}}}, \cL_{\cL_{\xi|\cG_{t_2}}}\big)\\
&&\le C|t_2-t_1| \sup_{t_1\le s\le t_2}\Big(1+ \dbE\big[ |X^{t_1,\xi}_s|^2\big]\Big)^{1\over 2} + C\Big(\dbE\big[ |X^{t_1,\xi}_{t_2}-\xi|^2\big]\Big)^{1\over 2}\\
&&\le C|t_2-t_1| \sup_{t_1\le s\le t_2}\Big(1+ \dbE[ |\xi|^2]\Big)^{1\over 2} + C\Big([1+\dbE[\xi|^2] (t_2-t_1)\Big)^{1\over 2}\\
&&\le C[1+\|\xi\|_2] \sqrt{t_2-t_1} = C[1+\|\bmu\|_2]\sqrt{t_2-t_1}.
\eeaa
This completes the proof.
\qed

\section{Stochastic calculus on $\mathcal P_2(\mathcal P_2(\mathbb R^d))$}
\label{sect-Ito}
In this section, we develop the stochastic calculus on the space $\mathcal P_2(\mathcal P_2(\mathbb R^d))$, in particular the It\^o formula, which plays a key role in deriving the PDE for the value function $V$ of the information control problem  \eqref{VtG}-\eqref{VcU}. We equip $\mathcal P_2(\mathcal P_2(\mathbb R^d))$ with $\cW_2$ in \eqref{W2}, and  recall the linear functional derivatives for a function $u: \cP_2(\dbR^d)\to \dbR$, see e.g. \cite[Section 5.4]{CD1} and \cite{cox2024controlled}.

\begin{definition}
\label{defn-uderivative}
Let  $u: \mathcal P_2(\mathbb R^d)\rightarrow \mathbb{R}$. We call a continuous function ${\d\over \d\mu} u: \mathcal P_2(\mathbb R^d)\times \mathbb{R}^d \to \dbR$  the linear functional derivative of $u$ if $\big|\frac{\delta}{\delta \mu}u(\mu, x)\big| \leq C\big(1+| x|^2\big)$ for some constant $C$, which may depend on $u$ but is uniform on $\mu$, and for any $ \mu_1, \mu_2\in \cP_2(\dbR^d)$,
\bea
\label{linearu}
 u(\mu_2)-u(\mu_1)=\int_0^1 \int_{\mathbb{R}^d} \frac{\delta}{\delta \mu}u( \theta\mu_2+(1-\theta) \mu_1,x)(\mu_2-\mu_1)(\mathrm{d}x) \mathrm{d}\theta.
 \eea
\end{definition}

\begin{remark}
\label{rem-uderivative}
Assume further that, for any $x\in \dbR^d$, the function ${\d\over \d\mu}u(\cdot,  x)$ has the linear functional derivative, denoted as ${\d^2\over \d\mu^2}u(\mu, x, \tilde x) := {\d\over \d\mu}({\d\over \d\mu} u(\cdot, x))(\mu, \tilde x)$. We call ${\d^2\over \d\mu^2}u: \mathcal P_2(\mathbb R^d)\times \mathbb{R}^d\times \dbR^d \to \dbR$  the second order linear functional derivative of $u$ if it is jointly continuous and satisfies $|{\d^2\over \d\mu^2}u(\mu, x, \tilde x)|\le C[1+|x|^2+|\tilde x|^2]$, for some constant $C$ which may depend on $u$ but is uniform on $\mu$. We note that  ${\d^2\over \d\mu^2}u$ is symmetric in $(x,\tilde x)$: ${\d^2\over \d\mu^2}u(\mu, x, \tilde x) = {\d^2\over \d\mu^2}u(\mu, \tilde x, x)$.
\end{remark}

We next extend the above derivative to functions $U:\mathcal P_2(\mathcal P_2(\mathbb R^d))\to\mathbb R$.

\begin{definition}
\label{defn-Uderivative}
Let  $U: \mathcal P_2(\cP_2(\mathbb R^d))\rightarrow \mathbb{R}$. We call a continuous function ${\d\over \d\bmu} U: \cP_2(\cP_2(\dbR^d))\times \cP_2(\dbR^d) \to \dbR$  the linear functional derivative of $U$ if $\big|\frac{\delta}{\delta \bmu}U(\bmu, \mu)\big| \leq C\big(1+\|\mu\|_2^2\big)$ for some constant $C$, which may depend on $U$ but is uniform on $\bmu$, and for any $ \bmu_1, \bmu_2\in \cP_2(\dbR^d)$,
\bea
\label{linearU}
 U(\bmu_2)-U(\bmu_1)=\int_0^1 \int_{\cP_2(\mathbb{R}^d)} \frac{\delta}{\delta \bmu}U( \theta\bmu_2+(1-\theta) \bmu_1,\mu)(\bmu_2-\bmu_1)(\mathrm{d}\mu) \mathrm{d}\theta.
 \eea
\end{definition}

\noindent Moreover, for a function $U: [0, T]\times \mathcal P_2(\cP_2(\mathbb R^d))\rightarrow \mathbb{R}$, we define $\partial_t U$ in the obvious sense. 

 To prepare for the It\^{o} formula, we introduce
 \bea
 \label{BG}
 B^\dbG_t:= \dbE[B_t|\cG_t].
 \eea

\begin{lemma}
\label{lem-BG}
Let $\dbG$ satisfy the (H)-hypothesis (not necessarily (H*)-hypothesis). Then $B^\dbG$ is a $\dbG$-martingale and hence an $\dbF$-martingale. Moreover, $B^\dbG_t = \int_0^t \si^\dbG_s dB_s$ where $\si^\dbG(\si^\dbG)^\top$ is $\dbG$-adapted with $0\le \si^\dbG(\si^\dbG)^\top \le I_{d\times d}$.
\end{lemma} 
\proof First, since $B^t_T$ is independent of $\cG_t\subset \cF_t$, $B^\dbG_t = \dbE[B_T|\cG_t]$ is obviously a $\dbG$-martingale. Then by the (H)-hypothesis it is also an $\dbF$-martingale, and thus $B^\dbG_t = \int_0^t \si^\dbG_s dB_s$ for some $\si^\dbG$. Since $B^\dbG$ is $\dbG$-adapted, as the quadratic variation it is clear that $\si^\dbG(\si^\dbG)^\top$ is also  $\dbG$-adapted.

Next, fix $0\le t<t+\d\le T$, for any $\eta\in \dbL^2(\cG_t, \dbR^d)$, we have
\beaa
&&\displaystyle\dbE\Big[ \int_t^{t+\d} |\eta^\top\si^\dbG_s|^2 ds \Big|\cG_t\Big] = \dbE\Big[\big|\int_t^{t+\d} \eta\si^\dbG_s dB_s\big|^2\Big|\cG_t\Big] = \dbE\Big[|\eta\cdot B^\dbG_{t+\d} - \eta\cdot B^\dbG_t|^2\Big|\cG_t\Big] \\
&&\displaystyle = \dbE\Big[\big|\dbE[\eta \cdot B^t_{t+\d}|\cG_{t+\d}] \big|^2\Big|\cG_t\Big] \le \dbE\Big[\big|\eta \cdot B^t_{t+\d}\big|^2\Big|\cG_t \Big] =|\eta|^2\d.
\eeaa
This implies that $0\le \si^\dbG(\si^\dbG)^\top \le I_{d\times d}$, $dt\times d\dbP$-a.s.
\qed

We are now ready to establish the It\^{o} formula. 
\begin{thm}
\label{thm-Ito}
Let $\dbG\in \cA_0$, and $\cX_t = \cX_0 + \int_0^t \a_s ds + \int_0^t \b_s dB_s$, where $\cX_0\in \dbL^2(\cF_0)$ and $\a, \b\in \dbL^2(\dbF)$. Let $U: [0, T]\times \mathcal P_2(\cP_2(\mathbb R^d))\rightarrow \mathbb{R}$ be such that all the following derivatives, as well as its lower order derivatives,  exist and are locally uniformly continuous in the sense of Assumption \ref{assum-standing} (ii): 
\bea
\label{derivatives}
\left.\ba{c}
\displaystyle\partial_t V(t, \bmu),\q  \partial_{xx} {\d\over \d\mu}{\d\over \d\bmu} V(t, \bmu, \mu, x),\q  \partial_{(x,\tilde x)} ^{(2)}{\d^2\over \d\mu^2}{\d\over \d\bmu} V(t, \bmu, \mu, x, \tilde x).
\ea\right.
\eea
Here $ \partial^{(2)}_{(x,\tilde x)}$ denotes all the second order derivatives with respect to $x$ and $\tilde x$. Then the following It\^{o} formula holds true, with $:$ denoting the Frobenius inner product:

\begin{equation}
\label{Ito}
\begin{aligned}
\frac{d}{dt} V(t,\mathcal L_{\mathcal L_{\cX_{t}|\mathcal G_{t}}})
&=\partial_t V(t,\mathcal L_{\mathcal L_{\cX_t|\mathcal G_t}})+\mathbb E\Big[\partial_x\frac{\delta}{\delta\mu}\frac{\delta}{\delta \bmu} V\Big(t,\mathcal L_{\mathcal L_{\cX_{t}|\mathcal G_{t}}}, \mathcal L_{\cX_{t}|\mathcal G_{t}},X_{t}\Big) \cdot \a_t\Big]\\
 &\quad+\mathbb E\Big[\frac{1}{2}\partial_{xx}\frac{\delta}{\delta\mu}\frac{\delta}{\delta\bmu}V\Big(t,\mathcal L_{\mathcal L_{\cX_{t}|\mathcal G_{t}}}, \mathcal L_{\cX_{t}|\mathcal G_{t}},X_{t}\Big): \b_t\b^\top_t\Big]\\
 &\quad+\mathbb E\Big[\frac{1}{2}\partial_{x\tilde x}\frac{\delta^2}{\delta\mu^2}\frac{\delta}{\delta\bmu}V\Big(t,\mathcal L_{\mathcal L_{\cX_{t}|\mathcal G_{t}}}, \mathcal L_{\cX_{t}|\mathcal G_{t}},  \cX_{t},\widetilde \cX_t\Big): \b_t \si^\dbG_t(\si_t^{\mathbb G})^\top\widetilde \b_t^\top\Big],
\end{aligned}
\end{equation}
where $(\tilde \b, \tilde \cX)$ is a conditionally independent copy of $(\b, \cX)$, conditional on $\mathcal G_t$. 
\end{thm}

\proof It suffices to verify $V(T, \mathcal L_{\mathcal L_{\cX_{T}|\mathcal G_{T}}})- V(0, \mathcal L_{\mathcal L_{\cX_{0}|\mathcal G_{0}}})$. For this purpose, fix $n$ and consider the uniform time partition: $t_i:= i\D t$, $i=0,\cdots, n$, with $\D t :=  {T\over n}$. Denote
\beaa
\mu_t:= \cL_{\cX_{t}|\mathcal G_{t}}, \q \bmu_t:= \mathcal L_{\mu_t},\q \mu^\th_i:= \th\mu_{t_{i+1}}+(1-\th)\mu_{t_i},\q \bmu^\th_i := \th \bmu_{t_{i+1}} + (1-\th) \bmu_{t_i}.
\eeaa
Then
\beaa
V(T, \mathcal L_{\mathcal L_{\cX_{T}|\mathcal G_{T}}})- V(0, \mathcal L_{\mathcal L_{\cX_{0}|\mathcal G_{0}}}) = \sum_{i=0}^{n-1} \Big[V(t_{i+1}, \bmu_{t_{i+1}})- V(t_{i}, \bmu_{t_{i}})\Big] = I^n_1 + I^n_2,
\eeaa
where, 
\beaa
I^n_1 &:=&  \sum_{i=0}^{n-1} \Big[V(t_{i+1}, \bmu_{t_{i+1}})- V(t_{i}, \bmu_{t_{i+1}})\Big] =\sum_{i=0}^{n-1} \int_0^1 \partial_t V(t_i + \th \D t, \bmu_{t_{i+1}}) d\th \D t;\\
I^n_2 &:=& \sum_{i=0}^{n-1} \Big[V(t_{i}, \bmu_{t_{i+1}})- V(t_{i}, \bmu_{t_i})\Big] =  \sum_{i=0}^{n-1} \int_0^1 \int_{\cP_2(\dbR^d)} {\d\over \d \bmu}V\big(t_{i},  \bmu^\th_{i}, \mu\big)\big( \bmu_{t_{i+1}} - \bmu_{t_i}\big)(d\mu) d\th\\
&=&  \sum_{i=0}^{n-1}\int_0^1  \dbE\Big[{\d\over \d \bmu}V\big(t_{i},\bmu^\th_{i},  \mu_{t_{i+1}}\big)- {\d\over \d \bmu}V\big(t_{i}, \bmu^\th_{i},  \mu_{t_{i}}\big)\Big] d\th\\
&=&  \sum_{i=0}^{n-1}\int_0^1\int_0^1 \dbE\Big[\int_{\dbR^d} {\d\over \d\mu}{\d\over \d \bmu}V\big(t_{i},\bmu^\th_{i},  \mu^{\th'}_i, x\big) (\mu_{t_{i+1}} -\mu_{t_i})(dx)\Big] d\th'd\th.
\eeaa
We shall decompose $I^n_2$ further. Denote
\beaa
I^n_{2,1} &:=& \sum_{i=0}^{n-1}\int_0^1 \dbE\Big[\int_{\dbR^d} {\d\over \d\mu}{\d\over \d \bmu}V\big(t_{i},\bmu^\th_{i}, \mu_{t_i}, x\big) (\mu_{t_{i+1}} -\mu_{t_i})(dx)\Big] d\th\\
&=&\sum_{i=0}^{n-1}\int_0^1 \dbE\Big[ \dbE\big[{\d\over \d\mu}{\d\over \d \bmu}V\big(t_{i},\bmu^\th_{i},  \mu_{t_i}, \cX_{t_{i+1}}\big)\big|\cG_{t_{i+1}}\big] \\
&&\qquad\qquad -\dbE\big[{\d\over \d\mu}{\d\over \d \bmu}V\big(t_{i},\bmu^\th_{i},  \mu_{t_i}, \cX_{t_{i}}\big)\big|\cG_{t_i}\big]\Big] d\th\\
&=&\sum_{i=0}^{n-1}\int_0^1 \dbE\Big[{\d\over \d\mu}{\d\over \d \bmu}V\big(t_{i},\bmu^\th_{i},  \mu_{t_i}, \cX_{t_{i+1}}\big)-{\d\over \d\mu}{\d\over \d \bmu}V\big(t_{i},\bmu^\th_{i},  \mu_{t_i}, \cX_{t_{i}}\big)\Big] d\th.
\eeaa
Here we used the fact that $\mu_{t_i}$ is $\cG_{t_i}\subset \cG_{t_{i+1}}$-measurable. Then, by applying the standard It\^{o} formula on $[t_i, t_{i+1}]$ we have
\beaa
&&I^n_{2,1} =\sum_{i=0}^{n-1}\int_0^1 \int_{t_i}^{t_{i+1}} \dbE\Big[\partial_x{\d\over \d\mu}{\d\over \d \bmu}V\big(t_{i},\bmu^\th_{i},  \mu_{t_i}, \cX_t\big) \cdot \a_t \\
&&\qquad\qquad\qquad + {1\over 2}\partial_{xx}{\d\over \d\mu}{\d\over \d \bmu}V\big(t_{i},\bmu^\th_{i},  \mu_{t_i}, \cX_t\big) : \b_t\b_t^\top\Big]dt d\th.
\eeaa
Next, denote $\D \mu_{t_{i+1}} := \mu_{t_{i+1}}-\mu_{t_i}$, and
\beaa
&& I^n_{2,2} := I^n_2 - I^n_{2,1}\\
&&=\sum_{i=0}^{n-1}\int_0^1\!\!\!\int_0^1\!\!\! \dbE\Big[\int_{\dbR^d}\big[ {\d\over \d\mu}{\d\over \d \bmu}V\big(t_{i},\bmu^\th_{i}, ~ \mu^{\th'}_i, x\big) - {\d\over \d\mu}{\d\over \d \bmu}V\big(t_{i},\bmu^\th_{i}, ~ \mu_{t_i}, x\big)\big] \D\mu_{t_{i+1}}(dx)\Big] d\th'd\th\\
&&=\sum_{i=0}^{n-1}\int_0^1\!\!\!\int_0^1\!\!\!\int_0^{\th'}\!\!\! \dbE\Big[\int_{\dbR^d\times \dbR^d} {\d^2\over \d\mu^2}{\d\over \d \bmu}V\big(t_{i},\bmu^\th_{i}, ~ \mu^{\tilde\th}_i, x, \tilde x\big) \D\mu_{t_{i+1}}(d\tilde x)\D\mu_{t_{i+1}} (dx)\Big] d\tilde \th d\th'd\th\\
\eeaa

We now analyze the convergence when sending $n\to \infty$. Note that
\bea
\label{ItoXest}
R_0^2:= \dbE\Big[\sup_{0\le t\le T} |\cX_s|^2\Big] <\infty.
\eea
Here and in the sequel, for given $R_0$, we use $o(1)$ to denote terms which converge to $0$ as $n\to \infty$, uniformly in $i$ or $t$. Similarly we use the notation $o(\D t)$.  Note that, by the (H)-hypothesis, $\cL_{\cX_t |\cG_t} = \cL_{\cX_t |\cG_{t+\d}}$. Then, for $\d\le \D t = {T\over n}$,
\beaa
\cW_2^2(\bmu_{t+\d}, \bmu_t) &\le& \dbE\Big[W_2^2(\mu_{t+\d}, \mu_t)\Big] = \dbE\Big[W_2^2\big(\cL_{\cX_{t+\d} |\cG_{t+\d}}, \cL_{\cX_t |\cG_{t+\d}}\big)\Big]\\
& \le& \dbE \Big[ \dbE\big[\big|\cX_{t+\d} - \cX_{t}\big|^2\big| \cG_{t+\d}\big]\Big] =\dbE\Big[\big|\cX_{t+\d} - \cX_{t}\big|^2\Big] = o(1).
\eeaa
Note that the local uniform continuity of the derivatives in \eqref{derivatives} implies their boundedness and uniform continuity in $\cK_{R_0}$. Then we have, for $n\to \infty$,
\beaa
&&\!\!\!\!\!\!\! I^n_1=\sum_{i=0}^{n-1}  \partial_t V(t_i, \bmu_{t_{i+1}}) \D t + o(1) =  \sum_{i=0}^{n-1} \int_{t_i}^{t_{i+1}} \partial_t V(t, \bmu_{t})dt + o(1) = \int_0^T \partial_t V(t, \bmu_{t})dt + o(1);\\
 &&\!\!\!\!\!\!\! I^n_{2,1} =\sum_{i=0}^{n-1}\int_{t_i}^{t_{i+1}}\dbE\Big[\partial_x{\d\over \d\mu}{\d\over \d \bmu}V\big(t,\bmu_t,  \mu_t, \cX_t\big) \cdot \a_t + {1\over 2}\partial_{xx}{\d\over \d\mu}{\d\over \d \bmu}V\big(t,\bmu_t,  \mu_{t}, \cX_t\big) : \b_t\b_t^\top\Big]dt\\
&&\!\!\!\!\!\!\!\q +\int_0^1\int_0^T\dbE\Big[\Big(\sum_{i=0}^{n-1} \partial_x{\d\over \d\mu}{\d\over \d \bmu}V\big(t,\bmu_i^\th,  \mu_{t_i}, \cX_t\big)1_{t\in[t_i,t_{i+1})}-\partial_x{\d\over \d\mu}{\d\over \d \bmu}V\big(t,\bmu_t,  \mu_t, \cX_t\big) \Big)\cdot \a_t\Big]dtd\th\\
&&\!\!\!\!\!\!\! \q+\int_0^1\int_0^T\dbE\Big[\Big(\sum_{i=0}^{n-1} {1\over 2}\partial_{xx}{\d\over \d\mu}{\d\over \d \bmu}V\big(t,\bmu_i^\th,  \mu_{t_i}, \cX_t\big)1_{t\in[t_i,t_{i+1})}\\
&&\qquad\qquad\qquad - {1\over 2}\partial_{xx}{\d\over \d\mu}{\d\over \d \bmu}V\big(t,\bmu_t,  \mu_{t}, \cX_t\big) \Big): \b_t\b_t^\top\Big]dtd\th\\
&&\!\!\!\!\!\!\! =\sum_{i=0}^{n-1}\int_{t_i}^{t_{i+1}}\dbE\Big[\partial_x{\d\over \d\mu}{\d\over \d \bmu}V\big(t,\bmu_t,  \mu_t, \cX_t\big) \cdot \a_t + {1\over 2}\partial_{xx}{\d\over \d\mu}{\d\over \d \bmu}V\big(t,\bmu_t,  \mu_{t}, \cX_t\big) : \b_t\b_t^\top\Big]dt+o(1)\\
&&\!\!\!\!\!\!\! = \int_0^T\dbE\Big[\partial_x{\d\over \d\mu}{\d\over \d \bmu}V\big(t,\bmu_t,  \mu_t, \cX_t\big) \a_t + {1\over 2}\partial_{xx}{\d\over \d\mu}{\d\over \d \bmu}V\big(t,\bmu_t,  \mu_{t}, \cX_t\big) : \b_t\b_t^\top\Big]dt+o(1).
\eeaa

It remains to analyze $I^n_{2,2}$.  Similarly decompose it:
\bea
\label{In22}
I^n_{2,2} &=& \sum_{i=0}^{n-1}\dbE\Big[{1\over 2}I^n_{2,2,i} + R_{2,2,i}^n\Big],\quad \mbox{where}\\
I^n_{2,2,i} &:=&  \int_{\dbR^d\times \dbR^d} {\d^2\over \d\mu^2}{\d\over \d \bmu}V\big(t_{i},\bmu_{t_i}, ~ \mu_{t_i}, x, \tilde x\big) \D\mu_{t_{i+1}}(d\tilde x)\D\mu_{t_{i+1}} (dx)\Big],\nonumber\\
R_{2,2,i}^n &:=&\int_0^1\int_0^1\int_0^{\th'} \int_{\dbR^d\times \dbR^d}\Big( {\d^2\over \d\mu^2}{\d\over \d \bmu}V\big(t_{i},\bmu^\th_{i}, ~ \mu^{\tilde\th}_i, x, \tilde x\big)\nonumber\\
&&  -{\d^2\over \d\mu^2}{\d\over \d \bmu}V\big(t_{i},\bmu_{t_i}, ~ \mu_{t_i}, x, \tilde x\big) \Big)\D\mu_{t_{i+1}} (d\tilde x)\D\mu_{t_{i+1}}(dx) d\tilde \th d\th'd\th.\nonumber
\eea
Denote $v_i(x, \tilde x) := {\d^2\over \d\mu^2}{\d\over \d \bmu}V\big(t_{i},\bmu_{t_i},  \mu_{t_i}, x, \tilde x\big)$, which is $\cG_{t_i}$-measurable, we have
\beaa
 I^n_{2,2,i} &:=&\int_{\dbR^d\times \dbR^d} v_i(x, \tilde x\big) (\mu_{t_{i+1}} -\mu_{t_i})(d\tilde x)(\mu_{t_{i+1}} -\mu_{t_i})(dx)\\
 &=& \int_{\dbR^d\times \dbR^d} v_i(x, \tilde x\big) (\cL_{\cX_{t_{i+1}}|\cG_{t_{i+1}}} -\cL_{\cX_{t_{i}}|\cG_{t_{i}}})(d\tilde x)(\cL_{\cX_{t_{i+1}}|\cG_{t_{i+1}}} -\cL_{\cX_{t_{i}}|\cG_{t_{i}}})(dx)\\
 &=& \dbE\Big[v_i\big(\cX_{t_{i+1}}, \tilde\cX_{t_{i+1}}\big) - v_i\big(\cX_{t_{i+1}}, \tilde\cX_{t_{i}}\big) - v_i\big(\cX_{t_{i}}, \tilde\cX_{t_{i+1}}\big) + v_i\big(\cX_{t_{i}}, \tilde\cX_{t_{i}}\big)\big|\cG_{t_{i+1}} \Big].
 \eeaa
 Here $\tilde\cX$ is an conditionally independent copy of $\cX$, conditional on $\dbG$, and we used the fact again that   $\cL_{\cX_t |\cG_t} = \cL_{\cX_t |\cG_{t+\d}}$. Denote $\D \cX_{i+1} := \cX_{t_{i+1}} - \cX_{t_{i}}$, $\D \tilde\cX_{i+1} := \tilde\cX_{t_{i+1}} - \tilde \cX_{t_{i}}$. Now apply the standard Taylor expansion, we have
\bea
\label{In22i}
\dbE[ I^n_{2,2,i}\big] &=& \dbE\Big[ \partial_{x\tilde x} v_i\big(\cX_{t_{i}}, \tilde\cX_{t_{i}}\big) : \D \cX_{i+1} \D \tilde\cX_{i+1}^\top \Big] + o(\D t)\nonumber\\
&=& \dbE\Big[\partial_{x\tilde x} v_i\big(\cX_{t_{i}}, \tilde\cX_{t_{i}}\big) :  (\b_{t_i} B^{t_i}_{t_{i+1}}) \big(\tilde \b_{t_i} \tilde B^{t_i}_{t_{i+1}})\big)^{\top}\Big] + o(\D t)\\
 &=& \dbE\Big[\partial_{x\tilde x} v_i\big(\cX_{t_{i}}, \tilde\cX_{t_{i}}\big) :  \b_{t_i} \dbE\big[B^{t_i}_{t_{i+1}} (\tilde B^{t_i}_{t_{i+1}})^{\top}\big|\cF_{t_i}\vee \tilde \cF_{t_i} \vee \cG_{t_{i+1}}\big] \tilde \b_{t_i}^\top\Big] + o(\D t),\nonumber
 \eea
where $\tilde \cF_0$ is constructed in an enlarged space in an appropriate way such that it consists of conditionally independent copies of events in $\cF_0$, conditional on $\cG_0$, and $\tilde \dbF = \tilde \cF_0\vee \dbF^{\tilde B}$. In particular, $\tilde B$ and $\tilde \b$ are $\tilde \dbF$-progressively measurable.

 We claim that 
 \bea
 \label{Ito-claim}
 \dbE\big[B^{t_i}_{t_{i+1}} \big|\cF_{t_i}\vee \tilde \cF_{t_i}\vee \cG_{t_{i+1}}\big] = \dbE\big[B^{t_i}_{t_{i+1}} \big|\cG_{t_{i+1}}\big].
 \eea
 Indeed, fix an arbitrary $\eta\in \dbL^\infty(\cF_{t_i}\vee \tilde \cF_{t_i}\vee \cG_{t_{i+1}}; \dbP)$. Let $\dbP^\o$ denote the conditional probability distribution of $\dbP$, conditional on $\cG_{t_i}$ (cf. \cite{Stroock-Varadhan}).  By the (H)-hypothesis and the definition of $\tilde\dbF$,  we see that $\cF_{t_i}\vee \tilde \cF_{t_i}$ and $\cG_{t_{i+1}}$ are  independent under $\dbP^\o$, for $\dbP$-a.e. $\o$.  
Moreover, by the (H*)-hypothesis, $\cG_{t_{i+1}}\subset\cG_{t_i}\vee \cF^{B^{t_i}}_{t_{i+1}}$, a.s., then, for $\dbP$-a.e. $\o$, $\cG_{t_{i+1}}\subset \cF^{B^{t_i}}_{t_{i+1}}$, $\dbP^\o$-a.s. In particular, this implies that $\eta$ and $B^{t_i}_{t_{i+1}}$ are conditionally independent under $\dbP^\o$, conditional on $\cG_{t_{i+1}}$.  Then, for $\dbP$-a.e. $\o$,
 \beaa
 \dbE^{\dbP^\o}\big[B^{t_i}_{t_{i+1}} \eta\big|\cG_{t_{i+1}}\big] = \dbE^{\dbP^\o}\big[B^{t_i}_{t_{i+1}} \big|\cG_{t_{i+1}}\big]\dbE^{\dbP^\o}\big[\eta\big|\cG_{t_{i+1}}\big].
 \eeaa
Thus
\beaa
\dbE^\dbP\big[B^{t_i}_{t_{i+1}} \eta \big|\cG_{t_i}\big] = \dbE^\dbP\Big[\dbE^\dbP\big[B^{t_i}_{t_{i+1}} \eta \big|\cG_{t_{i+1}}\big]\Big|\cG_{t_i}\Big] = \dbE^\dbP\Big[\dbE^{\dbP^\o}\big[B^{t_i}_{t_{i+1}} \big|\cG_{t_{i+1}}\big]\dbE^{\dbP^\o}\big[\eta\big|\cG_{t_{i+1}}\big]\Big|\cG_{t_i}\Big],
\eeaa
and therefore,
\beaa
&&\dbE^\dbP\Big[\dbE\big[B^{t_i}_{t_{i+1}} \big|\cF_{t_i}\vee \tilde \cF_{t_i}\vee \cG_{t_{i+1}}\big] ~ \eta\Big]= \dbE^\dbP\big[B^{t_i}_{t_{i+1}}\eta\big] = \dbE^\dbP\Big[\dbE^\dbP\big[B^{t_i}_{t_{i+1}} \eta \big|\cG_{t_i}\big] \Big]\\
&&=\dbE^\dbP\Big[\dbE^\dbP\Big[\dbE^{\dbP}\big[B^{t_i}_{t_{i+1}} \big|\cG_{t_{i+1}}\big]\dbE^{\dbP}\big[\eta\big|\cG_{t_{i+1}}\big]\Big|\cG_{t_i}\Big] \Big] = \dbE^\dbP\Big[\dbE^{\dbP}\big[B^{t_i}_{t_{i+1}} \big|\cG_{t_{i+1}}\big]\dbE^{\dbP}\big[\eta\big|\cG_{t_{i+1}}\big]\Big]\\
&&=\dbE^\dbP\Big[\dbE^{\dbP}\big[B^{t_i}_{t_{i+1}} \big|\cG_{t_{i+1}}\big]~\eta\Big].
\eeaa
Since $\eta\in \dbL^\infty(\cF_{t_i}\vee \tilde \cF_{t_i}\vee \cG_{t_{i+1}}; \dbP)$ is arbitrary, this implies \eqref{Ito-claim}.

 Note that, by \eqref{BG} and the (H)-hypothesis we have 
 \beaa
 \dbE\big[B^{t_i}_{t_{i+1}} \big|\cG_{t_{i+1}}\big] = \dbE\big[B_{t_{i+1}} \big|\cG_{t_{i+1}}\big] - \dbE\big[B_{t_i}\big|\cG_{t_{i+1}}\big] = B^\dbG_{t_{i+1}} - B^\dbG_{t_i}.
 \eeaa
Then, by \eqref{Ito-claim} and the conditional independence,
 \bea
\label{Ito-claim2}
 \dbE\big[B^{t_i}_{t_{i+1}} (\tilde B^{t_i}_{t_{i+1}})^{\top}\big|\cF_{t_i}\vee \tilde \cF_{t_i}\vee \cG_{t_{i+1}}\big] = (B^\dbG_{t_{i+1}} - B^\dbG_{t_i})(B^\dbG_{t_{i+1}} - B^\dbG_{t_i})^\top.
 \eea
 Plug this into \eqref{In22i}, we have
 \bea
\label{In22i2}
\dbE[ I^n_{2,2,i}\big] = \dbE\Big[\partial_{x\tilde x} v_i\big(\cX_{t_{i}}, \tilde\cX_{t_{i}}\big) :  \b_{t_i} (B^\dbG_{t_{i+1}} - B^\dbG_{t_i})(B^\dbG_{t_{i+1}} - B^\dbG_{t_i})^\top\tilde \b_{t_i}^\top\Big] + o(\D t).
 \eea

Moreover, denote $v_i^{\th,\tilde\th}(x, \tilde x) := {\d^2\over \d\mu^2}{\d\over \d \bmu}V\big(t_{i},\bmu_i^\th,  \mu_i^{\tilde \th}, x, \tilde x\big)$. Then, similarly to \eqref{In22i} we have
\beaa
~~\dbE[ R^n_{2,2,i}\big] = \dbE\big[ \int_0^1\!\!\!\int_0^1\!\!\!\int_0^{\th'}\!\!\! \big(\partial_{x\tilde x} v_i^{\th,\tilde\th}(\cX_{t_{i}}, \tilde\cX_{t_{i}})-\partial_{x\tilde x} v_i(\cX_{t_{i}}, \tilde\cX_{t_{i}}) \big) \!:\! \D \cX_{i+1} \D \tilde\cX_{i+1}^\top~ d\tilde \th d\th'd\th \big] \!+\! o(\D t).
 \eeaa
 By the standard approximation arguments, we may assume without loss of generality that $\a, \b$ are bounded and continuous in $t$. Then, since $ \partial_{x\tilde x} v_i$, $ \partial_{x\tilde x} v_i^{\th,\tilde\th}$ exist and are continuous and bounded, and $\dbE[|\D \cX_{i+1}|^2] \le C\D t$, we have 
\beaa
\dbE[R^n_{2,2,i}]=o\Big(\dbE[|\D \cX_{i+1}|^2]\Big)+o(\D t)=o(\D t).
\eeaa
 
 Plugging this and \eqref{In22i2} into \eqref{In22}, and applying Lemma \ref{lem-BG}, we have
\beaa
 I^n_{2,2} 
 &=& {1\over 2} \sum_{i=0}^{n-1}\dbE\Big[ \partial_{x\tilde x} v_i\big(\cX_{t_{i}}, \tilde\cX_{t_{i}}\big) :   \b_{t_i} \big(\int_{t_i}^{t_{i+1}} \si^\dbG_t dB_t\big)\big(\int_{t_i}^{t_{i+1}} \si^\dbG_t dB_t\big)^\top \tilde \b_{t_i}^\top\Big] + o(1)\\
&=& {1\over 2} \sum_{i=0}^{n-1}\dbE\Big[ \partial_{x\tilde x}  {\d^2\over \d\mu^2}{\d\over \d \bmu}V\big(t_{i},\bmu_{t_i},  \mu_{t_i}, \cX_{t_{i}}, \tilde\cX_{t_{i}}\big) :  \b_{t_i}\int_{t_i}^{t_{i+1}} \si^\dbG_t (\si^\dbG_t)^\top  dt  \tilde \b_{t_i}^\top\Big] + o(1)\\
 &=&{1\over 2} \dbE\Big[\int_0^T \partial_{x\tilde x}  {\d^2\over \d\mu^2}{\d\over \d \bmu}V\big(t,\bmu_{t},  \mu_{t}, \cX_{t}, \tilde\cX_{t}\big) : \b_t \si^\dbG_t (\si^\dbG_t)^\top \tilde\b^\top_t dt \Big] + o(1).
\eeaa
Put everything together and send $n\to \infty$, we obtain the desired equality \eqref{Ito}.
\qed

\section{Hamilton–Jacobi–Bellman equations on $\cP_2(\cP_2(\dbR^d))$}
\label{sect-HJB}
By applying the It\^{o} formula \eqref{Ito} on the  Dynamic Programming Principle (\ref{DPP}), we introduce the following differential operator for smooth functions $U: [0, T]\times \cP_2(\cP_2(\dbR^d))\to \dbR$:
\bea
\label{dbLV}
&&\dbL U(t, \bmu) = \partial_t U(t, \bmu)+ \int_{\mathcal P_2(\mathbb R^d)}\int_{\mathbb R^d}\cH_1U\big(t,\bmu, \mu,x\big)\mu(dx)\bmu(d\mu)\\
&&\qquad+\frac{1}{2}\int_{\mathcal P_2(\mathbb R^d)}\left(\int_{\mathbb R^{d}\times \dbR^d}\cH_2 U\big(t,\bmu, \mu,  x,\widetilde x, I_d\big) \mu(dx)\mu(d\widetilde x)\right)^+\bmu (d\mu),\q\mbox{where}\nonumber\\
&& \cH_1U\big(t,\bmu, \mu,x\big):= \partial_x\frac{\delta}{\delta\mu}\frac{\delta}{\delta\bmu}U\left(t,\bmu, \mu,x\right) \cdot b(t,x) +\frac{1}{2}\partial_{xx}\frac{\delta}{\delta\mu}\frac{\delta}{\delta\bmu}U\left(t,\bmu, \mu,x\right): \sigma\si^\top(t,x),\nonumber\\
&& \cH_2 U\big(t,\bmu, \mu,  x,\widetilde x, \si'\big):= \partial_{x\tilde x}\frac{\delta^2}{\delta\mu^2}\frac{\delta}{\delta\bmu}U\left(t,\bmu, \mu,  x,\widetilde x\right) : \sigma(t,x) \si' (\si')^\top\sigma^\top(t,\widetilde x).\nonumber
\eea
Here $\si'\in \dbR^{d\times d}$ and $I_d$ denotes the $d\times d$ identity matrix.
We first consider the optimal argument of the above Hamiltonian.
\begin{lemma}
\label{lem-si*}
Let $U: [0, T]\times \cP_2(\cP_2(\dbR^d))\to \dbR$ be smooth in the sense that it satisfies all the requirements in Theorem \ref{thm-Ito}. Then, for any $(t,\bmu)\in [0, T]\times\cP_2(\cP_2(\dbR^d))$, there exists a measurable function $\si^*_U(t, \bmu, \cdot): \cP_2(\dbR^d) \to \dbR^{d\times d}$ such that 
\bea
\label{si*}
&& \left(\int_{\mathbb R^{d}\times \dbR^d}\cH_2U\big(t,\bmu, \mu,  x,\widetilde x, I_d\big)\mu(dx)\mu(d\widetilde x)\right)^+\\
&&=\int_{\mathbb R^{d}\times \dbR^d}\cH_2U\big(t,\bmu, \mu,  x,\widetilde x, \si_U^*(t,\bmu, \mu)\big) \mu(dx)\mu(d\widetilde x
).\nonumber
\eea
\end{lemma}
\proof By the equivalence of convergence under the metric $W_2$ and weak convergence in $\cP_2(\dbR^d)$, see e.g. \cite[Theorem 6.9]{OTV}, one can easily see that the following mapping is measurable:
\beaa
\mu\in \cP_2(\dbR^d)&\mapsto& \phi(t, \bmu; \mu):= \int_{\mathbb R^{d}\times \dbR^d}\cH_2U\big(t,\bmu, \mu,  x,\widetilde x, I_d\big)\mu(dx)\mu(d\widetilde x).
\eeaa
Then $\si^*_U(t, \bmu, \mu) := I_{d\times d} 1_{\{\phi(t, \bmu; \mu)>0\}}$ satisfies the requirement. 
\qed

\smallskip\smallskip
Our main result of this paper is as follows.
\begin{thm}%[Verification Theorem]
\label{thm-verification}
Let $U: [0, T]\times \cP_2(\cP_2(\dbR^d))\to \dbR$ be smooth in the sense that it satisfies all the requirements in Theorem \ref{thm-Ito}. Then $U$ is equal to the value function $V$ defined in \eqref{VtG}-\eqref{VcU} if and only if $U$ satisfies the following HJB equation:
\bea
\label{HJB}
\dbL U(t, \bmu)  + F(t, \bmu)=0,\q U(T, \bmu) = G(\bmu).
\eea
\end{thm}
\proof
{\bf Step 1.} We first show that, if the value function $V$ is smooth, then it satisfies the HJB equation \eqref{HJB}. It is clear that $V(T, \bmu) = G(\bmu)$. For any $(t, \bmu)\in [0, T)\times \cP_2(\cP_2(\dbR^d))$, by Theorem \ref{thm-support}  let $\xi\in \dbL^2(\cF_t)$ and $\cG^0_t\subset \cF_t$ be such that $\cL_{\cL_{\xi|\cG^0_t}} = \bmu$.  For any $0<\d<T-t$, by DPP \eqref{DPP} we have
\beaa
V(t, \bmu)  = \sup_{\dbG \in \cA(t,\cG^0_t)} \Big[V(t+\d, \cL_{\cL_{X^{t,\xi}_{t+\d}|\cG_{t+\d}}}) + \int_t^{t+\d} F(s,\cL_{\cL_{X^{t,\xi}_s|\cG_s}})ds\Big]. 
\eeaa
Denote $X_s:= X^{t,\xi}_s$, $\mu^\dbG_s :=\cL_{X_s|\cG_s}$, $\bmu^\dbG_s:= \cL_{ \mu^\dbG_s}$. By the It\^{o} formula \eqref{Ito} we have
\bea
\label{LGF}
0 &=& \sup_{\dbG \in \cA(t,\cG^0_t)} {1\over \d} \Big[V(t+\d, \bmu^\dbG_{t+\d}) -  V(t, \bmu^\dbG_{t}) +  \int_t^{t+\d}F(s,\bmu^\dbG_s)ds\Big]\nonumber\\
&=&\sup_{\dbG \in \cA(t,\cG^0_t)} {1\over \d}\int_t^{t+\d} \Big[\dbL^\dbG V(s, \bmu^\dbG_s)+F(s, \bmu^\dbG_s)\Big]ds,
\eea
where
\beaa
\dbL^\dbG V(s, \bmu^\dbG_s) := \partial_t V(s, \bmu^\dbG_s) + \mathbb E\Big[\cH_1V\big(s,\bmu^\dbG_s, \mu^\dbG_s,X_s\big) +\frac{1}{2}\cH_2V\big(s,\bmu^\dbG_s, \mu^\dbG_s,  X_s,\widetilde X_s, \si^\dbG_s\big)\Big].
\eeaa

On one hand, by definitions we have $ \dbL^\dbG V \le \dbL V$. Then
\beaa
0 &\le&\sup_{\dbG \in \cA(t,\cG^0_t)} {1\over \d}\int_t^{t+\d} \Big[\dbL V(s, \bmu^\dbG_s)+ F(s, \bmu^\dbG_s)\Big]ds. 
\eeaa
Note that, for $s\in [t, t+\d]$,
\beaa
\cW^2_2(\bmu^\dbG_s, \bmu) \le \dbE\Big[W^2_2(\cL_{X_s|\cG_s}, \cL_{X_t|\cG_s})\Big]\le \dbE\big[|X_s-X_t|^2\big] \le C\d.
\eeaa
Then, by the continuity of $\dbL V$ and $F$, the above clearly implies
\beaa
0 \le \dbL V(t, \bmu)+ F(t, \bmu).
\eeaa
On the other hand, consider the function $\si^*_V$ from Lemma \ref{lem-si*}, and define
\beaa
\cG_s:= \cG^0_t \vee \cF^\cX_s,\q\mbox{where}\q \cX_s := \si^*_V(t, \bmu, \cL_{\xi|\cG^0_t}) B^t_s, \q t\le s\le T.
\eeaa
 It is clear that $\dbG\in \cA(t, \cG^0_t)$, $\si^\dbG_s = \si^*_V(t, \bmu,\cL_{\xi|\cG^0_t})$, and thus $\dbL^\dbG V(t, \bmu^\dbG_t) = \dbL V(t, \bmu)$. By \eqref{LGF},
\beaa
0\ge  {1\over \d}\int_t^{t+\d} \Big[\dbL^\dbG V(s, \bmu^\dbG_s)+ F(s, \bmu^\dbG_s)\Big]ds.
\eeaa
Send $\d\to 0$, we obtain
\beaa
0\ge \dbL^\dbG V(t, \bmu^\dbG_t)+ F(t, \bmu^\dbG_t) = \dbL V(t, \bmu)+ F(t, \bmu),
\eeaa
and therefore, $V$ satisfies \eqref{HJB}.

{\bf Step 2.} We next show that, for an arbitrary solution $U$ of  \eqref{HJB}, we must have $U=V$.  Fix $t, \bmu, \xi, \cG^0_t, X$ and denote $\mu^\dbG, \bmu^\dbG$ as in Step 1. 

We first show that $V(t,\bmu)\le U(t, \bmu)$. For any $\dbG\in \cA(t, \dbG^0)$, by \eqref{VtG},
\beaa
J(t, \ul \xi, \dbG) = G(\bmu^\dbG_T) + \int_t^T F(s,\bmu^\dbG_s)ds = U(T, \bmu^\dbG_T) + \int_t^T F(s,\bmu^\dbG_s)ds 
\eeaa
Then, by the It\^{o} formula \eqref{Ito},
\beaa
&&J(t, \ul \xi, \dbG) - U(t, \bmu) =  U(T, \bmu^\dbG_T) - U(t, \bmu^\dbG_t)+ \int_t^T F(s,\bmu^\dbG_s)ds \\
&&= \int_t^T \Big[\dbL^\dbG U(s, \bmu^\dbG_s) + F(s,\bmu^\dbG_s)\Big]ds =  \int_t^T \Big[\dbL^\dbG U(s, \bmu^\dbG_s) - \dbL U(s,\bmu^\dbG_s)\Big]ds. 
\eeaa
Since $U$ satisfies \eqref{HJB}, by Lemma \ref{lem-BG} we see that $\dbL^\dbG U(s, \bmu^\dbG_s) \le \dbL U (s,\bmu^\dbG_s)$. Then $J(t, \ul \xi, \dbG) \le U(t, \bmu)$ for all $\dbG\in \cA(t, \cG^0_t)$. This implies that $V(t,\bmu) \le U(t, \bmu)$.

To show the opposite inequality, fix $n$ and consider the uniform partition of $[t, T]$: $t=t_0<\cdots<t_n=T$. Recall the $\si^*_U$ in \eqref{si*}. We construct $\dbG^n\in \cA(t, \cG^0_t)$ recursively as follows. First, $\cG^n_t := \cG^0_t$.  For $i=0,\cdots, n-1$ and $s\in (t_i, t_{i+1}]$,
\beaa
\cG^n_s = \cG^n_{t_i} \vee \cF^{\cX^i}_s,\q \cX^i_s =  \si^*_U(t_i, \cL_{\cL_{X_{t_i}|\cG^n_{t_i}}}, \cL_{X_{t_i}|\cG^n_{t_i}}) B^{t_i}_s,\q t_i\le s\le t_{i+1}.
\eeaa
Then we have $\dbL^{\dbG^n}U(t_i, \bmu^{\dbG^n}_{t_i}) =  \dbL U(t_i, \bmu^{\dbG^n}_{t_i})$. Thus
\beaa
&&J(t, \ul \xi, \dbG^n) - U(t, \bmu) = \sum_{i=0}^{n-1} \int_{t_i}^{t_{i+1}} \Big[\dbL^{\dbG^n} U(s, \bmu^{\dbG^n}_s) - \dbL U(s,\bmu^{\dbG^n}_s)\Big]ds\\
&&= \sum_{i=0}^{n-1} \int_{t_i}^{t_{i+1}} \Big[\big[\dbL^{\dbG^n} U(s, \bmu^{\dbG^n}_s) - \dbL^{\dbG^n}U(t_i, \bmu^{\dbG^n}_{t_i}) \big]- \big[\dbL U(s,\bmu^{\dbG^n}_s) - \dbL U(t_i, \bmu^{\dbG^n}_{t_i}) \big]\Big]ds . 
\eeaa
Note that, for $s\in [t_i, t_{i+1}]$,  $\si^{\dbG^n}_s = \si^*_U(\bmu^{\dbG^n}_{t_i}, \cL_{X_{t_i}|\cG^n_{t_i}})$ is a constant variable, and
\beaa
&&\dbE\big[W_2^2(\mu^{\dbG^n}_s, \mu^{\dbG^n}_{t_i})\Big] = \dbE\big[W_2^2(\cL_{X_s|\cG^n_s}, \cL_{X_{t_i}|\cG^n_s})\Big] \le C\dbE\big[|X_s-X_{t_i}|^2\big]\le {C\over n}; \\
&&\cW^2_2(\bmu^{\dbG^n}_s, \bmu^{\dbG^n}_{t_i}) \le {C\over n}.
\eeaa 
Then, by the local uniform regularity of $U$ and the estimate \eqref{Xest}, we can easily see that, for some modulus of continuity function $\rho$,
\beaa
\big|\dbL^{\dbG^n} U(s, \bmu^{\dbG^n}_s) - \dbL^{\dbG^n}U(t_i, \bmu^{\dbG^n}_{t_i}) \big|\le \rho({1\over \sqrt{n}}), \q \big|\dbL U(s,\bmu^{\dbG^n}_s) - \dbL U(t_i, \bmu^{\dbG^n}_{t_i}) \big|\le \rho({1\over \sqrt{n}}). 
\eeaa
This implies that
\beaa
V(t, \bmu) - U(t, \bmu)\ge J(t, \ul \xi, \dbG^n) - U(t, \bmu) \ge \sum_{i=0}^{n-1} \int_{t_i}^{t_{i+1}} 2 \rho({1\over \sqrt{n}}) ds  = 2T \rho({1\over \sqrt{n}}). 
\eeaa
Send $n\to \infty$, we obtain $V(t, \bmu)\ge U(t, \bmu)$.
\qed

We remark that, when $V$ is a classical solution, in Step 2 above we already constructed an $\e$-optimal control $\dbG^n$ for $V(t, \bmu)$ by using $V$ and in particular $\si^*_V$, where $n$ is large enough such that $T\rho({1\over \sqrt{n}})\le \e$. We next construct, at least formally, an optimal control. 

\begin{corollary}
\label{cor-optimal}
Assume $V$ is the classical solution of \eqref{HJB}. For any $(t, \bmu)\in [0, T]\times \cP_2(\cP_2(\dbR^d))$, $\xi\in \dbL^2(\cF_t)$, and $\cG^0_t\subset \cF_t$ with $\cL_{\cL_{\xi|\cG^0_t}}=\bmu$, if the following SDE  on $[t, T]$ has a strong solution:
\bea
\label{SDE*}
d\cX^*_s = \si^*_V(s, \cL_{\cL_{X^{t,\xi}_s|\cG^0_t \vee \cF^{\cX^*}_s}},  \cL_{X^{t,\xi}_s|\cG^0_t \vee \cF^{\cX^*}_s}) dB_s,\q \cX^*_t=0.
\eea
Then the following $\dbG^*\in \cA(t, \cG^0_t)$ is an optimal control of $V(t, \bmu)$: $\cG^*_s := \cG^0_t\vee \cF^{\cX^*}_s$, $s\in [t, T]$. 
\end{corollary}

The proof is straightforward, and thus we omit it. In particular, in the scalar case: $d=1$, $\si^*_V$ takes the following form
\bea
\label{si*V2}
\left.\ba{c}
\displaystyle \si^*_V(t, \bmu, \mu) = I^*\Big(\int_\dbR\int_\dbR \partial_{x\tilde x}\frac{\delta^2}{\delta\mu^2}\frac{\delta}{\delta\bmu} V\big(t, \bmu, \mu, x, \tilde x) \si(t,x) \si(t, \tilde x)\mu(dx)\mu(d\tilde x)\Big),\medskip\\
\displaystyle \mbox{where}\q I^*(x) = 1, ~x>0;\q I^*(x)=0,~x<0;\q 0\le I^*(0) \le 1.
\ea\right.
\eea
We remark that \eqref{SDE*} involves conditional distribution, conditional on the filtration generated by the solution process, so it is highly nontrivial. Moreover, $I^*$ is discontinuous, so even in the scalar case, it is very challenging to obtain a strong solution.

\subsection{ A dynamic insider trading problem} 
\label{sect-insider2}
As an application of our general theory, we conclude this section by extending the insider trading problem in Section \ref{sect-insider1} to the dynamic setting. Inspired by \eqref{tildeu}, we consider 
\bea
\label{insider2-v}
V(t, \xi, \cG^0_t) := -\inf_{\dbG\in \cA(t, \cG^0_t)} \dbE\Big[\int_t^T\big(|\mathbb E[X^{t,\xi}_s|\mathcal G_s]|-1\big)^2ds\Big],\q \mbox{where}\q X^{t,\xi}_s := \xi + B^t_s.
\eea
 That is, we consider \eqref{X}-\eqref{VtG} with, 
\bea
\label{insider2}
b\equiv 0,\q \si\equiv 1,\q G\equiv 0,\q F(\bmu) =- \int_{\cP_2(\dbR)}  \big(|m_\mu|-1\big)^2~  \bmu(d\mu),~ \bmu\in \cP_2(\cP_2(\dbR)),
\eea
where $m_\mu := \int_\dbR x\mu(dx)$. Then the HJB equation \eqref{dbLV}-\eqref{HJB} becomes: 
\bea
\label{insiderHJB}
&&\partial_t V(t, \bmu)  + \int_{\mathcal P_2(\mathbb R)}\Big[\int_{\mathbb R}\frac{1}{2}\partial_{xx}\frac{\delta}{\delta\mu}\frac{\delta}{\delta\bmu}V\big(t,\bmu, \mu,x\big)\mu(dx)\\
&&+\frac{1}{2}\Big(\int_{\mathbb R\times \dbR}\partial_{x\tilde x}\frac{\delta^2}{\delta\mu^2}\frac{\delta}{\delta\bmu}V\big(t,\bmu, \mu,  x,\widetilde x\big) \mu(dx)\mu(d\widetilde x)\Big)^+ - \big(|m_\mu|-1\big)^2\Big]  \bmu(d\mu)=0.\nonumber
\eea
We remark that here we consider only a toy model so that we can apply the results in the previous section. The insider trading problem with dynamic information has independent interest, and we shall leave general models for future research. 

Note that, for $\dbG\in \cA(t, \cG^0_t)$ and for $s\ge t$,
\beaa
\dbE[X^{t,\xi}_s|\mathcal G_s] = \dbE[\xi|\cG^0_t] + B^{\dbG}_s - B^{\dbG}_t =  \dbE[\xi|\cG^0_t] + \int_t^s \si^{\dbG}_r dB_r.
\eeaa
Introduce the following standard control problem:
\bea
\label{insiderv}
v(t,x) := - \inf_{0\le \si'\le 1} \dbE\Big[\int_t^T\big(\big| x+\int_t^s \si'_r dB_r\big|-1\big)^2ds\Big],
\eea
which satisfies the standard HJB equation on $[0, T]\times \dbR$:
\bea
\label{HJBv}
\partial_t v + {1\over 2} \big(\partial_{xx} v\big)^+ - \big(|x|-1\big)^2 =0,\q v(T,x)=0.
\eea
Then one can easily see that
\bea
\label{insiderVv}
V(t, \cL_{\cL_{\xi|\cG^0_t}}) =\dbE\Big[ v\big(t, \dbE[\xi|\cG^0_t]\big)\Big],\q\mbox{or say}\q V(t, \bmu) = \int_{\cP_2(\dbR)} v(t, m_\mu) \bmu(d\mu),
\eea
 In particular, by using the information of $v$, we can easily construct approximate optimal $\dbG\in \cA(t, \cG^0_t)$ for the optimization problem \eqref{insider2-v}.

We next identify the HJB equations \eqref{insiderHJB} and \eqref{HJBv}.  Given $\bmu_1, \bmu_2$, denoting $\bmu^\th := \th \bmu_2 + (1-\th) \bmu_1$, by \eqref{insiderVv} we have
\beaa
V(t, \bmu_2) - V(t, \bmu_1) &=& \int_0^1 \int_{\cP_2(\dbR)}{\d\over \d \bmu}V(t, \bmu^\th, m_\mu) \big(\bmu_2 - \bmu_1\big)(d\mu)\big) d\th;\\
&=&\int_{\cP_2(\dbR)} v(t, m_\mu) \big(\bmu_2-\bmu_1\big)(d\mu).
\eeaa
Then we can easily see that 
\beaa
{\d\over \d\bmu} V(t, \bmu, \mu) = v(t, m_\mu).
\eeaa
We may continue to obtain the other derivatives:
\bea
\label{insider-pav}
&\partial_t V(t, \bmu) = \int_{\cP_2(\dbR)} \partial_t v(t, m_\mu) \bmu(d\mu),\\
&{\d\over \delta\mu}\frac{\delta}{\delta\bmu}V\left(t,\bmu, \mu,x\right) = \partial_x v(t,m_\mu) x,\q {\d^2\over \delta\mu^2}\frac{\delta}{\delta\bmu}V\left(t,\bmu, \mu,x, \tilde x\right) =  \partial_{xx} v(t,m_\mu) x\tilde x.\nonumber
\eea
Plug these into \eqref{insiderHJB}, we obtain
\beaa
\int_{\mathcal P_2(\mathbb R)}\Big[\partial_t v(t, m_\mu)  +\frac{1}{2}\big(\partial_{xx} v(t,m_\mu) \big)^+- \big(|m_\mu|-1\big)^2\Big]  \bmu(d\mu)=0.
\eeaa
Clearly this is equivalent to 
\beaa
 \partial_t v(t, m_\mu) +\frac{1}{2}\big(\partial_{xx} v(t,m_\mu) \big)^+- \big(|m_\mu|-1\big)^2=0,
\eeaa
which identifies with \eqref{HJBv}.

Finally, assume $V$ or $v$ is sufficiently smooth. In this case \eqref{si*} becomes:
\beaa
\big(\partial_{xx} v(t, x) \big)^+=\partial_{xx}  v(t, x) |\si^*_v(t,x)|^2.
\eeaa
That is, recalling the $I^*$ in \eqref{si*V2},
\beaa
\si^*_v(t,  x) = I^*\big(\partial_{xx}  v(t, x)\big).
\eeaa
 Consequently, \eqref{SDE*} becomes:
\bea
\label{insiderSDE*}
d\cX^*_s = I^*\Big(\partial_{xx}  v\big(s, \dbE[X_s|\cG^0_t \vee \cF^{\cX^*}_s]\big)\Big) dB_s,\q \cX^*_t=0.
\eea
Then we obtain the optimal $\cG^*_s := \cG^0_t \vee \cF^{\cX^*}_s$, provided that the above SDE has a strong solution. 

Note that \eqref{insiderSDE*} still involves the filtration generated by the solution and thus is non-standard. Due to the special structure here, we may simplify it further. Given $X_t = \xi$, consider the following standard SDE (but with discontinuous coefficients):
\bea
\label{insiderSDE*2}
\hat\cX^*_s =  \dbE[\xi |\cG^0_t] + \int_t^s I^*\big(\partial_{xx}  v\big(r, \hat\cX^*_r\big)\big) dB_r.
\eea
Then $\si'_s =I^*\big(\partial_{xx}  v\big(s, \hat\cX^*_s\big)\big)$ is an optimal control for the optimization problem $v(t, \dbE[\xi |\cG^0_t])$ in \eqref{insiderv}, in a pointwise sense for $x= \dbE[\xi |\cG^0_t]$. Thus, by \eqref{insiderVv} we may identify the optimal control (assuming it is unique):
 \beaa
 I^*\big(\partial_{xx}  v\big(r, \hat\cX^*_r\big)\big) =  I^*\Big(\partial_{xx}  v\big(s, \dbE[X_s|\cG^0_t \vee \cF^{\cX^*}_s]\big)\Big).
 \eeaa
 Consequently, \eqref{insiderSDE*} becomes:
 \bea
\label{insiderSDE*3}
\cX^*_s = \int_t^s I^*\big(\partial_{xx}  v\big(r, \hat\cX^*_r\big)\big) dB_r = \hat \cX^*_s - \dbE[\xi |\cG^0_t].
\eea
That is, we may first solve (standard) SDE \eqref{insiderSDE*2}, next define $\cX^*_s$ by \eqref{insiderSDE*3}, then $\cG^*_s:= \cG^0_t \vee \cF^{\cX^*}_s$ is the optimal control of $V(t, \cL_{\cL_{\xi|\cG^0_t}})$. We shall note though, the well-posedness of  SDE \eqref{insiderSDE*2} is still not easy, in particular, it is not clear to us if $v$ is smooth.

\begin{remark}
For $|x|\ge 1$, one can easily see that, for any $0\le \si'\le 1$ and $t\le s$,
\beaa
\dbE\big[\big| x+\int_t^s \si'_r dB_r\big|\big]\ge \big|\dbE\big[ x+\int_t^s \si'_r dB_r\big]\big| = |x|\ge 1.
\eeaa
Then
\beaa
 \dbE\Big[\int_t^T\big(\big| x+\int_t^s \si'_r dB_r\big|-1\big)^2ds\Big] &\ge&   \int_t^T\Big(\dbE\big[\big| x+\int_t^s \si'_r dB_r\big|\big]-1\Big)^2ds\\
 &\ge& \int_t^T(| x|]-1)^2ds =(T-t)(|x|-1)^2,
\eeaa
with equalities holding if and only if $\si'\equiv 0$. This implies that
\beaa
v(t,x) = -(T-t)(|x|-1)^2 =: \tilde v(t,x),\q\mbox{with optimal $\si^{'*}\equiv 0$},\q |x|\ge 1.
\eeaa

One can verify that the above function $\tilde v$ satisfies \eqref{HJBv} for all $x\neq 0$. 
However, $\tilde v$ is not a viscosity solution of \eqref{HJBv}, since the viscosity property fails at $x=0$. In particular, the control $\sigma' \equiv 0$, which induces $\tilde v$,  is not optimal for the problem \eqref{insiderv} at $x=0$. So $v(t,x) \ge \tilde v(t,x)$ for $|x|<1$ and in particular $v(t,0) > \tilde v(t, 0)$ when $t<T$.
\end{remark}

\section{A generalized It\^o formula}
\label{sect-general}
We first recall an It\^{o} formula on the space $[0, T]\times \dbR^d\times \cP_2(\dbR^d)$ which is widely used for mean field games with common noise, see e.g. \cite{CD2}. Let $(\O, \dbF, \dbP)$ be a filtered probability space, $B^0, B'$ be independent Brownian motions, and, for $i=1,2$,
\bea
\label{cXi}
d \cX^i_t= b^i_t dt + \si^i_t dB'_t + \si^{i,0}_t d B^0_t,
\eea
where $b^i, \si^i, \si^{i,0}$ are appropriate $\dbF$-progressively measurable processes. Let $v: [0, T]\times \dbR^d\times \cP_2(\dbR^d) \to \dbR$ be a smooth function. Then,  denoting $\dbF^0:= \dbF^{B^0}$ and $\mu_t:=\mathcal{L}_{\cX^2_t|\cF^{0}_t}$,
\bea
\label{MFIto}
&&\!\!\!\!\!\!\!\!\!\! d v(t, \cX^1_t, \mu_t) = \Big\{ \Big[ \partial_t v +  \partial_x v\cdot b^1_t + \frac{1}{2} \partial_{xx} v: [\si_t^1 (\si_t^1)^\top + \si_t^{1,0}(\si_t^{1,0})^\top]\Big](t, \cX^1_t, \mu_t)  \nonumber\\
&&\!\!\!\!\!\!\!\!\!\! + \dbE_{\cF^{0}_t}\Big[ \Big(\partial_{\tilde x}{\d \over \d\mu} v \cdot \tilde b^{2}_t +  \partial_{x\tilde x}{\d \over \d\mu} v : \si^{1,0}_t (\tilde \si_t^{2,0})^\top + \frac{1}{2}  \partial_{\tilde x\tilde x} {\d \over \d\mu} v : [\tilde \si_t^2 (\tilde \si_t^2)^\top + \tilde \si_t^{2,0}(\tilde \si_t^{2,0})^\top]\Big)(t,\cX^1_t,\mu_t,\tilde \cX^2_t)\nonumber\\
&&\!\!\!\!\!\!\!\!\!\! +\frac{1}{2} \partial_{\tilde x\bar x} {\d^2 \over \d\mu^2} v (t,\cX^1_t,\mu_t,\tilde \cX^2_t,\bar \cX^2_t) : \tilde\si_t^{2,0}(\bar \si_t^{2,0})^\top\Big]\Big\}dt+ \partial_x v(t,\cX^1_t,\mu_t)\cdot\si_t^1dB'_t \\
&&\!\!\!\!\!\!\!\!\!\! +\Big[(\si^{1,0}_t)^\top \partial_xv(t,\cX^1_t,\mu_t)+\mathbb E_{\cF^0_t}\big[(\tilde \si^{2,0}_t)^\top \partial_{\tilde x} {\d \over \d\mu} v (t,\cX^1_t,\mu_t,\tilde \cX^2_t) \big]\Big]\cdot dB_t^0.\nonumber
\eea
Introduce
\bea
\label{Vv}
V(t, \cL_{(\cX^1_t, \mu_t)}) := \dbE\big[v(t, \cX^1_t, \mu_t)\big].
\eea
Then $V$ is a function on $[0, T]\times \cP_2(\dbR^d\times \cP_2(\dbR^d))$. In this subsection we extend the It\^{o} formula \eqref{Ito} to smooth functions on this space and verify its consistency with \eqref{MFIto}. We remark that we can easily extend further our formula to cover more general It\^{o} formulae on measure valued processes in the spirit of \eqref{Vv}, and we leave that for interested readers. 

For this purpose, we first extend Definition \ref{defn-Uderivative} in an obvious manner. For $U: \mathcal P_2(\dbR^d\times \cP_2(\mathbb R^d))\rightarrow \mathbb{R}$ and abusing the notation $\bmu$ to denote the elements of $\cP(\dbR^d\times \cP_2(\mathbb R^d))$, we call a continuous function ${\d\over \d\bmu} U: \cP_2(\dbR^d\times \cP_2(\dbR^d))\times \dbR^d\times \cP_2(\dbR^d) \to \dbR$  the linear functional derivative of $U$ if $\big|\frac{\delta}{\delta \bmu}U(\bmu, x, \mu)\big| \leq C\big(1+|x|^2+\|\mu\|_2^2\big)$ for some constant $C$, which may depend on $U$ but is uniform on $\bmu$, and for any $ \bmu_1, \bmu_2\in \cP_2(\dbR^d\times \cP_2(\dbR^d))$,
\bea
\label{linearU2}
 U(\bmu_2)-U(\bmu_1)=\int_0^1 \int_{\dbR^d\times \cP_2(\mathbb{R}^d)} \frac{\delta}{\delta \bmu}U( \theta\bmu_2+(1-\theta) \bmu_1,x, \mu)(\bmu_2-\bmu_1)(dx, \mathrm{d}\mu) \mathrm{d}\theta.
 \eea
Define the higher order derivatives in an obvious manner, we then have the following generalized It\^{o} formula whose proof is similar to Theorem \ref{thm-Ito} and is postponed to Appendix \ref{appendixB}.

\begin{thm}
\label{thm-Ito-general}
Let $\dbG\in \cA_0$, and for $i=1,2$, $\cX^i_t = \cX^i_0 + \int_0^t \a^i_s ds + \int_0^t \b^i_s dB_s$, where $\cX^i_0\in \dbL^2(\cF_0)$ and $\a^i,  \b^i\in \dbL^2(\dbF)$. Let $V: [0, T]\times \mathcal P_2(\dbR^d\times \cP_2(\mathbb R^d))\rightarrow \mathbb{R}$ be such that all the following derivatives, and its lower order derivatives, exist and are locally uniformly continuous in the sense of Assumption \ref{assum-standing} (ii) with $\cK_R:=\{(x, \bmu): |x|\le R, \|\bmu\|_2\le R\}$: 
\bea
\label{derivatives-common}
\left.\ba{c}
\displaystyle\partial_t V(t, \bmu),\q \partial_{xx}{\d\over \d\bmu} V(t, \bmu, x,\mu),\\
\displaystyle  \partial^{(2)}_{(x,\tilde x)} {\d\over \d\mu}{\d\over \d\bmu} V(t, \bmu, x, \mu, \tilde x), \q \partial^{(2)}_{(x, \tilde x, \bar x)}  {\d^2\over \d\mu^2}{\d\over \d\bmu} V(t, \bmu, x,\mu, \tilde x, \bar x).
\ea\right.
\eea
Here $ \partial^{(2)}_{(x,\tilde x)}$ denotes all the second order derivatives with respect to $x$ and $\tilde x$, and similarly for  $\partial^{(2)}_{(x,\tilde x, \bar x)}$. 
Then the following It\^{o} formula holds true:
\bea
\label{Ito-general}
\displaystyle \frac{d}{dt} V(t,\bmu_t)
&=&\partial_t V(t,\bmu_t)+\dbE\Big[\partial_x{\d\over \d \bmu}V\big(t,\bmu_t, \cX^1_t, \mu_t \big) \cdot \a^1_t + {1\over 2}\partial_{xx}{\d\over \d \bmu}V\big(t,\bmu_t,  \cX^1_t, \mu_{t}\big) : \b^1_t(\b^1_t)^\top \nonumber\\
\displaystyle && + \partial_{\tilde x}{\d\over \d\mu}{\d\over \d \bmu}V\big(t,\bmu_t,  \cX^1_t, \mu_t, \tilde\cX^2_t\big)\cdot \tilde\a^2_t + {1\over 2}\partial_{\tilde x\tilde x}{\d\over \d\mu}{\d\over \d \bmu}V\big(t,\bmu_t,\cX^1_t,  \mu_{t}, \tilde\cX^2_t\big) : \tilde\b^2_t(\tilde\b^2_t)^\top\nonumber\\
 &&+\partial_{x{\tilde x}}  {\d\over \d\mu}{\d\over \d \bmu}V\big(t,\bmu_{t}, \cX_t^1, \mu_{t}, \tilde\cX^2_{t}\big) : \b^1_t \si^\dbG_t (\si^\dbG_t)^\top ({\tilde\b}_t^2)^\top \\
 &&+{1\over 2} \partial_{\tilde x\bar x}  {\d^2\over \d\mu^2}{\d\over \d \bmu}V\big(t,\bmu_{t}, \cX_t^1, \mu_{t}, \tilde\cX^2_{t}, \bar\cX^2_{t}\big) : \tilde\b^2_t \si^\dbG_t (\si^\dbG_t)^\top (\bar\b^2_t)^\top\Big],\nonumber
\eea
where $\bmu_t=\mathcal L_{(\cX^1_t, \mathcal L_{\cX^2_{t}|\mathcal G_{t}})}$, $\mu_t=\cL_{\cX^2_t|\cG_t}$, $(\tilde\a^2, \tilde \b^2,\bar \b^2, \tilde \cX^2,\bar \cX^2)$ is a conditionally independent copy of $(\alpha^2, \b^2, \b^2, \cX^2,\cX^2)$, conditional on $\mathcal G_t$.
\end{thm}

\smallskip
\begin{remark}
\label{rem-consistency}
We now show that \eqref{MFIto} and \eqref{Ito-general} are consistent, in the sense of \eqref{Vv}. For notational simplicity, assume $d=2$, $B= (B', B^0)^\top$ with independent $1$-dimensional Brownian motions $B'$ and $B^0$, and $\cX^1, \cX^2$ are both $1$-dimensional.  Then $\a^i=b^i$ and $\b^i = [\si^i, \si^{i,0}]$.  Set $\dbG := \dbF^{B^0}$, then clearly $\dbG$ satisfies the (H*)-hypothesis and $\si^\dbG = \left[\begin{matrix} 0 & 0\\ 0 & 1\end{matrix}\right]$. Take expectation on both sides of \eqref{MFIto}, we have
\bea
\label{MFIto2}
&&{d\over dt} \dbE\big[v(t, \cX^1_t, \mu_t)\big] = \dbE\Big\{ \Big[ \partial_t v +  \partial_x v b^1_t + \frac{1}{2} \partial_{xx} v[|\si_t^1|^2 + |\si_t^{1,0}|^2]\Big](t, \cX^1_t, \mu_t)  \nonumber\\
&&+\Big(\partial_{\tilde x}{\d \over \d\mu} v  \tilde b^{2}_t +  \partial_{x\tilde x}{\d \over \d\mu} v  \si^{1,0}_t \tilde \si_t^{2,0} + \frac{1}{2}  \partial_{\tilde x\tilde x} {\d \over \d\mu} v [|\tilde \si_t^2 \tilde \si_t^2|^2 + |\tilde \si_t^{2,0}|^2]\Big)(t,\cX^1_t,\mu_t,\tilde \cX^2_t)\nonumber\\
&&+\frac{1}{2} \partial_{\tilde x\bar x} {\d^2 \over \d\mu^2} v (t,\cX^1_t,\mu_t,\tilde \cX^2_t,\bar \cX^2_t) \tilde\si_t^{2,0}\bar \si_t^{2,0}\Big\}.
\eea
On the other hand, by \eqref{Vv} we have ${\d\over \d\bmu} V(t, \bmu, x, \mu) = v(t, x, \mu)$. Plug this into \eqref{Ito-general}, we have
\bea
\label{Ito-general2}
\displaystyle \frac{d}{dt} V(t,\bmu_t)
&=&\partial_t V(t,\bmu_t)+\dbE\Big[\partial_xv\big(t, \cX^1_t, \mu_t \big)  \a^1_t + {1\over 2}\partial_{xx}v\big(t,\cX^1_t, \mu_{t}\big)  |\b^1_t|^2 \nonumber\\
\displaystyle && + \partial_{\tilde x}{\d\over \d\mu}v\big(t,\cX^1_t, \mu_t, \tilde\cX^2_t\big) \tilde\a^2_t + {1\over 2}\partial_{\tilde x\tilde x}{\d\over \d\mu}v\big(t,\cX^1_t,  \mu_{t}, \tilde\cX^2_t\big) |\tilde\b^2_t|^2\\
 &&+\partial_{x{\tilde x}}  {\d\over \d\mu}v\big(t,\cX_t^1, \mu_{t}, \tilde\cX^2_{t}\big) \si^{1,0}_t \tilde\si^{2,0}_t +{1\over 2} \partial_{\tilde x\bar x}  {\d^2\over \d\mu^2}v\big(t,\cX_t^1, \mu_{t}, \tilde\cX^2_{t}, \bar\cX^2_{t}\big) \tilde\si^{2,0}_t \bar\si^{2,0}_t\Big].\nonumber
\eea
Since $\partial_t V(t,\bmu_t) = \dbE\big[\partial_t v (t, \cX^1_t, \mu_t) \big]$, and $\a^i=b^i$, $|\b^i|^2 = |\si^i|^2 + |\si^{i,0}|^2$, one can easily verify that \eqref{MFIto2} and \eqref{Ito-general2} are the same.
\end{remark}

\begin{appendices}
\section{Some examples}
\label{appendixA}
In this Appendix we present a few (static) games where control of the accessible information is highly relevant for optimizing the game value.

\subsection{Braess's Paradox}
This is a classical example in algorithmic game theory, see e.g. \cite{Braess_2005,Roughgarden_2016}.
Consider a network consisting of four nodes, as shown in Fig.\ref{fig:a}. There are two independent routes, $A\to B\to D$ and $A\to C\to D$, connecting Node $A$ to Node $D$. Each edge bears a cost $c(x)$, where $x$ is the traffic load on that edge. Assume the total traffic is $1$ unit, and each driver wants to choose a route to minimize the total cost. By a simple analysis, there is a unique equilibrium: $x=0.5$ on each route, and each driver ends up incurring a cost of $1.5$ units .

Now, suppose someone tries to reduce traffic costs by adding a new, fast, zero-cost link between the midpoints $B$ and $C$ of the existing two routes, as shown in Fig.\ref{fig:b}, where the new edge $BC$ has no cost: $c(x)=0$. Alternatively, in our context, we may reinterpret this as that the edge BC was already there, but the information of its existence is newly disclosed.
By direct analysis, we see that  the unique equilibrium in this new setting is that all drivers take the route $A\to B\to C\to D$, which leads to a total cost of $2$ units for each driver. That is, by adding a new edge (or by disclosing a new information), the cost at equilibrium increases for all drivers.

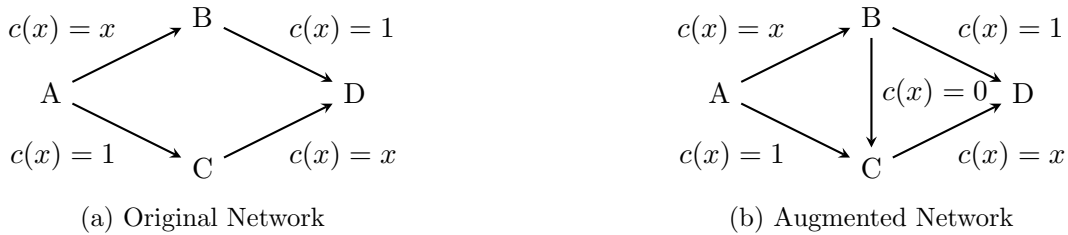
\begin{figure}[H]
\centering
\begin{subfigure}[t]{0.45\textwidth}
\centering
\begin{tikzpicture}[->, >=stealth, node distance=1cm, thick]
    % Nodes
    \node (A) at (0,0) {A};
    \node (B) at (2,1) {B};
    \node (C) at (2,-1) {C};
    \node (D) at (4,0) {D};
    
    % Edges
    \draw (A) -- (B) node[midway, above left] {$c(x)=x$};
    \draw (A) -- (C) node[midway, below left] {$c(x)=1$};
    \draw (B) -- (D) node[midway, above right] {$c(x)=1$};
    \draw (C) -- (D) node[midway, below right] {$c(x)=x$};
\end{tikzpicture}
\caption {Original Network}
\label{fig:a}
\end{subfigure}
\hfill
\begin{subfigure}[t]{0.45\textwidth}
\centering
\begin{tikzpicture}[->, >=stealth, node distance=1cm, thick]
    % Nodes
    \node (A) at (0,0) {A};
    \node (B) at (2,1) {B};
    \node (C) at (2,-1) {C};
    \node (D) at (4,0) {D};
    
    % Edges
    \draw (A) -- (B) node[midway, above left] {$c(x)=x$};
    \draw (A) -- (C) node[midway, below left] {$c(x)=1$};
    \draw (B) -- (D) node[midway, above right] {$c(x)=1$};
    \draw (C) -- (D) node[midway, below right] {$c(x)=x$};
    \draw (B) -- (C) node[midway, right] {$c(x)=0$};
\end{tikzpicture}
\caption{Augmented Network}
\label{fig:b}
\end{subfigure}
\caption{Braess’s Paradox: adding a link to the network leads to worse outcomes for all drivers.}
    \label{fig:network-comparison}

\end{figure}

\subsection{An example of nonzero-sum Nash game}
The Braess’s Paradox is in the spirit of mean field games with infinitely many players. We now consider a two person nonzero-sum Nash game with asymmetric information. 

Let $\xi$ be a standard normal, and, for $\a = (\a_1, \a_2)$
\beaa
X := \xi + \a_1 + \a_2,\q J_i(\a) := \dbE\big[\a_i X - 2 \a_i^2 - \l \a_{-i}^2\big], ~ i=1,2.
\eeaa
Here $\l>0$, $\a_{-i}$ refers to the $\a_j$ for $j\neq i$ in the expression of $J_i(\a)$, and Player $i$ aims to maximize $J_i$ by choosing $\a_i$. Moreover, we assume Player $1$ observes $\xi$ and thus $\a_1$ is $\si(\xi)$-measurable; while Player $2$ has no information and thus $\a_2$ has to be deterministic. It is clear that the equilibrium $\a^*=(\a_1^*, \a_2^*)$ can be characterized by the following system of linear equations: 
\beaa
\xi +\a_2^*  = 2 \a_1^*,\q \dbE\big[\xi + \a^*_1\big]  = 2\a_2^*.
\eeaa
Note that the first equation above implies $\dbE[\xi] +\a_2^* = 2 \dbE[\a_1^*]$, then one can easily obtain:
\bea
\label{Nash1}
\left.\ba{c}
\displaystyle \dbE[\a_1^*] = \a_2^* = \dbE[\xi],\q  \a_1^* = {1\over 2}\big( \xi + \dbE[\xi]\big);\smallskip\\
\displaystyle J_1(\a^*) = {1\over 4} \dbE\big[ \big(\xi + \dbE[\xi]\big)^2\big] - \l \big(\dbE[\xi]\big)^2 = {1\over 4} \dbE[\xi^2] + \big({3\over 4}-\l\big)\big(\dbE[\xi]\big)^2.
\ea\right.
\eea
In particular, since $\xi$ is a standard normal, we see that  $J_1(\a^*)= {1\over 4}$.\footnote{Since $\dbE[\xi]=0$, then $\a_2^*=0$, $\dbE[X] = 0$, and one can easily see that $J_2(\a^*) = -\l \dbE[|\a_1^*|^2] = -{\l\over 4}$. In particular, $J_1(\a^*) > J_2(\a^*)$, that is,  Player 1 has larger value due to the information advantage. However, we shall emphasize that, our purpose here is not to compare Player 1's value with Player 2's value, but rather to compare Player 1's values under different information setting. In fact, in this example, when $\l<{3\over 4}$, by sharing the full information $\xi$, the values of both players increase to $1-\l$.   }

However, Player 1 may choose to share some information $\cG\subset \si(\xi)$ to Player 2, and we assume Player 2 cannot refuse to receive the information. Then $\a_2$ will be $\si(\cG)$-measurable. Similarly to \eqref{Nash1}, for given $\cG$, we see that the equilibrium and Player 1's value become:
\bea
\label{Nash2}
\left.\ba{c}
\displaystyle \dbE[\a_1^*|\cG] = \a_2^* = \dbE[\xi|\cG],\q  \a_1^* = {1\over 2}\big( \xi + \dbE[\xi|\cG]\big);\smallskip\\
\displaystyle J_1(\cG; \a^*) = {1\over 4} \dbE[\xi^2] + \big({3\over 4}-\l\big)\dbE\big[\big(\dbE[\xi|\cG]\big)^2\big].
\ea\right.
\eea
Player 1 would like to choose $\cG\subset \si(\xi)$ to maximize the above $J(\cG;\a^*)$. It is clear that, when $\l>{3\over 4}$, the optimal $\cG^* = \{\emptyset, \O\}$ and $J_1(\cG^*; \a^*) = {1\over 4}\dbE[\xi^2] = {1\over 4}$. However, when $\l<{3\over 4}$, the optimal $\cG^* =\si(\xi)$ and $J_1(\cG^*; \a^*) =(1-\l) \dbE[\xi^2] = 1-\l>{1\over 4}$.

We note that, due to the simple structure here, in this example the optimal $\cG^*$ is either the smallest or the largest sub-$\si$-field. However, as we saw in Section \ref{sect-insider1}, in general models it is possible that it is optimal for the Player 1 to share the information partially.

\subsection{Persuasion Games}
\label{sec_persuation}
The standard Bayesian persuasion model is formulated as a two-player Stackelberg game, see e.g.  \cite{kamenica2011bayesian}. In this subsection, for illustrative purpose we consider only a very simple model.

The game involves two players, the information Sender and the information Receiver. The state of the world is modeled as an $\dbX$-valued random variable $X$ with prior distribution $\mathcal L_X$, known by both players.  The Sender strategically selects an information structure by committing to a signal $S$, an $\dbS$-valued random variable. This defines a joint distribution, known as the information structure, $\mathcal L_{(X,S)}\in \cP(\dbX\times \dbS)$, chosen from an admissible set $\Pi$ in which the marginal distribution of $X$ is $\mathcal L_X$. Upon receiving a signal $S$, the Receiver will choose an action $a$, taking values in $\dbA$.  The timing of the games contains the following five steps:
\begin{enumerate}
    \item Sender chooses an information structure $\pi=\mathcal L_{(X,S)}\in\Pi$.
    \item Receiver is informed about the information structure $\pi=\mathcal L_{(X,S)}$.
    \item Nature chooses a realization of the state and the signal.
    \item Receiver observes the realization of the signal $S$ (but not $X$). 
    \item Receiver takes the action $a\in\dbA$ according to the realization of $S$.
\end{enumerate}
Moreover, the Sender and the Receiver have utility functions $u(x,a)$ and $v(x,a)$, respectively. We consider the Stackelberg game with the Sender as the leader and the Receiver as  the follower: 
\begin{itemize}
\item The Receiver's Problem: given $\pi = \cL_{(X,S)} \in \Pi$,
\bea
\label{Receiver}
V_R(\pi):= \sup_{a\in\dbA}\mathbb E^\pi\big[v(X,a)|S\big].
\eea

    \item The Sender's Problem: assuming \eqref{Receiver} has a unique optimizer $a^*=a^*(\pi; S)$,
\bea
\label{Sender}
V_S := \sup_{\pi\in\Pi}J_S(\pi),\q J_S(\pi):= \mathbb E^\pi[u(X, a^*(\pi; S))].
\eea
\end{itemize}
We note that the Receiver's Problem is a standard control problem, which in particular involves the conditional law $\cL_{X|S}$. The Sender's problem can be viewed as an information control problem. Since $\cL_X$ is given, the joint law $\pi$ is equivalent to the conditional law $\cL_{S|X}$, so the Sender's control is the conditional law $\cL_{S|X}$. Note that, given $\cL_{S|X}$, the conditional law $\cL_{X|S}$ can be obtained by Bayes' rule, thus the Persuasion game is often called Baysian persuasion game. 
 
We next describe the above persuasion game by using our notations and compare it with our information control problem. Let $\mu_0 = \cL_X$ denote the prior distribution of $X$. The sender's control can be viewed as a measure-valued mapping $\cM: \dbX \to \cP(S)$. Then one can easily see that both $V_R(\pi)$ in \eqref{Receiver} and $J_S(\pi)$ in \eqref{Sender} can be written as nonlinear functions of $\cL_{\cM(X)}\in \cP(\cP(S))$.  More importantly, note that $V_S$ in \eqref{Sender} relies on $\cL_X\in \cP(\dbX)$, which can be viewed as $\d_{\cL_X}\in \cP(\cP(\dbX))$. When we consider the problem dynamically, the Sender's dynamic value function will rely on $\bmu=\cL_{\cL_{X|S}}\in \cP(\cP(\dbX))$. So it will be possible to use the techniques developed in this paper to study dynamic persuasion game problems.

We note that, in the persuasion game problem, we disclose information of $X$ by sending new signals $S$. Given $\pi = \cL_{(X, S)}$, we may view $(\cL_X, \cL_{X|S}, \d_X)$ as a $\cP(\dbX)$-valued martingale: 
\bea
\label{mg}
\dbE[\cL_{X|S}] = \mu_0,\q \dbE[\d_X |\si(S)] = \cL_{X|S}.
\eea
In particular, this means that $\d_{\cL_X} \le \bmu = \cL_{\cL_{X|S}} \le \cL_{\d_X}$ under the convex order, namely for any convex function $\Phi: \cP(\cP(\dbX))\to \dbR$, 
\bea
\label{convex}
\Phi(\d_{\cL_X}) \le \Phi(\bmu) \le \Phi(\cL_{\d_X}).
\eea
For our information control problem, we choose sub-information of $X$ directly. In particular, for given $\cG\subset \si(X)$,  $(\cL_X, \cL_{X|\cG}, \d_X)$ is a also $\cP(\dbX)$-valued martingale and $\bmu = \cL_{\cL_{X|\cG}}$ satisfies \eqref{convex} as well. However, in \eqref{mg} $\si(S) \subset \!\!\!\!\!\backslash ~\si(X)$, and thus $\si(S)$ is not an admissible control $\cG$ in our sense, even though it is possible that $\cL_{\cL_{X|S}}$ may identify with $\cL_{\cL_{X|\cG}}$. 

We remark that, the persuasion games use a particular mechanism for disclosing private information. Such mechanism may not be appropriate for information control in more general setting, for example acquiring information (with cost) or various information constraints.  We shall study the information control and game problems in more general settings, and possibly provide a unified framework to cover the persuasion game approach, in future research. 
\end{appendices}

\begin{appendices}
\section{Some technical proofs}
\label{appendixB}

\noindent{\bf Proof of Theorem \ref{thm-support} Case 2: $ supp(\bmu)\subset   \cP([0,1)^d)$.} As mentioned, our idea is to use discrete $\bmu^n$ to approximate $\bmu$, and to prove the convergence of $(X^n, Y^n)$, where $\bmu^n = \cL_{\cL_{X^n|Y^n}}$. To overcome the discontinuity of the conditional distribution operator, we shall construct $(X^n, Y^n)$ in a monotone way: $X^n$ is (essentially) increasing component wise and the $\si$-algebra $\si(Y^n)$ is increasing. We proceed in four steps.

\smallskip
{\it Step 1.} For each $n$, denote $D_n := \{x\in [0, 1)^d: 2^n x \in \dbZ^d\}$, and $\cD_n := \{\D_n(x): x\in D_n\}$, where $\D_n(x):= \big\{\tilde x\in [0, 1)^d: x_i \le \tilde x_i < x_i + 2^{-n}\big\}$. Let $\cP^0_n$ denote the set of probability measures $\mu$ such that $supp(\mu) \subset D_n$ and $2^{2dn} \mu(x)\in \dbZ$ for each $x\in D_n$. Note that $|D_n|= 2^{dn}$.

Fix an arbitrary $\mu\in \cP([0,1)^d)$. For each $n$, construct $\mu_n\in \cP^0_n$ as follows: 
\bea
\label{mun}
\mu_n(x) := {1\over 2^{2dn}} \lfloor 2^{2dn} \mu(\D_n(x))\rfloor, \forall x\in D_n \backslash \{0\},\q \mu_n(0) :=  1-\sum_{x\in D_n \backslash \{0\}} \mu_n(x).
\eea
Here $\lfloor y\rfloor $ denotes the largest integer below $y$. Then clearly 
\beaa
0 \le  \mu(\D_n(x)) - \mu_n(x) < {1\over 2^{2dn}},\q \forall x\in D_n \backslash \{0\}.
\eeaa
Let $\xi\in\sigma(U_2)$ be a random variable with $\cL_\xi=\mu$, which is independent with $U_1$. 
Denote $\xi_n := \sum_{x\in D_n\backslash \{0\}} x 1_{\{\xi\in \D_n(x), U_1 \le {\mu_n(x)\over \mu(\D_n(x))}\}}$. Then one can easily see that $\cL_{\xi_n} = \mu_n$, and 
\beaa
&&\dbE\big[|\xi_n - \xi|^2\big] =\dbE\Big[\sum_{x\in D_n\backslash \{0\}} \big|x 1_{\{\xi\in \D_n(x), U_1 \le {\mu_n(x)\over \mu(\D_n(x))}\}} - \xi 1_{\{\xi\in \D_n(x)\}}\big|^2 + |\xi|^2 1_{\{\xi\in \D_n(0)\}}\Big]\\
&&\le C\dbE\Big[\sum_{x\in D_n\backslash \{0\}} \big[\big|\xi-x|^2 1_{\{\xi\in \D_n(x)\}} +|x|^2 1_{\{\xi\in \D_n(x), U_1 > {\mu_n(x)\over \mu(\D_n(x))}\}} \big] + |\xi|^2 1_{\{\xi\in \D_n(0)\}}\Big]\\
&&\le C\dbE\Big[\sum_{x\in D_n\backslash \{0\}} \big[{1\over 2^{2n}} 1_{\{\xi\in \D_n(x)\}} + 1_{\{\xi\in \D_n(x), U_1 > {\mu_n(x)\over \mu(\D_n(x))}\}} \big] + {1\over 2^{2n}} 1_{\{\xi\in \D_n(0)\}}\Big]\\
&&= {C\over 2^{2n}} \sum_{x\in D_n} \dbP(\xi\in \D_n(x)) + C\sum_{x\in D_n\backslash \{0\}}\dbP(\xi\in \D_n(x)) \dbP(U_1 > {\mu_n(x)\over \mu(\D_n(x))})\\
&&= {C\over 2^{2n}} +C\sum_{x\in D_n\backslash \{0\}} \mu(\D_n(x))\big[1- {\mu_n(x)\over \mu(\D_n(x))}\big] \\
&&= {C\over 2^{2n}} +C\sum_{x\in D_n\backslash \{0\}} \big[\mu(\D_n(x))-\mu_n(x)\big] \le  {C\over 2^{2n}} +C\sum_{x\in D_n\backslash \{0\}} {1\over 2^{2dn}} \\
&&= {C\over 2^{2n}} +C [2^{dn}-1]\times  {1\over 2^{2dn}} \le {C\over 2^{n}}.
\eeaa

We emphasize that it is crucial here that we used $2^{2dn}$ instead of $2^{dn}$ in \eqref{mun}. 
Thus  $W_2(\mu_n, \mu)\le {C_0\over 2^{\frac{n}{2}}}$. In particular, this implies that 
 $\cP^0_\infty := \cup_{n\ge 1} \cP^0_n$ is dense in $\cP([0,1)^d)$. 

\smallskip
{\it Step 2.} Note that $\cP^0_n \subset \cP^0_{n+1}$ and $|\cP^0_n|<\infty$ for each $n$. One may construct a sequence of $\Pi_n:=\{O_i^n\}_{i\ge 1}$ (abusing the notation $\Pi$ with that in \eqref{W2}) such that: 

\begin{itemize}
\item For each $n$, $ \Pi_n$ is a partition of $ \cP([0,1)^d)$;

\item For each $n, i$, there exists $\mu^n_i \in \cP^0_{n}$ (not necessarily in $O^n_i$) such that, for any $\mu\in O^n_i$,
\bea
\label{Pin}
  0 \le  \mu(\D_n(x)) - \mu^n_i(x) < {1\over 2^{2dn}},~\forall x\in D_n \backslash \{0\},\q W_2(\mu^n_i, \mu)\le {C_0\over 2^{\frac{n}{2}}}.
 \eea

\item For each $n, j$, $O^{n+1}_j \subset O^n_i$ for some $i$, that is, $\Pi_{n+1}$ is a refinement  of $\Pi_n$.
\end{itemize}
Indeed, for $n=1$, enumerate $\cP_1^0 = \{\tilde \mu^1_i\}_{i\ge 1}$ and set $\tilde O^1_i$ the collection of $\mu\in \cP([0,1)^d)$ satisfying \eqref{Pin} with the $\mu^n_i$ there replaced with $\tilde \mu^1_i$. By Step 1 we have $\cup_{i\ge 1} \tilde O^1_i =  \cP([0,1)^d)$. Then, $O^1_i := \tilde O^1_i \backslash \cup_{k<i} \tilde O^1_k$ (and removing the empty set) is a partition of $\cP([0,1)^d)$. One can easily see that $\Pi_1 = \{O^1_i\}_{i\ge 1}$ satisfy the first two requirements at above.  

Now assume $\Pi_n$ is given and we want to construct $\Pi_{n+1}$. For each $i$, enumerate $\cP_{n+1}^0 = \{\tilde \mu^{n+1}_j\}_{j\ge 1}$ and set $\tilde O^{n,i}_j$ the collection of $\mu\in O^n_i$ satisfying \eqref{Pin} with the $\mu^n_i$ there replaced with $\tilde \mu^{n+1}_j$. By Step 1 we have $\cup_{j\ge 1} \tilde O^{n,i}_j =  O^n_i$. Then, $O^{n,i}_j := \tilde O^{n,i}_i \backslash \cup_{k<j} \tilde O^{n,i}_k$ (and removing the empty set) is a partition of $O^n_i$. One can easily see that $\Pi_{n+1} = \{O^{n,i}_j\}_{i\ge 1, j\ge 1}$ satisfy all the requirements. We note that, when enumerating $\Pi_{n+1}$, we always list $\{O^{n,i_1}_j\}_{j\ge 1}$ before $\{O^{n,i_2}_j\}_{j\ge 1}$ for $i_1<i_2$.
 
 Define $\bmu^n:= \sum_{i\ge 1} \bmu(O^n_i)\d_{\mu^n_i}$. It is clear that $\cW_2(\bmu^n, \bmu) \le {C\over 2^{\frac{n}{2}}}$.

\smallskip

{\it Step 3.} We now construct appropriate $(X^n, Y^n)$ such that $\cL_{\cL_{X^n|Y^n}}=\bmu^n$.  First, denote 
\bea
\label{Ynconstruction}
y^n_0:= 0,\q y^n_i := \sum_{j\le i} \bmu(O^n_j), ~ i\ge 1, \q Y^n:= \sum_{i\ge 1} y^n_i 1_{(y^n_{i-1}, y^n_{i}]}(U_1).
\eea
 Then $\dbP(Y^n=y^n_i) =  \bmu(O^n_i)$. Without loss of generality, assume for each $i,n\ge 1$, $\bmu(O_i^n)>0$. That is,  the support of $\bmu$ is $\cP(\dbR^d)$. Otherwise, we restrict the following analysis to the support of $\bmu$. Moreover, due to our way of enumerating the elements of $\Pi_{n+1}$ , we can easily see that  
 \bea
\label{Ynconstruction2}
\{y^n_i\}_{i\ge 1} \subset \{y^{n+1}_i\}_{i\ge 1},\q Y^{n+1}\le Y^n, \q\mbox{and}\q \si(Y^n)\subset \si(Y^{n+1}).
\eea

Next, we construct $X^n$ recursively. We shall use $D_n$ for the values of $X^n$ and $\D_{2n}$ for the probabilities of $X^n$. To be precise, for $n=1$ and $i\ge 1$, note that $\sum_{x\in D_n} 2^{2dn} \mu^n_i(x) = 2^{2dn} = |D_{2n}|=|\cD_{2n}|$, and for each $x$, $ 2^{2dn} \mu^n_i(x)$ is an integer,  we may have a partition of $D_{2n}$, denoted as $D_{2n} = \cup_{x\in D_n} D_{2n}(\mu^n_i, x)$ such that $|D_{2n}(\mu^n_i, x)| = 2^{2dn} \mu^n_i(x)$.  We then construct:
\beaa
X^1 :=  \sum_{i\ge 1} 1_{\{Y^1=y^1_i\}} \sum_{x\in D_1} x1_{\D_2(\mu^1_i, x)}(U_2),\q\mbox{where}\q \D_{2n}(\mu^n_i, x):= \bigcup_{\tilde x\in D_{2n}(\mu^n_i, x)} \D_{2n}(\tilde x).
\eeaa
One can easily see that $\cL_{\cL_{X^1|Y^1}}=\bmu^1$.

Now assume we have constructed $X^n$ in the following form such that $\cL_{\cL_{X^n|Y^n}}=\bmu^n$:
\bea
\label{Xnconstruction}
X^n :=  \sum_{i\ge 1} 1_{\{Y^n=y^n_i\}} \sum_{x\in D_n} x1_{\D_{2n}(\mu^n_i, x)} (U_2),
\eea
where  $\{D_{2n}(\mu^n_i, x)\}_{x\in D_n}$ (corresponding to  $\D_{2n}(\mu^n_i, x)$) is a partition of $D_{2n}$ with $|D_{2n}(\mu^n_i, x)| = 2^{2dn}\mu^n_i(x)$.  For each $y^{n+1}_j$, there exists unique $i$ such that $y^n_{i-1} < y^{n+1}_j \le y^n_{i}$, then $\{Y^{n+1}= y^{n+1}_j\}\subset \{Y^n= y^n_i\}$ and $O^{n+1}_j \subset O^n_i$. Consider $\mu^n_i$ and $\mu^{n+1}_j$, and note that  $\D_n(x) = \bigcup_{x'\in  D_{n+1}(x)} \D_{n+1}(x')$, where $D_{n+1}(x):= \D_n(x)\cap D_{n+1}$, $x\in D_n$. Now fix an $\mu\in O^{n+1}_j \subset O^n_i$, for any $x\in D_n\backslash \{0\}$, $x'\in  D_{n+1}(x)$, by \eqref{Pin} we have
 \beaa
 0 \le  \mu(\D_n(x)) - \mu^n_i(x) < {1\over 2^{2dn}},\q  0 \le  \mu(\D_{n+1}(x')) - \mu^{n+1}_j(x') < {1\over 2^{2d(n+1)}}.
 \eeaa
 Then, denoting $\e_n(x) :=  \sum_{x'\in D_{n+1}(x)}\mu^{n+1}_j(x') -  \mu^n_i(x)$,
 \bea
 \label{enx}
 &&\e_n(x) <  \sum_{x'\in D_{n+1}(x)} \mu(\D_{n+1}(x')) - \big[  \mu(\D_n(x)) - {1\over 2^{2dn}}\big] = {1\over 2^{2dn}};\\
 && \e_n(x) >  \sum_{x'\in D_{n+1}(x)} \big[\mu(\D_{n+1}(x')) - {1\over 2^{2d(n+1)}}\big] - \mu(\D_n(x))  =- {2^d\over 2^{2d(n+1)}} = -{1\over 2^{2dn +d}}.\nonumber
 \eea
 Note further that $2^{2d(n+1)}\mu^{n+1}_j(x')$ and  $2^{2d(n+1)}\mu^n_i(x)$ are integers. We then construct a partition 
 \bea
 \label{D-partition}
 D_{2(n+1)} = \bigcup_{x'\in D_{n+1}} D_{2(n+1)}(\mu^{n+1}_j ,x')\q\mbox{with}\q |D_{2(n+1)}(\mu^{n+1}_j ,x')| =  2^{2d(n+1)} \mu^{n+1}_j(x')
 \eea
 as follows. First, for each $x\in D_n\backslash\{0\}$ such that $\e_n(x)<0$, consider a partition  $$
 \D_{2n}(\mu^n_i, x)\cap D_{2(n+1)} =\Big( \bigcup_{x'\in D_{n+1}(x)} D_{2(n+1)}(\mu^{n+1}_j, x')\Big)\bigcup D_{2(n+1)}(\mu^n_i, 0,x),
 $$
  where  $|D_{2(n+1)}(\mu^{n+1}_j, x')| = 2^{2d(n+1)}\mu_j^{n+1}(x')$ for each $x'\in D_{n+1}(x)$. We note that 
  $$|D_{2(n+1)}(\mu^n_i, 0,x) |=-2^{2d(n+1)}\e_n(x)>0.
  $$ Next, consider a partition 
\beaa
&&\displaystyle\Big(\D_{2n}(\mu^n_i, 0)\cap D_{2(n+1)}\Big)\bigcup\Big(\bigcup_{x\in D_n\backslash\{0\}, \e_n(x)<0} D_{2(n+1)}(\mu^n_i, 0,x)\Big) \\
&&\displaystyle = \Big(\bigcup_{x\in D_n\backslash\{0\}, \e_n(x)\ge0} D_{2(n+1)}(\mu^n_i, 0,x)\Big)\bigcup D_{2(n+1)}(\mu^n_i, 0,0),  
\eeaa
such that $|D_{2(n+1)}(\mu^n_i, 0,x)| = 2^{2d(n+1)}\e_n(x)$ for each $x\in D_n\backslash \{0\}$ such that $\e_n(x)>0$.  We note that $ D_{2(n+1)}(\mu^n_i, 0,x)$ is empty when $\e_n(x)=0$. Then for each $x\in D_n\backslash\{0\}$ such that $\e_n(x)\ge 0$, we can construct a partition 
  \beaa
  D_{2(n+1)}(\mu^n_i, 0,x)  \bigcup \big(\D_{2n}(\mu^n_i, x)\cap D_{2(n+1)} \big) = \bigcup_{x'\in D_{n+1}(x)} D_{2(n+1)}(\mu^{n+1}_j, x'),
  \eeaa
 such that $|D_{2(n+1)}(\mu^{n+1}_j, x')| = 2^{2d(n+1)}\mu^{n+1}_j(x')$ for each $x'\in D_{n+1}(x)$. Finally, by setting $D_{2(n+1)}(\mu^{n+1}_j, 0)=D_{2(n+1)}(\mu^n_i, 0,0)$, we obtain the desired partition \eqref{D-partition}.
We are now ready to define
\beaa
X^{n+1} :=  \sum_{j\ge 1} 1_{\{Y^{n+1}=y^{n+1}_{j}\}} \sum_{x'\in D_{n+1}} x' 1_{\D_{2(n+1)}(\mu^{n+1}_j, x')} (U_2).
\eeaa
 It is straightforward to verify that $\cL_{\cL_{X^{n+1}|Y^{n+1}}} = \bmu^{n+1}$.
 
Fix $j$ and the corresponding $i$, and consider the event $\{Y^{n+1}=y^{n+1}_{j}\}$. Note that, $X^{n+1}\ge 0 = X^n$ when $U_2\in \D_{2n}(\mu^n_i, 0)$, and for each $x\in D_n\backslash\{0\}$,  $X^{n+1}=x'\ge x = X^n$ when $U_2 \in \D_{2n}(\mu^n_i, x)\cap \D_{2(n+1)}(\mu^{n+1}_j, x')$. Then, by \eqref{enx},
\beaa
&&\dbP\big(X^{n+1}< X^n \big|Y^{n+1}=y^{n+1}_{j}\big) \le \sum_{x\in D_n\backslash\{0\}}\dbP\Big(U_2 \in  \D_{2n}(\mu^n_i, x)\backslash \bigcup_{x'\in D_{n+1}(x)} \D_{2(n+1)}(\mu^{n+1}_j, x')\Big)\\
&&= \sum_{x\in D_n\backslash\{0\}} \Big[\mu^n_i(x)-\sum_{x'\in D_{n+1}(x)} \mu^{n+1}_j(x')\Big]^+ \le  \sum_{x\in D_n\backslash\{0\}} (-\e_n(x))^+\\
&&\le \sum_{x\in D_n\backslash\{0\}} {1\over 2^{2dn+d}} = {2^{dn}-1\over 2^{2dn+d}}\le {1\over 2^{d(n+1)}}.
\eeaa
That implies immediately that
\bea
\label{Xnmon}
\dbP\big(X^{n+1}< X^n\big)\le {1\over 2^{d(n+1)}}.
\eea

\smallskip

 {\it Step 4.} Note that $\sum_{n=1}^\infty {1\over 2^{d(n+1)}} < \infty$. By \eqref{Xnmon} we see that 
 \beaa
 \dbP\Big(\bigcup_{n\ge 1} \bigcap_{m\ge n} \{X^{m+1}\ge X^m\}\Big) = 1.
 \eeaa
  That is, for a.e. $\o\in \O$, $X^n$ is increasing when $n$ is large enough. Thus there exists $X$ such that $X^n\to X$, a.s. Since $X^n$ takes values in $[0, 1)^d$, then by the bounded convergence theorem we have $\lim_{n\to\infty}\dbE[|X^n-X|^2]=0$.
  
Since $U_1$, $U_2$ are independent and $\si(Y^n)$ is increasing in $n$, by \eqref{Ynconstruction} and \eqref{Xnconstruction} we can easily see that $\cL_{X^n|Y^{m}}= \cL_{X^n|Y^n}$ for all $m\ge n$. Moreover, since $Y^{n+1}\le Y^n$, then clearly $Y^n\downarrow Y$, for some random variable $Y$. By \eqref{Ynconstruction} again, 
we see that $Y^n= \sum_{i\ge 1} y^n_i 1_{(y^n_{i-1}, y^n_{i}]}(Y)$,  then we have $\si(Y) = \vee_{n\ge 1} \si(Y^n)$, which implies further that $\cL_{\cL_{X^n|Y}}= \cL_{\cL_{X^n|Y^n}} = \bmu^n$. Thus, by \eqref{W2},
\begin{equation}
\label{verify law}
\begin{aligned}
\cW_2^2(\bmu,\cL_{\cL_{X|Y}})&\le 2\cW_2^2(\bmu,\bmu^n)+2\cW_2^2(\cL_{X^n|Y}, \cL_{\cL_{X|Y}})
\le  \frac{C}{2^{n}}+2\dbE[|X^n-X|^2].
\end{aligned}
\end{equation}
Send $n\to\infty$, we obtain    $\bmu=\cL_{\cL_{X|Y}}$.
\qed

\bigskip

\noindent{\bf Proof of Theorem \ref{thm-coupling}.} For the convenience of the construction below, we assume $(\O, \cF, \dbP) = (\O_0\times \O', \cF_0\otimes \cF', \dbP_0\times \dbP')$ is a product space and both $(\O_0, \cF_0, \dbP_0)$ and $(\O', \cF', \dbP')$ are atomless. Since $(\O, \cF, \dbP)$ is atomless, this assumption is without loss of generality. Indeed, we may construct independent uniformly distributed random variables $U_0, U'$ on $(\O, \cF, \dbP)$, and then view $(\O_0, \cF_0, \dbP_0)$ and $(\O', \cF', \dbP')$ as the canonical spaces of $U_0$ and $U'$, respectively. We shall use $\o$ and $\o'$ to  denote the elements of $\O_0$ and $\O'$, respectively, and thus the elements of $\O$ are denoted as $(\o, \o')$.  Recall \eqref{cPmu12}.  We proceed the proof in three steps. 

{\bf Step 1.} Since $\cP_p(\dbR^d)$ is a Polish space,  for any $\bmu_1, \bmu_2\in\cP_p(\cP_p(\dbR^d))$, by \cite[Theorem 4.1]{OTV} there exists an optimal coupling $\Pi\in \cP(\bmu_1, \bmu_2)$ such that
\bea
\label{Pimu}
\cW_p^p(\bmu_1,\bmu_2)=\int_{\cP_p(\dbR^d)\times\cP_p(\dbR^d)}W_p^p(\mu_1,\mu_2)\Pi(d\mu_1,d\mu_2).
\eea
By Corollary \ref{cor-multidim-support}, on the space $(\O_0, \cF_0, \dbP_0)$ there exist  $\xi_1,\xi_2\in \dbL^0(\cF_0)$ and $\cG:= \si(Y)\subset \cF_0$ such that $\Pi=\cL_{\left(\cL_{\xi_1|\cG},~\cL_{\xi_2|\cG}\right)}$, where the laws are under $\dbP_0$. Denote $\mu_i^{\o}:=\cL_{\xi_i|\cG}(\o)$, $\o \in\O_0$, $i=1,2$. Then, the mapping $\o \in \O_0 \to \mu_i^{\o }\in \cP_p(\dbR^d)$ is $\cG$-measurable. Moreover, for each $\o $, again by \cite[Theorem 4.1]{OTV},  there exists an optimal coupling $\pi^{\o }\in \cP(\mu^{\o }_1, \mu^{\o }_2)\subset \cP(\dbR^{2d})$ such that
\bea
\label{piomega}
W_p^p(\mu^{\o }_1,\mu^{\o }_2)=\int_{\dbR^d\times\dbR^d}|x_1-x_2|^p\pi^{\o }(dx_1,dx_2).
\eea

{\bf Step 2.} In this step, we apply the measurable maximum theorem
\cite[Theorem 18.19]{CAKB}  to choose a version of $\pi^{\o }$ so that the mapping $\o \in \O_0 \to \pi^{\o}\in \cP_p(\dbR^{2d})$ is $\ol\cG$-measurable, where $\ol\cG$ is the $\dbP$-completion of $\cG$. For that purpose, we need to verify the three conditions required there.

Firstly,  by \cite[Lemma 4.4]{OTV}, $\cP(\mu^{\o }_1, \mu^{\o }_2)\subset \cP_p(\dbR^{2d})$ is non-empty and tight. Meanwhile, since $\|\pi\|^p_p=\|\mu_1^{\o}\|_p^p+\|\mu_2^{\o}\|_p^p$, for any $\pi\in\cP(\mu^{\o }_1, \mu^{\o }_2)$, by \cite[Definition 6.8]{OTV}, any weakly convergent sequence in  
$\cP(\mu^{\o }_1, \mu^{\o }_2)$ converges weakly in $\cP_p(\dbR^{2d})$, and consequently converges with respect to $W_p$ by \cite[Theorem 6.9]{OTV}. Thus, $\cP(\mu^{\o }_1, \mu^{\o }_2)$ is a precompact subset of  $\cP_p(\dbR^{2d})$. Moreover, if $W_p(\pi_n,\pi)\to0$, as $n\to\infty$ for any $\{\pi_n\}\subset\cP(\mu^{\o }_1, \mu^{\o }_2)$, then by \cite[Theorem 6.9]{OTV} again, $\pi_n$ converges weakly to $\pi$, which implies $\pi\in\cP(\mu^{\o }_1, \mu^{\o }_2)$ and $\cP(\mu^{\o }_1, \mu^{\o }_2)$ is closed.  Therefore, $\cP(\mu^{\o }_1, \mu^{\o }_2)$ is compact in $\cP_p(\dbR^{2d})$.  This in particular means that the following mapping is compact-valued:
\bea
\label{mapping}
\phi: \o \in \O_0 ~\mapsto ~ \cP(\mu^{\o }_1, \mu^{\o }_2) \in 2^{\cP_p(\dbR^{2d})}.
\eea

Secondly, since the projection mapping is continuous under $W_p$:
\beaa
\operatorname{Proj}: \pi\in \cP_p(\dbR^{2d}) ~\mapsto~ (\pi(dx_1,\dbR^d), \pi(\dbR^d, dx_2))\in \cP_p(\dbR^d)\times \cP_p(\dbR^d),
\eeaa
and $\o \mapsto(\mu_1^{\o },\mu_2^{\o })$ is $\cG$-measurable, then the graph of $\phi$ is $\cG\otimes\cB(\cP_p(\dbR^{2d}))$-measurable:
\beaa
\operatorname{Gr}(\phi)
:=
\{(\o,\pi):\operatorname{Proj}(\pi)=(\mu_1^\o,\mu_2^\o)\}.
\eeaa
Since $\cP_p(\dbR^{2d})$ is Polish, for every open set $O\subset \cP_p(\dbR^{2d})$, by \cite[Theorem 2.12]{Crauel2002} the set 
\beaa
\{\o\in \O_0: \phi(\o)\cap O\neq\emptyset\}
=
\mathrm{Proj}_{\O_0}\big(\operatorname{Gr}(\phi)\cap (\O_0\times O)\big)
\eeaa
is $\overline\cG$-measurable, where $\mathrm{Proj}_{\O_0}$ is the projection onto $\O_0$. That is, the mapping  $\phi$ is  weakly measurable.

Thirdly, by \cite[Theorem 6.9]{OTV}, the following function $\cJ$ on $\cP_p(\dbR^{2d})$ is continuous under $W_p$:
\beaa
\cJ(\pi):= \int_{\dbR^d\times\dbR^d}|x_1-x_2|^p\pi(dx_1,dx_2),~\pi\in \cP_p(\dbR^{2d}).
\eeaa
Note that $\cJ$ is independent of $\omega$. Then, by extending it to $\O_0\times\cP_p(\dbR^{2d})$: $\cJ(\o, \pi) = \cJ(\pi)$, we see that $\cJ$ is a Carath\'eodory function (c.f. \cite[Definition 4.50]{CAKB}) on $\O_0\times\cP_p(\dbR^{2d})$.

We are now ready to apply the measurable selection theorem. Applying \cite[Theorem 18.19]{CAKB} to the mapping $\phi$ and the function $\cJ$, we see that the set-valued mapping 
\beaa
\o\in \O_0 ~\mapsto ~ \arg\min_{\pi\in \phi(\o)}\cJ(\pi)
\eeaa
admits a $\ol\cG$-measurable selector. Noting from \eqref{piomega} that $\pi^\o\in \arg\min_{\pi\in \phi(\o)}\cJ(\pi)$, this exactly means that we can have a $\ol\cG$-measurable version of $\pi^\o$.

{\bf Step 3.} We now extend $Y$ to $\O=\O_0\times \O'$ in an obvious way and construct the desired $X=(X_1, X_2)$ on $\O$. Let $U'$ be a random variable on $(\O', \cF', \dbP')$ with distribution Uniform$([0,1])$. By \cite[Lemma 3.2 (vii)]{Kallenberg2021}, there exists a $\ol\cG\otimes \cB([0,1])$-measurable function $\Phi: \O_0\times [0,1]\to \dbR^{2d}$ such that, for any fixed $\o\in \O_0$,  $\cL_{\Phi(\o, U')}=\pi^\o$. Define
\beaa
X(\o, \o') := \Phi(\o, U'(\o')),\q (\o, \o') \in \O.
\eeaa
Then $X$ is $\ol\cG\otimes \cF'$-measurable. Since we are using the product space, in particular $\ol\cG$ and $U'$ are independent, for any bounded and Borel measurable function $\f: \dbR^{2d}\to\dbR$ we have
\beaa
\dbE[\f (X)|\ol \cG](\o)=\dbE[\f(\Phi(\cdot,U'(\cdot)))|\ol\cG](\o)=\dbE[\f(\Phi(\o,U'(\cdot)))],\q \mbox{for}~\dbP_0\mbox{-a.e.}~\o\in \O_0.
\eeaa
That is, $\cL_{X|\ol\cG}(\o) = \cL_{\Phi(\o,U')} = \pi^\o$, for $\dbP_0$-a.e. $\o$. Since $\cL_{X|\ol\cG}=\cL_{X|\cG}$, $\dbP_0$-a.s., then, 
\bea
\label{X12mu12}
\cL_{(X_1, X_2)|\cG}(\o) = \cL_{X|\cG}(\o) = \pi^\o\in \cP(\mu^\o_1, \mu^\o_2)\q \mbox{for}~\dbP_0\mbox{-a.e.}~\o\in \O_0.
\eea
This implies further that the marginals also coincide:
\beaa
(\cL_{X_1|\cG}, \cL_{X_2|\cG})(\o) =(\mu^\o_1, \mu^\o_2) = (\cL_{\xi_1|\cG}, \cL_{\xi_2|\cG})(\o)\q \mbox{for}~\dbP_0\mbox{-a.e.}~\o\in \O_0.
\eeaa
Thus, for $i=1,2$, 
\bea
\label{marginal}
\cL_{\cL_{X_i|\cG}} = \cL_{\cL_{\xi_i|\cG}} = \bmu_i.
\eea
Moreover, recall that $\Pi = \cL_{(\cL_{\xi_1|\cG}, \cL_{\xi_2|\cG})}$. First by  \eqref{X12mu12},  then by \eqref{piomega} and \eqref{Pimu},  we have
\beaa
&&\dbE[|X_1-X_2|^p] = \dbE\Big[\dbE[|X_1-X_2|^p|\cG]\Big] = \int_{\O_0} \int_{\dbR^d\times \dbR^d} |x_1-x_2|^p \pi^\o(dx_1, dx_2) \dbP_0(d\o)\\
&&= \int_{\O_0} W^p_p(\mu^\o_1, \mu^\o_2) \dbP_0(d\o) = \dbE\Big[ W^p_p(\cL_{\xi_1|\cG}, \cL_{\xi_2|\cG}) \Big]\\
&&=\int_{\cP_p(\dbR^d)\times\cP_p(\dbR^d)}W_p^p(\mu_1,\mu_2)\Pi(d\mu_1,d\mu_2)=\cW_p^p(\bmu_1,\bmu_2).
\eeaa
This, together with \eqref{marginal}, completes the proof.
\qed

\bigskip

We now turn to the proof of Proposition \ref{prop-H}. We need the following simple lemma.
 \begin{lemma}
 \label{lem-H}
Assume $\cG\subset \cF_1 \vee \cF_2$, and $\cG\vee \cF_1$ is independent of $\cF_2$, then $\cG\subset \cF_1$, a.s.
 \end{lemma}
 \proof For any $A\in \cG$ and $A_1\in \cF_1$, $A_2\in \cF_2$, denote $X:= 1_A - \dbE[1_A|\cF_1]$.  Note that $X 1_{A_1}$ is which is $\cG\vee \cF_1$ measurable, and thus independent of $\cF_2$. Then we have
 \beaa
 &&\dbE\big[ X 1_{A_1\cap A_2}\big] =  \dbE\big[ (X 1_{A_1}) 1_{A_2}\big] = \dbE\big[ (X 1_{A_1}) \big]\dbP(A_2)\\
 &&= \dbE\Big[ 1_A 1_{A_1} -  \dbE[1_A|\cF_1]1_{A_1} \Big]\dbP(A_2) = \dbE\Big[ 1_A 1_{A_1} -  \dbE[1_A1_{A_1}|\cF_1] \Big]\dbP(A_2)=0.
 \eeaa
 Note that $X$ is $\cF_1 \vee \cF_2$-measurable, and $\cF_1 \vee \cF_2$ is generated by $\{A_1\cap A_2: A_1\in \cF_1, A_2\in \cF_2\}$, then the above implies $X=0$, a.s. That is, $\dbE[1_A|\cF_1] = 1_A$, a.s. Then $A\in \cF_1$, a.s., namely $A\in \cF_1\vee \cN$. Since $A\in \cG$ is arbitrary, we obtain $\cG\subset \cF_1$, a.s.
\qed
 
 \smallskip \smallskip
 
\noindent{\bf Proof of Proposition \ref{prop-H}.} The only if direction is obvious. We now prove the if direction. Fix $t<s$ and assume \eqref{H*} holds true. Let $\dbP^\o$ denote the conditional probability distribution of $\dbP$, conditional on $\cG_{t}$. Then, for $\dbP$-a.e. $\o\in \O$,  $\cG_s \vee \cF^t_s$ and $\cF_t$ are independent under $\dbP^\o$. Note that $\cG_s \subset \cF_s \subset (\cG_t \vee \cF^t_s) \vee \cF_t$, and $\cG_s \vee (\cG_t \vee \cF^t_s) = \cG_s \vee \cF^t_s$, by Lemma \ref{lem-H} we have $\cG_s \subset \cG_t \vee \cF^t_s$, $\dbP^\o$-a.s. for $\dbP$-a.e. $\o$. This implies that $\cG_s \subset \cG_t \vee \cF^t_s$, $\dbP$-a.s.
\qed

\bigskip

\noindent{\bf Proof of Lemma \ref{lem-invariance}.} We proceed in four steps.

{\it Step 1.} Fix a large $m$ and denote  $s_m:= t+{1\over m}<t_1$. By considering $\cP_2(\dbR^d) \subset \cP_1(\dbR^d)$ and $\cP_2(\cP_2(\dbR^{nd})) \subset \cP_1(\cP_1(\dbR^{nd}))$ as subspaces, let $\{O_i\}_{i\ge 1}$ be a partition of  $\cP_2(\dbR^d)$ such that $W_1(\mu, \mu')\le {1\over 2^m}$ for all $\mu, \mu'\in O_i$, and $\{\cO_j\}_{j\ge 1}$  a partition of $\cP_2(\cP_2(\dbR^{nd}))$ such that $\cW_1(\bmu,\bmu')\le {1\over 2^m}$ for all $\bmu,\bmu'\in  \cO_j$. Denote  $\vec B^{s_m}_{1:n} := (B^{s_m}_{t_1},B^{s_m}_{t_2},\cdots,B^{s_m}_{t_n})$ and 
\beaa
E_i := \big\{\cL_{\xi|\cG^0_t} \in O_i\big\},\q E_{ij} := E_i \bigcap \big\{ \cL_{\cL_{\vec B^{s_m}_{1:n}|\cG_{t_n}} \big|\cG_{s_m}}\in \cO_j\big\}, \q E_i':= \big\{\cL_{\xi'|\cG^{'0}_t} \in O_i\big\}.
\eeaa
Then $E_i\in \cG^0_t= \cG_t$, $E_{ij}\in \cG_{s_m}$, $E_i'\in \cG^{'0}_t$, $\{E_i\}_{i\ge 1}$, $\{E'_i\}_{i\ge 1}$ form partitions of $\O$, $\{E_{ij}\}_{j\ge 1}$ form a partition of $E_i$, and $\dbP(E_i) = \dbP(E_i')$. We emphasize that $\cL_{\cL_{B^{s_m}_s|\cG_s}\big|\cG_{s_m}}$ is different from $\cL_{\cL_{B^{s_m}_s|\cG_s|\cG_{s_m}}}$,\footnote{For example, when $\cG_{s} = \cF_s$, we have $\cL_{\cL_{B^{s_m}_s|\cG_s}\big|\cG_{s_m}} =\cL_{\d_{B^{s_m}_s}\big|\cF_{s_m}}= \cL_{(\d_{B^{s_m}_s})}$, which is different from $\cL_{\cL_{B^{s_m}_s|\cG_s|\cG_{s_m}}} =\d_{(\cL_{B^{s_m}_s})} $.} and without loss of generality we assume $\dbP(E_{ij})>0$ for all $i, j$. Moreover, by the (H)-hypothesis,
\bea
\label{law-cFcG}
\cL_{\cL_{\vec B^{s_m}_{1:n}|\cG_{t_n}} \big|\cG_{s_m}} = \cL_{\cL_{\vec B^{s_m}_{1:n}|\cG_{t_n}} \big|\cF_{s_m}}.
\eea

Next, set $\cG'_s:= \cG^{'0}_t \vee \cF^{B^t}_s$, $t\le s\le s_m$. Since $\cF^{B^t}_{s_m}$ has normal distribution and $\dbP(E_i) = \dbP(E_i')$, we can construct a partition $\displaystyle E'_i = \cup_{j\ge 1} E'_{ij}$ such that $E'_{ij}\in  \cG'_{s_m}$ and $\dbP(E'_{ij}) = \dbP(E_{ij})$ for any $j$. 

\smallskip
{\it Step 2.} To construct the $\cG'_s$ for $s>s_m$, we let $\dbP^{\o}$ denote the regular conditional probability distribution (r.c.p.d.), conditional on $\cF_{s_m}$. That is, $\cL_{\cL_{\vec B^{s_m}_{1:n}|\cG_{t_n}} \big|\cF_{s_m}}(\o)$ is equal to the $\dbP^\o$-distribution of $\cL_{\vec B^{s_m}_{1:n}|\cG_{t_n}}$ for $\dbP$-a.e. $\o$, and $\dbP^\o(\O_\o)=1$, where $\O_\o := \{\tilde\o\in \O: \tilde \o_{[0, s_m]} = \o_{[0, s_m]}\}$. Note that the existence of $\dbP^\o$ is guaranteed since $\cF_{s_m}$ is countably generated, see, e.g. \cite[Theorem 1.1.8]{Stroock-Varadhan}. 
For each $\o\in \O$, $A\in \cF_T$, and $E\in \cF_{s_m}$, denote
\bea
\label{Aw}
A^\o :=  \big\{B^{s_m}_{[s_m, T]}(\tilde\o): \tilde \o\in A \cap \O_{\o}\big\},\q E\oplus A^\o := \big\{\tilde \o \in E: B^{s_m}_{[s_m, T]}(\tilde\o)\in A^\o\big\}.
\eea
Given $A\in \cF_s$, it is clear that $A^\o\in \cF^{s_m}_s$, and thus $E\oplus A^\o \in \cG_{s_m} \vee \cF^{s_m}_s$ when $E\in \cG_{s_m}$.

For each $i, j$, by Step 1 we may fix an $\o^{ij}\in E_{ij} \subset E_i$  such that  
\bea
\label{oij}
\cW_1(\bmu^{ij}, \bmu)<{1\over 2^m},\ \forall\ \bmu\in \cO_j,\q\mbox{where}\q \bmu^{ij} := \cL^{\dbP^{\o^{ij}}}_{\cL^{\dbP^{\o^{ij}}}_{\vec B^{s_m}_{1:n}|\cG_{t_n}}}.
\eea
Denote $\dbP^{ij}:= \dbP^{\o^{ij}}$ and  $A^{i,j}:= A^{\o^{ij}}$. For each $s\ge s_m$, we construct $ \cG'_s$ as follows:  
\bea
\label{tildecG'}
  \cG'_s :=  \cG'_{s_m} \vee \si\big(E'_{ij}\oplus A^{i,j}: A\in \cG_s, i,j\ge 1\big).
\eea
In this step we show that  $\dbG'\in \cA(t, \cG^{'0}_t)$. 

First, for any $t\le s_1\le s_2\le s_m$, it is obvious that 
\bea
\label{hypothesis-verify}
\cG'_{s_2}  =  \cG'_{s_1}\vee \cF^{s_1}_{s_2}.
\eea
We next fix $s_m\le s_1\le s_2\le T$. By the (H*)-hypothesis, we have $\cG_{s_2}\subset\cG_{s_1}\vee \cF^{s_1}_{s_2}$, a.s., and thus
\beaa
\cG'_{s_2}\subset \cG'_{s_m}\vee \sigma\big(E'_{ij}\oplus A^{i,j}: A\in \cG_{s_1}\vee \cF_{s_2}^{s_1},  i,j\ge 1\big),~\mbox{a.s.}
\eeaa
Then it suffices to prove that, for any $i, j\ge 1$.
\bea
\label{subsetgs2'}
E'_{ij}\oplus A^{i,j} \in \cG'_{s_1} \vee \cF^{s_1}_{s_2},~\mbox{a.s.},\q\forall A\in \cG_{s_1}\vee \cF_{s_2}^{s_1}.
\eea
To see this, denote
\beaa
\cG_{ij}:= \big\{A\in   \cG_{s_1}\vee \cF_{s_2}^{s_1}: E'_{ij}\oplus A^{i,j} \in \cG'_{s_1} \vee \cF^{s_1}_{s_2},~\mbox{a.s.}\big\}.
\eeaa
First, if $A\in \cG_{s_1}$, by definition $E'_{ij}\oplus A^{i,j} \in \cG'_{s_1}$ and thus $A\in \cG_{ij}$. Next, if $A\in \cF^{s_1}_{s_2}$, by \eqref{Aw} we see that $E'_{ij}\oplus A^{i,j} = \{\o\in E'_{ij}:  B^{s_m}_{[s_1, s_2]}(\o) \in A\}\in \cG'_{s_m} \vee \cF^{s_1}_{s_2}$, and thus we also have $A\in \cG_{ij}$. Moreover, for any $A\in \cG_{ij}$ and disjoint $A_n \in \cG_{ij}$, $n\ge 1$, by \eqref{Aw} one can easily see that
\beaa
E'_{ij}\oplus (A^c)^{i,j} = E'_{ij}\backslash (E'_{ij}\oplus A^{i,j}),\q  E'_{ij}\oplus (\cup_{n\ge 1}A_n)^{i,j}  = \cup_{n\ge 1} \big(E'_{ij}\oplus A_n^{i,j}\big).
\eeaa 
Then it is clear that $A^c, \cup_{n\ge 1} A_n\in \cG_{ij}$. This implies that $\cG_{ij}$ is a $\si$-algebra and thus $\cG_{ij} =  \cG_{s_1}\vee \cF_{s_2}^{s_1}$. That is, \eqref{subsetgs2'} holds, and hence $\cG'_{s_2} \subset \cG'_{s_1}\vee \cF^{s_1}_{s_2}$ for $s_m\le s_1\le s_2\le T$.

Finally, for $t\le s_1 \le s_m \le s_2\le T$, by the previous two cases we have
\beaa
\cG'_{s_2}\subset \cG'_{s_m}\vee \cF^{s_m}_{s_2}\subset (\cG'_{s_1}\vee  \cF^{s_1}_{s_m})\vee \cF^{s_m}_{s_2} = \cG'_{s_1}\vee  \cF^{s_1}_{s_2}.
\eeaa
This completes the proof of $ \dbG'\in \cA(t, \cG^{'0}_t)$.

\smallskip
{\it Step 3.} In this step we show that\footnote{We remark that \eqref{invariance-claim} holds true under the (H)-hypothesis. However, to obtain the estimate for the joint distribution in \eqref{invariance}, we require the stronger (H*)-hypothesis.}
\bea
\label{invariance-claim}
 \cW_1\Big(\cL_{\big(\cL_{\xi'|\cG^{'0}_{t}},\cL_{\vec B^{s_m}_{1:n}|\cG'_{t_n}} \big)},\cL_{\big(\cL_{\xi|\cG^0_{t}}, \ \cL_{\vec B^{s_m}_{1:n}|\cG_{t_n}}\big)}\Big)\le \frac{3}{2^m}.
\eea
Here, denoting $\cP_1^{d,k}:= \cP_1(\dbR^d)\times \cP_1(\dbR^k)$, we extend \eqref{W2} to the space $\cP_1(\cP_1^{d,k})$ as follows:
\bea
\label{W1extend}
\left.\ba{lll}
W_1((\mu_1, \nu_1), (\mu_2, \nu_2)) := W_1(\mu_1, \mu_2) +W_1(\nu_2, \nu_2),\q \mu_1, \mu_2\in \cP_1(\dbR^d), \nu_1,\nu_2\in \cP(\dbR^{k});\\
\cW_1(\bmu_1, \bmu_2) :=  \inf\Big\{ \int_{\cP_1^{d,k}\times\cP_1^{d,k}} W_1((\mu_1,\nu_1), (\mu_2,\nu_2))\Pi(d(\mu_1,\nu_1),\, d(\mu_2,\nu_2)): \\
 \qquad\qquad\Pi(d(\mu_1,\nu_1), \cP_1^{d,k}) = \bmu_1(d(\mu,\nu)),\q  \Pi(\cP_1^{d,k}, d(\mu_2,\nu_2)) = \bmu_2(d(\mu_2,\nu_2))\Big\}
\ea\right.
\eea

First, for each $i\ge 1$, fix $\mu_i\in O_i$. Then, by the definition of $O_i$, we have 
\bea
\label{invariance-est1}
\left.\ba{lll}
\displaystyle\cW_1\Big(\cL_{\big(\cL_{\xi|\cG^0_{t}},\cL_{\vec B^{s_m}_{1:n}|\cG_{t_n}} \big)},\cL_{\big(\sum_{i}\mu_i1_{E_i}, \ \cL_{\vec B^{s_m}_{1:n}|\cG_{t_n}}\big)}\Big)\le \dbE\Big[W_1\big(\cL_{\xi|\cG^0_{t}}, \sum_{i}\mu_i1_{E_i} \big)\Big]\le \frac{1}{2^{m}};\\
\displaystyle\cW_1\Big(\cL_{\big(\cL_{\xi'|\cG^{'0}_{t}},\cL_{\vec B^{s_m}_{1:n}|\cG'_{t_n}} \big)},\cL_{\big(\sum_{i}\mu_i1_{E'_i}, \ \cL_{\vec B^{s_m}_{1:n}|\cG'_{t_n}}\big)}\Big)\le \dbE\Big[W_1\big(\cL_{\xi'|\cG^{'0}_{t}}, \sum_{i}\mu_i1_{E'_i} \big)\Big]\le \frac{1}{2^{m}}.
\ea\right.
\eea
Next, for any $i, j\ge 1$, by \eqref{law-cFcG} and \eqref{oij} we have $\cL_{\{\cL_{\vec B^{s_m}_{1:n}| \cG'_{t_n}}\} | \cG'_{s_m}} =\bmu^{ij}$, $\dbP$-a.s. on $E'_{ij}$. Then, for any measurable function $\Phi: \cP_1(\dbR^d) \times \cP_1(\dbR^{nd})\to \dbR$, we have
\beaa
&&\dbE\Big[\Phi\big(\sum_{i}\mu_i1_{E'_i},\cL_{\vec B^{s_m}_{1:n}|\cG'_{t_n}}  \big)\Big]= \sum_{i, j\ge 1}\dbE\Big[1_{E'_{ij}}\Phi\big(\mu_i,\cL_{\vec B^{s_m}_{1:n}|\cG'_{t_n}}  \big)\Big] \\
&&= \sum_{i, j\ge 1}\dbE\Big[1_{E'_{ij}}\dbE\big[\Phi\big(\mu_i,\cL_{\vec B^{s_m}_{1:n}|\cG'_{t_n}}  \big)\big|\cG'_{s_m}\big]\Big] = \sum_{i, j\ge 1}\dbE\Big[1_{E'_{ij}} \int_{\cP_1(\dbR^{nd})} \Phi\big(\mu_i,\mu  \big) \bmu^{ij}(d\mu)\Big] \\
&&= \sum_{i, j\ge 1}\dbP(E'_{ij}) \int_{\cP_1(\dbR^{nd})} \Phi\big(\mu_i,\mu  \big) \bmu^{ij}(d\mu) = \sum_{i, j\ge 1}\dbP(E_{ij}) \int_{\cP_1(\dbR^{nd})} \Phi\big(\mu_i,\mu  \big) \bmu^{ij}(d\mu);\\
&&=\sum_{i, j\ge 1}\dbE\Big[1_{E_{ij}}\int_{\cP_1(\dbR^{nd})} \Phi\big(\mu_i,\mu  \big) \bmu^{ij}(d\mu)\Big] 
\eeaa
Similarly, denote $\bmu^\o := \cL^{\dbP^{\o}}_{\cL^{\dbP^{\o}}_{\vec B^{s_m}_{1:n}|\cG_{t_n}}}$, we have
\beaa
\dbE\Big[\Phi\big(\sum_{i}\mu_i1_{E_i},\cL_{\vec B^{s_m}_{1:n}|\cG_{t_n}}  \big)\Big]= \sum_{i, j\ge 1}\dbE\Big[1_{E_{ij}}(\o)\int_{\cP_1(\dbR^{nd})} \Phi\big(\mu_i,\mu  \big) \bmu^{\o}(d\mu)\Big].
\eeaa
 In particular, when $\Phi$ is Lipschitz continuous with Lipschitz constant $1$, by \eqref{oij} we have
\beaa
&&\dbE\Big[\Phi\big(\sum_{i}\mu_i1_{E_i},\cL_{\vec B^{s_m}_{1:n}|\cG_{t_n}} \big)\Big] - \dbE\Big[\Phi\big(\sum_{i}\mu_i1_{E'_i},\cL_{\vec B^{s_m}_{1:n}|\tilde\cG'_{t_n}}  \big)\Big]\\
&& = \sum_{i, j\ge 1} \dbE\Big[1_{E_{ij}}(\o)\int_{\cP_1(\dbR^{nd})} \Phi\big(\mu_i,\mu  \big) (\bmu^{\o}-\bmu^{ij})(d\mu)\Big]\\
&&\le \sum_{i, j\ge 1} \dbE\Big[1_{E_{ij}}(\o)\cW_1(\bmu^{\o},\bmu^{ij})\Big]\le \sum_{i, j\ge 1} {1\over 2^m}\dbE\Big[1_{E_{ij}}(\o)\Big] = {1\over 2^m}.
\eeaa
One may easily extend the representation formula \eqref{W1rep} to this setting. Then, by the arbitrariness of $\Phi$, this implies that 
\beaa
\cW_1\Big(\cL_{\big(\sum_{i}\mu_i1_{E'_i}, \ \cL_{\vec B^{s_m}_{1:n}|\cG'_{t_n}}\big)}, \cL_{\big(\sum_{i}\mu_i1_{E_i}, \ \cL_{\vec B^{s_m}_{1:n}|\cG_{t_n}}\big)}\Big) \le {1\over 2^m}.
\eeaa
Combining this with \eqref{invariance-est1}, we obtain \eqref{invariance-claim} immediately.

{\it Step 4.} We now prove \eqref{invariance}. Fix $\e>0$ and $t_1, \cdots, t_n$. Note that
\beaa
&&\cW_1\Big(\cL_{\big(\cL_{\xi'|\cG^{'0}_{t}}, \ \cL_{\vec B^t_{1:n}|\cG'_{t_n}} \big)}, \cL_{\big(\cL_{\xi'|\cG^{'0}_{t}}, \cL_{\vec B^{s_m}_{1:n}|\cG'_{t_n}} \big)}\Big)\le \dbE\Big[W_1\big(\cL_{\vec B^t_{1:n}|\cG'_{t_n}},\cL_{\vec B^{s_m}_{1:n}|\cG'_{t_n}} \big)\Big] \\
&&\le \dbE\Big[\sum_{i=1}^n|B^t_{t_i}-B^{s_m}_{t_i}|\Big]=\frac{2n}{\sqrt{2\pi m}}.
\eeaa
Similarly, $\cW_1\Big(\cL_{\big(\cL_{\xi|\cG^0_{t}}, \ \cL_{\vec B^t_{1:n}|\cG_{t_n}} \big)}, \cL_{\big(\cL_{\xi|\cG^0_{t}}, \cL_{\vec B^{s_m}_{1:n}|\cG_{t_n}} \big)}\Big)\le \frac{2n}{\sqrt{2\pi m}}$. 
Then, for $m$ large enough, by \eqref{invariance-claim} we have
\beaa
\cW_1\Big(\cL_{\big(\cL_{\xi'|\cG^{'0}_{t}}, \cL_{\vec B^t_{1:n}|\cG'_{t_n}}\big)}, \cL_{\big(\cL_{\xi|\cG^0_{t}}, \cL_{\vec B^t_{1:n}|\cG_{t_n}}\big)}\Big) \le \frac{4n}{\sqrt{2\pi m}}+\frac{3}{2^{m} }\le \e.
\eeaa

Finally, for $i=1,\cdots,n$, denote $\vec B^t_{1:i} := (B^t_{t_1},\cdots, B^t_{t_i})$. By Lemma \ref{lem-H} we have
 \beaa
\cL_{(\xi, \vec B^t_{1:i})|\cG_{t_i}}=\cL_{\xi|\cG^0_{t}}\otimes\cL_{\vec B^t_{1:i}|\cG_{t_i}}, \q \cL_{(\xi', \vec B^t_{1:i})|\cG'_{t_i}}=\cL_{\xi'|\cG^{'0}_{t}}\otimes\cL_{\vec B^t_{1:i}|\cG'_{t_i}},
\eeaa
 where $\otimes$ denotes independent compositions. By \eqref{W2} and \eqref{W1extend}, one can easily see that
 \beaa
 W_1(\mu_1\otimes \nu_1, \mu_2\otimes \nu_2) = W_1(\mu_1, \mu_2) + W_1(\mu_2, \nu_2) = W_1((\mu_1, \nu_1), (\mu_2, \nu_2)).
 \eeaa 
 Now let $\Phi: \cP_1(\dbR^{(i+1)d})\to\dbR$ be a  Lipschitz continuous function with Lipschitz constant. Denote 
 \beaa
\hat\Phi(\mu,\nu):=\Phi(\mu\otimes\nu),\q (\mu, \nu)\in \cP_1^{d, id}.
\eeaa
Then $\hat\Phi: \cP_1^{d, id} \to \dbR^d$ is also Lipchitz continuous with Lipschitz constant $1$. Moreover, note that $\cL_{\vec B^t_{1:i}|\cG_{t_i}} = \cL_{\vec B^t_{1:i}|\cG_{t_n}}$, $\cL_{\vec B^t_{1:i}|\cG'_{t_i}} = \cL_{\vec B^t_{1:i}|\cG'_{t_n}} $, thanks to the (H)-hypothesis. 
 Then, by \eqref{invariance-claim},
\beaa
&&\!\!\! \!\!\! \!\!\!  \dbE[\Phi(\cL_{\xi|\cG^0_{t}}\otimes\cL_{\vec B^t_{1:i}|\cG_{t_i}})-\Phi(\cL_{\xi'|\cG^{'0}_{t}}\otimes\cL_{\vec B^t_{1:i}|\cG'_{t_i}})]=\dbE[\hat\Phi(\cL_{\xi|\cG^0_{t}},\cL_{\vec B^t_{1:i}|\cG_{t_i}})-\hat\Phi(\cL_{\xi'|\cG^{'0}_{t}},\cL_{\vec B^t_{1:i}|\cG'_{t_i}} )]\\
&&\!\!\! \!\!\! \!\!\! \le \cW_1\Big(\cL_{\big(\cL_{\xi|\cG^0_{t}}, \cL_{\vec B^t_{1:i}|\cG_{t_n}}\big)}, \cL_{\big(\cL_{\xi'|\cG^{'0}_{t}}, \cL_{\vec B^t_{1:i}|\cG'_{t_n}}\big)}\Big) \le \cW_1\Big(\cL_{\big(\cL_{\xi|\cG^0_{t}}, \cL_{\vec B^t_{1:n}|\cG_{t_n}}\big)}, \cL_{\big(\cL_{\xi'|\cG^{'0}_{t}}, \cL_{\vec B^t_{1:n}|\cG'_{t_n}}\big)}\Big) \le \e.
\eeaa
Since $\Psi$ is arbitrary, by \eqref{W1rep} we prove \eqref{invariance}.
\qed

\bigskip
 \noindent{\bf Proof of Theorem \ref{thm-Ito-general}.} The proof follows from similar arguments as in Theorem  \ref{thm-Ito}, so we shall only sketch it and focus on the estimates that differ from Theorem  \ref{thm-Ito}. Again it suffices to verify $V(T, \bmu_T)- V(0, \bmu_0)$, where $\D t:= {T\over n}$, $t_i:= i\D t$, $i=0,\cdots, n$, and 
\beaa
&&\mu_t:= \cL_{\cX^2_{t}|\mathcal G_{t}}, \q \bmu_t:= \cL_{(\cX^1_t,\mathcal L_{\mu_t})},\q \cX_i^{\th}:=\th\cX^1_{t_{i+1}}+(1-\th)\cX^1_{t_i},\\ &&\mu^\th_i:= \th\mu_{t_{i+1}}+(1-\th)\mu_{t_i},\q \bmu^\th_i := \th \bmu_{t_{i+1}} + (1-\th) \bmu_{t_i}.
\eeaa
Then
\beaa
V(T, \bmu_T)- V(0, \bmu_0)= \sum_{i=0}^{n-1} \Big[V(t_{i+1}, \bmu_{t_{i+1}})- V(t_{i}, \bmu_{t_{i}})\Big] = I^n_1 + I^n_2,
\eeaa
where, 
\beaa
I^n_1 &:=&  \sum_{i=0}^{n-1} \Big[V(t_{i+1}, \bmu_{t_{i+1}})- V(t_{i}, \bmu_{t_{i+1}})\Big] =\sum_{i=0}^{n-1} \int_0^1 \partial_t V(t_i + \th \D t, \bmu_{t_{i+1}}) d\th \D t;\\
I^n_2 &:=&  \sum_{i=0}^{n-1} \Big[V(t_{i}, \bmu_{t_{i+1}})- V(t_{i}, \bmu_{t_i})\Big]\\ 
&=&  \sum_{i=0}^{n-1}\int_0^1  \dbE\Big[{\d\over \d \bmu}V\big(t_{i},\bmu^\th_{i}, \cX^1_{t_{i+1}}, \mu_{t_{i+1}}\big)- {\d\over \d \bmu}V\big(t_{i}, \bmu^\th_{i}, \cX^1_{t_i}, \mu_{t_{i}}\big)\Big] d\th\\
&=&I^n_{2.1}+I^n_{2.2}+I^n_{2.3};\\
I^n_{2,1}&:=&\sum_{i=0}^{n-1}\int_0^1  \dbE\Big[{\d\over \d \bmu}V\big(t_{i},\bmu^\th_{i}, \cX^1_{t_{i+1}}, \mu_{t_{i}}\big)- {\d\over \d \bmu}V\big(t_{i}, \bmu^\th_{i}, \cX^1_{t_i}, \mu_{t_{i}}\big)\Big] d\th;\\
I^n_{2,2}&:=&\sum_{i=0}^{n-1}\int_0^1  \dbE\Big[{\d\over \d \bmu}V\big(t_{i},\bmu^\th_{i}, \cX^1_{t_{i}}, \mu_{t_{i+1}}\big)- {\d\over \d \bmu}V\big(t_{i}, \bmu^\th_{i}, \cX^1_{t_{i}}, \mu_{t_{i}}\big)\Big] d\th;\\
I^n_{2,3}&:=&\sum_{i=0}^{n-1}\int_0^1  \dbE\Big[{\d\over \d \bmu}V\big(t_{i},\bmu^\th_{i}, \cX^1_{t_{i+1}}, \mu_{t_{i+1}}\big)- {\d\over \d \bmu}V\big(t_{i}, \bmu^\th_{i}, \cX^1_{t_{i+1}}, \mu_{t_{i}}\big)\\
&&\qquad\qquad\quad-{\d\over \d \bmu}V\big(t_{i},\bmu^\th_{i}, \cX^1_{t_{i}}, \mu_{t_{i+1}}\big) + {\d\over \d \bmu}V\big(t_{i}, \bmu^\th_{i}, \cX^1_{t_{i}}, \mu_{t_{i}}\big)\Big] d\th.
\eeaa
Following the arguments in Theorem  \ref{thm-Ito} for $I^n_1$ and $I^n_{2,2}$, and following the arguments for the standard It\^{o} formula for $I^n_{2,1}$, we obtain
\beaa
I^n_1 &=& \int_0^T \partial_t V(t_i + \th \D t, \bmu_t) dt + o(1);\\
I^n_{2,1}&=& \int_0^T  \dbE\Big[\partial_x {\d\over \d \bmu}V\big(t,\bmu_t, \cX^1_t, \mu_t\big) \cdot \a^1_t + {1\over 2} \partial_{xx} {\d\over \d \bmu}V\big(t,\bmu_t, \cX^1_t, \mu_t\big) : \b^1_t(\b^1_t)^{\top}\Big] dt + o(1);\\
I^n_{2,2}&=&\int_0^T  \dbE\Big[\partial_{\tilde x}{\d\over \d\mu}{\d\over \d \bmu}V\big(t,\bmu_t,  \cX^1_t, \mu_t, \tilde\cX^2_t\big) \cdot \tilde\a^2_t  + {1\over 2}\partial_{\tilde x\tilde x}{\d\over \d\mu}{\d\over \d \bmu}V\big(t,\bmu_t,\cX^1_t,  \mu_{t}, \tilde\cX^2_t\big) : \tilde\b^2_t(\tilde\b^2_t)^\top\\
 &&\qquad+{1\over 2} \partial_{\tilde x\bar x}  {\d^2\over \d\mu^2}{\d\over \d \bmu}V\big(t,\bmu_{t}, \cX_t^1, \mu_{t}, \tilde\cX^2_{t}, \bar\cX^2_{t}\big) : \tilde\b^2_t \si^\dbG_t (\si^\dbG_t)^\top (\bar\b^2_t)^\top\Big] dt + o(1).
\eeaa

It remains to estimate $I^n_{2,3}$. By \eqref{linearu} we have
\beaa
I^n_{2,3}&=& \sum_{i=0}^{n-1}\int_0^1\int_0^1 \dbE\Big[  \int_{\dbR^d} \Big({\d\over \d\mu}{\d\over \d \bmu}V\big(t_{i},\bmu^\th_{i}, \cX^1_{t_i+1}, \mu^{\th'}_i, \tilde x\big) \\
&&\quad - {\d\over \d\mu}{\d\over \d \bmu}V\big(t_{i},\bmu^\th_{i}, \cX^1_{t_i}, \mu^{\th'}_i, \tilde x\big)\Big) (\mu_{t_{i+1}} -\mu_{t_i})(d\tilde x)\Big]d\th'd\th.
\eeaa
Denote $\D \cX^1_{t_{i+1}} :=  \cX^1_{t_i+1} -  \cX^1_{t_i}$.  Similarly define $\D \tilde\cX^2_{t_{i+1}}$ and $\D \bar\cX^2_{t_{i+1}}$. Assume again without loss of generality that $\a^i, \b^i$ are continuous. Introduce
\beaa
I^n_{2,3,1}&:=&\sum_{i=0}^{n-1}\int_0^1 \dbE\Big[ \int_{\dbR^d}  \Big({\d\over \d\mu}{\d\over \d \bmu}V\big(t_{i},\bmu^\th_{i}, \cX^1_{t_i+1}, \mu_{t_i}, \tilde x\big) \\
&&\quad - {\d\over \d\mu}{\d\over \d \bmu}V\big(t_{i},\bmu^\th_{i}, \cX^1_{t_i}, \mu_{t_i}, \tilde x\big)\Big) (\mu_{t_{i+1}} -\mu_{t_i})(d\tilde x)\Big]d\th\\
&=&\sum_{i=0}^{n-1}\int_0^1 \dbE\Big[{\d\over \d\mu}{\d\over \d \bmu}V\big(t_{i},\bmu^\th_{i},  \cX^1_{t_{i+1}}, \mu_{t_i},\tilde\cX^2_{t_{i+1}}\big)-{\d\over \d\mu}{\d\over \d \bmu}V\big(t_{i},\bmu^\th_{i}, \cX^1_{t_{i+1}}, \mu_{t_i},\tilde\cX^2_{t_i}\big)\Big] d\th\\
&&\quad- {\d\over \d\mu}{\d\over \d \bmu}V\big(t_{i},\bmu^\th_{i},  \cX^1_{t_{i}},\mu_{t_i},\tilde\cX^2_{t_{i+1}}\big)+ {\d\over \d\mu}{\d\over \d \bmu}V\big(t_{i},\bmu^\th_{i}, \cX^1_{t_{i}}, \mu_{t_i},\tilde\cX^2_{t_i}\big)\Big] d\th\\
&=&\sum_{i=0}^{n-1} \dbE\Big[\partial_{x\tilde x}{\d\over \d\mu}{\d\over \d \bmu}V\big(t_{i},\bmu_{t_i},  \cX^1_{t_i}, \mu_{t_i},\tilde\cX^2_{t_i}\big) : \D \cX^1_{t_{i+1}} \big(\D \tilde \cX^2_{t_{i+1}}\big)^\top \Big] + o(1)\\
&=&\sum_{i=0}^{n-1} \dbE\Big[\partial_{x\tilde x}{\d\over \d\mu}{\d\over \d \bmu}V\big(t_{i},\bmu_{t_i},  \cX^1_{t_i}, \mu_{t_i},\tilde\cX^2_{t_i}\big) : \big(\b^1_{t_i} B^{t_i}_{t_{i+1}}\big) \big(\tilde \b^2_{t_i} \tilde B^{t_i}_{t_{i+1}}\big)^\top \Big] + o(1).
\eeaa
By \eqref{Ito-claim} and \eqref{Ito-claim2}, and denoting $\D B^\dbG_{t_{i+1}} := B^\dbG_{t_{i+1}}- B^\dbG_{t_i} = \int_{t_i}^{t_{i+1}} \si^\dbG_t dB_s$, we have
\beaa
I^n_{2,3,1} &=&\sum_{i=0}^{n-1} \dbE\Big[\partial_{x\tilde x}{\d\over \d\mu}{\d\over \d \bmu}V\big(t_{i},\bmu_{t_i},  \cX^1_{t_i}, \mu_{t_i},\tilde\cX^2_{t_i}\big) : \big(\b^1_{t_i} \D B^\dbG_{t_{i+1}}\big) \big(\tilde \b^2_{t_i} \D B^\dbG_{t_{i+1}}\big)^\top \Big] + o(1)\\
&=&\int_0^T \dbE\Big[\partial_{x\tilde x}{\d\over \d\mu}{\d\over \d \bmu}V\big(t,\bmu_t,  \cX^1_t, \mu_t,\tilde\cX^2_t\big) : \b^1_t \si^\dbG_t (\si^\dbG_t)^\top (\tilde \b^2_t )^\top \Big] dt+ o(1) .
\eeaa

Finally, denote
\beaa
I_{2,3,2}^n &:=& I_{2,3}^n-I_{2,3,1}^n\\
&=& \sum_{i=0}^{n-1}\int_0^1\int_0^1\int_0^{\th'}\dbE\Big[\int_{\dbR^d\times \dbR^d} \Big({\d^2\over \d\mu^2}{\d\over \d \bmu}V\big(t_{i},\bmu^\th_{i}, \cX^1_{t_{i+1}}, ~ \mu^{\tilde\th}_i, \tilde x, \bar x\big)\\
&&\q -{\d^2\over \d\mu^2}{\d\over \d \bmu}V\big(t_{i},\bmu^\th_{i}, \cX^1_{t_i}, ~ \mu^{\tilde\th}_i,  \tilde x, \bar x\big)\Big) (\mu_{t_{i+1}} -\mu_{t_i})(d\tilde x) (\mu_{t_{i+1}} -\mu_{t_i})(d\bar x)\Big] d\tilde \th d\th'd\th\\
&=& \sum_{i=0}^{n-1}\dbE\Big[u_i(\cX^1_{t_{i+1}}, \tilde \cX^2_{t_{i+1}}, \bar \cX^2_{t_{i+1}}) - u_i(\cX^1_{t_i}, \tilde \cX^2_{t_{i+1}}, \bar \cX^2_{t_{i+1}}) \\
&&\q\q\q - u_i(\cX^1_{t_{i+1}}, \tilde \cX^2_{t_i}, \bar \cX^2_{t_{i+1}})  - u_i(\cX^1_{t_{i+1}}, \tilde \cX^2_{t_{i+1}}, \bar \cX^2_{t_i}) + u_i(\cX^1_{t_i}, \tilde \cX^2_{t_i}, \bar \cX^2_{t_{i+1}}) \\
&&\q\q\q+ u_i(\cX^1_{t_{i+1}}, \tilde \cX^2_{t_i}, \bar \cX^2_{t_i}) + u_i(\cX^1_{t_i}, \tilde \cX^2_{t_{i+1}}, \bar \cX^2_{t_i}) - u_i(\cX^1_{t_i}, \tilde \cX^2_{t_i}, \bar \cX^2_{t_i})\Big]+o(1),
\eeaa
where  $u_i(x, \tilde x, \bar x):={\d^2\over \d\mu^2}{\d\over \d \bmu}V\big(t_{i},\bmu_{t_{i}}, x, \mu_{t_i}, \tilde x, \bar x\big)$. Since $ \partial^{(2)}_{(x,\tilde x, \bar x)} u_i$ exist and are continuous and bounded, and $\dbE[|\D \cX^j_{t_{i+1}}|^2] \le C\D t$, $j=1,2$, applying the standard Taylor expansion we obtain
\beaa
I_{2,3,2}^n =o(1).
\eeaa
Put everything together and send $n\to \infty$, we prove \eqref{Ito-general}.
\qed

\end{appendices}

%\end{linenumbers}

\bibliographystyle{abbrv}
\bibliography{info_control}

\end{document}